\documentclass[a4paper,11pt]{article}
\usepackage{amsmath,amsthm}
\usepackage{authblk}
\usepackage{booktabs}
\usepackage{cancel}
\usepackage[font=small]{caption}
\usepackage{colortbl}
\usepackage{comment}
\usepackage[hmargin={25mm,25mm},vmargin={30mm,35mm}]{geometry}
\usepackage{ifthen}
\usepackage[colorlinks,allcolors={blue},linkcolor={black}]{hyperref}
\usepackage{multirow}
\usepackage{nicefrac}
\usepackage{paralist}
\usepackage{pgfplots,pgfplotstable}
\usepackage[font=small]{subcaption}
\usepackage{tikz-cd}
\usepackage[small,compact]{titlesec}
\usepackage{xcolor}
\usepackage{enumitem}

\usepackage{makeidx}
\usepackage{graphicx}

\usepackage{newtxtext}
\usepackage{newtxmath}

\usepackage{accents}

\graphicspath{{figures/pdf/}}

\newcommand{\logLogSlopeTriangle}[5]
{
    \pgfplotsextra
    {
        \pgfkeysgetvalue{/pgfplots/xmin}{\xmin}
        \pgfkeysgetvalue{/pgfplots/xmax}{\xmax}
        \pgfkeysgetvalue{/pgfplots/ymin}{\ymin}
        \pgfkeysgetvalue{/pgfplots/ymax}{\ymax}

        \pgfmathsetmacro{\xArel}{#1}
        \pgfmathsetmacro{\yArel}{#3}
        \pgfmathsetmacro{\xBrel}{#1-#2}
        \pgfmathsetmacro{\yBrel}{\yArel}
        \pgfmathsetmacro{\xCrel}{\xArel}

        \pgfmathsetmacro{\lnxB}{\xmin*(1-(#1-#2))+\xmax*(#1-#2)} 
        \pgfmathsetmacro{\lnxA}{\xmin*(1-#1)+\xmax*#1} 
        \pgfmathsetmacro{\lnyA}{\ymin*(1-#3)+\ymax*#3} 
        \pgfmathsetmacro{\lnyC}{\lnyA+#4*(\lnxA-\lnxB)}
        \pgfmathsetmacro{\yCrel}{\lnyC-\ymin)/(\ymax-\ymin)}

        \coordinate (A) at (rel axis cs:\xArel,\yArel);
        \coordinate (B) at (rel axis cs:\xBrel,\yBrel);
        \coordinate (C) at (rel axis cs:\xCrel,\yCrel);

        \draw[#5]   (A)-- node[pos=0.5,anchor=north] {\scriptsize{1}}
                    (B)-- 
                    (C)-- node[pos=0.,anchor=west] {\scriptsize{#4}} 
                    cycle;
    }
}

\newcommand{\email}[1]{\href{mailto:#1}{#1}}

\newtheorem{theorem}{Theorem}
\newtheorem{proposition}[theorem]{Proposition}
\newtheorem{lemma}[theorem]{Lemma}
\newtheorem{corollary}[theorem]{Corollary}
\theoremstyle{remark}
\newtheorem{remark}[theorem]{Remark}
\theoremstyle{definition}
\newtheorem{assumption}[theorem]{Assumption}

\newcommand{\st}{\,:\,}
\newcommand{\Real}{\mathbb{R}}
\newcommand{\Natural}{\mathbb{N}}

\newcommand{\Mh}[1][h]{\mathcal{M}_{#1}}
\newcommand{\Th}[1][h]{\mathcal{T}_{#1}}
\newcommand{\Fh}[1][h]{\mathcal{F}_{#1}}
\newcommand{\Fhb}{\Fh^{{\rm b}}}
\newcommand{\Eh}[1][h]{\mathcal{E}_{#1}}
\newcommand{\Vh}{\mathcal{V}_h}

\DeclareRobustCommand{\bvec}[1]{\boldsymbol{#1}}
\newcommand{\uvec}[1]{\underline{\bvec{#1}}}
\newcommand{\cvec}[1]{\bvec{\mathcal{#1}}}

\newcommand{\rotation}[1]{\varrho_{#1}}

\DeclareMathOperator{\GRAD}{\bf grad}
\DeclareMathOperator{\CURL}{\bf curl}
\DeclareMathOperator{\DIV}{div}
\DeclareMathOperator{\ROT}{rot}
\DeclareMathOperator{\VROT}{\bf rot}

\newcommand{\Hcurl}[1]{\bvec{\mathrm{H}}(\CURL;#1)}

\newcommand{\Hdiv}[1]{\bvec{\mathrm{H}}(\DIV;#1)}
\newcommand{\Leb}{\mathrm{L}}
\newcommand{\Hil}{\mathrm{H}}

\newcommand{\Xgrad}[1][T]{\underline{\mathrm{X}}_{\GRAD,#1}^k}
\newcommand{\Xcurl}[1][T]{\underline{\bvec{\mathrm{X}}}_{\CURL,#1}^k}
\newcommand{\Xcurlz}[1][T]{\underline{\mathcal{\mathrm{E}}}_{#1}}
\newcommand{\Xdiv}[1][T]{\underline{\bvec{\mathrm{X}}}_{\DIV,#1}^k}
\newcommand{\Xdivz}[1][T]{\underline{\mathcal{\mathrm{F}}}_{#1}}

\newcommand{\ext}{\uvec{\mathfrak{E}}_h}

\newcommand{\uIgrad}[1][T]{\underline{I}_{\GRAD,#1}^k}

\newcommand{\uIdiv}[1][T]{\uvec{I}_{\DIV,#1}^{k}}

\newcommand{\lproj}[2]{\pi_{\mathcal{P},#2}^{#1}}
\newcommand{\vlproj}[2]{\boldsymbol{\pi}_{\cvec{P},#2}^{#1}}

\newcommand{\Rproj}[2]{\bvec{\pi}_{\cvec{R},#2}^{#1}}
\newcommand{\ROproj}[2]{\bvec{\pi}_{\cvec{R},#2}^{\perp,#1}}

\newcommand{\Gproj}[2]{\bvec{\pi}_{\cvec{G},#2}^{#1}}
\newcommand{\GOproj}[2]{\bvec{\pi}_{\cvec{G},#2}^{\perp,#1}}

\newcommand{\uGh}[1][k]{\uvec{G}_h^{#1}}

\newcommand{\GE}{G_E^k}
\newcommand{\GF}{\accentset{\bullet}{\bvec{G}}_F^k}
\newcommand{\GT}{\accentset{\bullet}{\bvec{G}}_T^k}

\newcommand{\CT}{\accentset{\bullet}{\bvec{C}}_T^k}
\newcommand{\Ch}{\bvec{C}_h^k}
\newcommand{\CF}{C_F^k}
\newcommand{\DT}[1][T]{D_{#1}^k}
\newcommand{\Dh}{D_h^k}
\newcommand{\uCT}[1][T]{\uvec{C}_{#1}^k}
\newcommand{\uCh}{\uvec{C}_h^k}
\DeclareMathOperator{\CURLz}{\mathsf{CURL}}

\newcommand{\trE}{\gamma_E^{k+1}}
\newcommand{\trF}{\gamma_F^{k+1}}

\newcommand{\trFtilde}{\tilde{\gamma}_F^{k+1}}
\newcommand{\trFt}{\bvec{\gamma}_{{\rm t},F}^k}

\newcommand{\PcurlT}{\bvec{P}_{\CURL,T}^k}

\newcommand{\hPcurlT}{\hat{\bvec{P}}_{\CURL,T}^k}
\newcommand{\PdivT}{\bvec{P}_{\DIV,T}^k}
\newcommand{\Pdivh}{\bvec{P}_{\DIV,h}^k}

\newcommand{\FT}{\mathcal{F}_T}
\newcommand{\ET}{\mathcal{E}_T}
\newcommand{\EF}{\mathcal{E}_F}
\newcommand{\FE}{\mathcal{F}_E}
\newcommand{\VT}{\mathcal{V}_T}
\newcommand{\VF}{\mathcal{V}_F}
\newcommand{\VE}{\mathcal{V}_E}

\newcommand{\normal}{\bvec{n}}
\newcommand{\tangent}{\bvec{t}}

\newcommand{\Poly}[2][]{\mathcal{P}_{#1}^{#2}}

\newcommand{\Roly}[1]{\boldsymbol{\mathcal{R}}^{#1}}
\newcommand{\Goly}[1]{\boldsymbol{\mathcal{G}}^{#1}}

\newcommand{\norm}[2][]{\|#2\|_{#1}}

\newcommand{\vvvert}{\vert\kern-0.25ex\vert\kern-0.25ex\vert}
\newcommand{\tnorm}[2][]{\vvvert #2\vvvert_{#1}}

\DeclareMathOperator{\Ker}{Ker}
\DeclareMathOperator{\Image}{Im}

\DeclareMathOperator{\card}{card}

\newcommand{\ball}[1]{\mathcal{B}_{#1}}

\newcommand{\Id}{{\rm Id}}

\DeclareRobustCommand{\MAT}[1]{\boldsymbol{\mathsf{#1}}}
\DeclareRobustCommand{\VEC}[1]{\mathsf{#1}}
\DeclareRobustCommand{\UVEC}[1]{\underline{\mathsf{#1}}}
\newcommand{\trans}{\intercal}

\newcommand{\NPT}[1]{N_{\mathcal{P},T}^{#1}}
\newcommand{\NGT}[1]{N_{\cvec{G},T}^{#1}}
\newcommand{\NGOT}[1]{N_{\cvec{G},T}^{#1,\perp}}
\newcommand{\NRT}[1]{N_{\cvec{R},T}^{#1}}

\newcommand{\NPF}[1]{N_{\mathcal{P},F}^{#1}}

\newcommand{\NDIV}{N_{\DIV,T}^k}

\newcommand{\BPoly}[2][T]{\mathfrak{P}_{#1}^{#2}}
\newcommand{\BPolyd}[2][T]{\bvec{\mathfrak{P}}_{#1}^{#2}}
\newcommand{\BGoly}[2][T]{\bvec{\mathfrak{G}}_{#1}^{#2}}
\newcommand{\BGOoly}[2][T]{\bvec{\mathfrak{G}}_{#1}^{#2,\perp}}
\newcommand{\BRoly}[2][T]{\bvec{\mathfrak{R}}_{#1}^{#2}}
\newcommand{\BROoly}[2][T]{\bvec{\mathfrak{R}}_{#1}^{#2,\perp}}
\newcommand{\BXdivT}{\bvec{\mathfrak{B}}_{\DIV,T}^k}

\begin{document}

\title{An arbitrary-order method for magnetostatics on polyhedral meshes based on a discrete de Rham sequence}
\author[1]{Daniele A. Di Pietro}
\author[2]{J\'er\^ome Droniou}
\affil[1]{IMAG, Univ Montpellier, CNRS, Montpellier, France, \email{daniele.di-pietro@umontpellier.fr}}
\affil[2]{School of Mathematics, Monash University, Melbourne, Australia, \email{jerome.droniou@monash.edu}}

\maketitle

\begin{abstract}
  In this work we develop a discretisation method for the mixed formulation of the magnetostatic problem supporting arbitrary orders and polyhedral meshes.
  The method is based on a global discrete de Rham (DDR) sequence, obtained by patching the local spaces constructed in \cite{Di-Pietro.Droniou.ea:20} by enforcing the single-valuedness of the components attached to the boundary of each element.
  The first main contribution of this paper is a proof of exactness relations for this global DDR sequence, obtained leveraging the exactness of the corresponding local sequence and a topological assembly of the mesh valid for domains that do not enclose any void.
  The second main contribution is the formulation and well-posedness analysis of the method, which includes the proof of uniform Poincar\'e inequalities for the discrete divergence and curl operators.
  The convergence rate in the natural energy norm is numerically evaluated on standard and polyhedral meshes.
  When the DDR sequence of degree $k\ge 0$ is used, the error converges as $h^{k+1}$, with $h$ denoting the mesh size.
\end{abstract}



\section{Introduction}

In this work we develop a discretisation method for the mixed formulation of the magnetostatic problem supporting arbitrary orders and polyhedral meshes.
The stability of the method hinges on a global version of the discrete de Rham (DDR) sequence of \cite{Di-Pietro.Droniou.ea:20}.

Let $\Omega\subset\Real^3$ be an open connected polyhedral domain that does not enclose any void (that is, its second Betti number is zero), with boundary $\partial\Omega$ and unit outward normal $\normal$.
Denote by $\Hcurl{\Omega}$ the space of vector-valued functions over $\Omega$ that are square-integrable along with their curl and by $\Hdiv{\Omega}$ the space of vector-valued functions over $\Omega$ that are square-integrable along with their divergence.
Let $\bvec{J}\in\CURL\Hcurl{\Omega}$ and $\bvec{g}\in \Leb^2(\partial\Omega)^3$ denote, respectively, the free current density and boundary datum.
We consider the following problem (see, e.g., \cite[Section 4.5.3]{Arnold:18}):
Find $(\bvec{H},\bvec{A})\in\Hcurl{\Omega}\times\Hdiv{\Omega}$ such that
\begin{equation}\label{eq:weak}
  \begin{alignedat}{2}
    a(\bvec{H},\bvec{\zeta}) - b(\bvec{\zeta},\bvec{A}) &= -\int_{\partial\Omega}\bvec{g}\cdot\bvec{\zeta}
    &\qquad& \forall\bvec{\zeta}\in\Hcurl{\Omega},
    \\ 
    b(\bvec{H},\bvec{v}) + c(\bvec{A},\bvec{v}) &= \int_\Omega\bvec{J}\cdot\bvec{v}
    &\qquad& \forall\bvec{v}\in\Hdiv{\Omega},
  \end{alignedat}
\end{equation}
with bilinear forms $a:\Hcurl{\Omega}\times\Hcurl{\Omega}\to\Real$, $b:\Hcurl{\Omega}\times\Hdiv{\Omega}\to\Real$, and $c:\Hdiv{\Omega}\times\Hdiv{\Omega}\to\Real$ such that, for all $\bvec{\upsilon},\bvec{\zeta}\in\Hcurl{\Omega}$ and all $\bvec{w},\bvec{v}\in\Hdiv{\Omega}$,
\[
a(\bvec{\upsilon},\bvec{\zeta})\coloneq\int_\Omega\mu\bvec{\upsilon}\cdot\bvec{\zeta},\qquad
b(\bvec{\zeta},\bvec{v})\coloneq\int_\Omega\bvec{v}\cdot\CURL\bvec{\zeta},\qquad
c(\bvec{w},\bvec{v})\coloneq\int_\Omega\DIV\bvec{w}\DIV\bvec{v}.
\]
The solution of problem \eqref{eq:weak} satisfies almost everywhere 
\begin{subequations}\label{eq:strong}
  \begin{alignat}{2}\label{eq:strong:magnetic.field}
    \mu\bvec{H} - \CURL\bvec{A} &= \bvec{0} &\qquad& \text{in $\Omega$},
    \\ \label{eq:strong:ampere}
    \CURL\bvec{H} &= \bvec{J} &\qquad& \text{in $\Omega$},
    \\ \label{eq:strong:coulomb.gauge}
    \DIV\bvec{A} &= 0 &\qquad& \text{in $\Omega$},
    \\ \label{eq:strong:bc}
    \bvec{A}\times\normal &= \bvec{g} &\qquad& \text{on $\partial\Omega$}.
  \end{alignat}
\end{subequations}
In the context of magnetostatics, the interpretation of the above relations is as follows:
equation \eqref{eq:strong:magnetic.field} expresses the magnetic field $\bvec{H}$ in terms of the vector potential $\bvec{A}$, with $\mu:\Omega\to\Real$ denoting the permeability such that $\overline{\mu}\ge \mu\ge\underline{\mu}>0$ almost everywhere in $\Omega$, where $\overline{\mu}$ and $\underline{\mu}$ are constant numbers;
equation \eqref{eq:strong:ampere} is Amp\`ere's law relating the magnetic field to the current density;
\eqref{eq:strong:coulomb.gauge} is the so-called Coulomb's gauge which, together with the boundary condition \eqref{eq:strong:bc}, ensures the uniqueness of the vector potential (notice that, since the second Betti number of $\Omega$ is zero, the space of $2$-harmonic forms is trivial).

The well-posedness of problem \eqref{eq:weak} hinges on the fact that, under the above assumptions on the domain, the image of the curl operator coincides with the kernel of the divergence operator, and the latter is surjective in $\Leb^2(\Omega)$; cf., e.g., the discussion in \cite[Section 2]{Di-Pietro.Droniou.ea:20}.
These relations correspond to the exactness of the rightmost portion of the de Rham sequence
\begin{equation}\label{eq:continuous.sequence}
  \begin{tikzcd}
    \Real\arrow{r}{i_\Omega}
    & \Hil^1(\Omega)\arrow{r}{\GRAD}
    & \Hcurl{\Omega}\arrow{r}{\CURL}
    & \Hdiv{\Omega}\arrow{r}{\DIV}
    & \Leb^2(\Omega)\arrow{r}{0} & \{0\},
  \end{tikzcd}
\end{equation}
where $i_\Omega$ is the operator that maps a real value to a constant function over $\Omega$ and $\Hil^1(\Omega)$ the space of scalar-valued functions over $\Omega$ that are square-integrable along with their gradient.
The design of stable numerical approximations of problem \eqref{eq:weak} requires to mimic these exactness properties at the discrete level.

  Other formulations of the magnetostatic problem are possible, most notably the $\bvec{H}$- and $\bvec{A}$-methods presented in \cite{Kanayama.Motoyama.ea:90}. The $\bvec{H}$-method consists in eliminating $\bvec{A}$ from \eqref{eq:strong} by expressing the fact that $\mu\bvec{H}\in \Hcurl{\Omega}$ (see \eqref{eq:strong:magnetic.field}) to write $\DIV (\mu\bvec{H})=0$, and by retaining \eqref{eq:strong:ampere}; the divergence constraint on $\mu\bvec{H}$ is enforced in the weak formulation by the introduction of a Lagrange multiplier $p$ (scalar potential, which appears through its gradient).
  The $\bvec{A}$-method uses \eqref{eq:strong:magnetic.field} to eliminate $\bvec{H}$ from \eqref{eq:strong:ampere}, leading to the second order $\CURL (\mu^{-1}\CURL \bvec{A})=\bvec{J}$ equation on $\bvec{A}$. Here too, the constraint $\DIV\bvec{A}=0$ is imposed in the weak formulation via a scalar potential that appears through its gradient.
  In either of these methods, the well-posedness of the weak formulation hinges on the exactness of the leftmost portion of the de Rham sequence \eqref{eq:continuous.sequence}, specifically the relation $\Ker \CURL\subset \Image \GRAD$, which holds only when the first Betti number of $\Omega$ is zero.
  The most convenient formulation to use as a starting point for the numerical approximation (magnetic field-vector potential as in \eqref{eq:weak}, or $\bvec{H}$- or $\bvec{A}$-methods as in \cite{Kanayama.Motoyama.ea:90}) thus depends on the topology of the domain: for solid toroidal or cylindrical domains (that have non-zero first Betti number but zero second Betti number), e.g., problem \eqref{eq:weak} is naturally well-posed, whereas the formulations of \cite{Kanayama.Motoyama.ea:90} require an additional condition to account for the existence of non-trivial 1-harmonic forms.
  From the standpoint of numerical approximation, this necessitates the (possibly expensive) computation of cohomology generators (see, e.g., \cite{Dotko.Specogna:13,Rodriguez.Bertolazzi.ea:13}).
  Clearly, for domains that enclose voids (non-zero second Betti number) but do not have tunnels (zero first Betti number), the situation is reversed, and an additional condition enforcing the $\Leb^2$-orthogonality of the vector potential $\bvec{A}$ to 2-harmonic forms is required to ensure the well-posedness of \eqref{eq:weak}; see, e.g., \cite[Section 4.5.4]{Arnold:18}. We note in passing that, although this paper focuses on the formulation \eqref{eq:weak}, the framework we develop provides an entire discrete De Rham sequence that could equally be used on the $\bvec{H}$- and $\bvec{A}$-methods.

In the context of Finite Element (FE) approximations, the exactness property discussed above is achieved by a nontrivial choice of finite-dimensional subspaces of $\Hcurl{\Omega}$ and $\Hdiv{\Omega}$; cf.\ \cite{Arnold:18} for a comprehensive introduction to this topic.
Finite Elements can, however, display severe practical limitations:
the construction of the finite-dimensional spaces hinges upon conforming meshes with elements of simple geometric shape;
the number of degrees of freedom on hexahedral elements can become very large when increasing the polynomial degree (see, e.g., \cite[Table 2]{Di-Pietro.Droniou.ea:20});
the definition of unisolvent degrees of freedom can be tricky for high-order versions (cf., e.g., \cite{Bonazzoli.Rapetti:17} and references therein).
A recent generalisation of FE methods is provided by the Isogeometric Analysis (IGA), which is designed to facilitate exchanges with Computer Assisted Design software.
In this framework, spline spaces and projection operators that verify a de Rham diagram have been developed in \cite{Buffa.Rivas.ea:11}, placing on solid grounds the IGA method for electromagnetism originally proposed in \cite{Buffa.Sangalli.ea:10}; see also \cite{Buffa.Sangalli.ea:14} for further developments.
Low-order frameworks that involve exact discrete sequences of spaces and operators on general polyhedral meshes have been developed over the last years, among which we cite Mimetic Finite Differences (see \cite{Beirao-da-Veiga.Lipnikov.ea:14} and references therein), the Discrete Geometric Approach (see, e.g., \cite{Codecasa.Specogna.ea:09}), or Compatible Discrete Operators (see \cite{Bonelle.Ern:15,Bonelle.Di-Pietro.ea:15} and also \cite{Bonelle:14}).
The above polyhedral frameworks are tightly related to the lowest-order version of the DDR sequence corresponding to $k=0$, and have already found their way in commercial codes. Compatible Discrete Operators, along with the more recent Hybrid High-Order (HHO) technology \cite{Di-Pietro.Droniou:20}, are available in widely used simulators such as Code\_Aster (\url{https://www.code-aster.org}) or Code\_Saturne (\url{https://www.code-saturne.org}).
More recently, a de Rham sequence of arbitrary-order virtual spaces has been proposed in \cite{Beirao-da-Veiga.Brezzi.ea:16}; see also the related works \cite{Beirao-da-Veiga.Brezzi.ea:18*1,Beirao-da-Veiga.Brezzi.ea:18*2} concerning the Virtual Element approximation of the magnetostatic problem based on the formulation of \cite{Kanayama.Motoyama.ea:90}.
Virtual spaces are spanned by functions whose expression is not available at each point, therefore projections on polynomial spaces are used in practice.
For this reason, the exactness of the virtual sequence cannot be directly exploited to prove the stability of numerical schemes.
Fully nonconforming approximations of the magnetostatic problem have also been explored, where stability is ensured by additional penalty terms.
We mention here, in particular, the Discontinuous Galerkin method of \cite{Perugia.Schotzau.ea:02}, the Hybridizable Discontinuous Galerkin methods of \cite{Nguyen.Peraire.ea:11,Chen.Cui.ea:19,Chen.Qiu.ea:17}, and, on polyhedral meshes, the HHO methods of \cite{Chave.Di-Pietro.ea:20,Chave.Di-Pietro.ea:20*1}.

The approach proposed in this work relies on a global DDR sequence involving spaces of (polynomial) discrete unknowns and discrete counterparts of vector operators acting thereon.
This sequence is obtained patching the local spaces constructed in \cite{Di-Pietro.Droniou.ea:20} by enforcing the single-valuedness of the components of these spaces on the element boundaries.
Its exactness, which constitutes the first key contribution of this work, is proved in Theorem \ref{thm:exactness} below, leveraging the exactness of the local sequence and a topological assembly of the mesh valid for domains that do not enclose any void.
The second contribution of this work is the development and well-posedness analysis of a discretisation method for the mixed formulation \eqref{eq:weak} of the magnetostatic problem (the first, to our knowledge, supporting polyhedral meshes).
The discrete problem is formulated in terms of the spaces and operators appearing in the global DDR sequence, along with discrete counterparts of $\Leb^2$-products.
For this reason, the stability and well-posedness (Theorem \ref{thm:stability} and Corollary \ref{cor:well-posedness}) of the discrete problem are direct consequences of the exactness of the DDR sequence, together with uniform Poincar\'e inequalities for the discrete divergence and curl operators.
Besides supporting polyhedral meshes and arbitrary orders, the proposed method has fewer unknowns than (non-serendipity) Finite Elements on hexahedra (cf.\ Remark \ref{rem:comparison} below) and allows for great freedom in the practical implementation of the polynomial spaces that lie at its core.
The convergence rate of the method is numerically evaluated on a set of standard and polyhedral refined mesh families.
When the DDR sequence of degree $k\ge 0$ is used, the error in the natural energy norm associated with the problem behaves as $h^{k+1}$, with $h$ denoting the mesh size.

The rest of this paper is organised as follows.
In Section \ref{sec:setting} we introduce the setting, recalling the appropriate notion of polyhedral mesh, the definitions of vector operators on faces, and those of the local polynomial spaces.
In Section \ref{sec:sequence} we define the global DDR sequence and prove the required exactness relations.
Section \ref{sec:discretisation} contains the statement of the discrete problem along with a theoretical well-posedness result and a numerical assessment of its convergence rate.
In Section \ref{sec:implementation} we discuss the practical implementation.
Finally, Appendix \ref{sec:inf-sup} contains the proof of the stability result, which is based on uniform discrete Poincar\'e inequalities.
The paper is structured so as to offer two levels of reading. In particular, the implementation aspects in Section \ref{sec:implementation} and the proofs in Appendix \ref{sec:inf-sup} are rather technical, and the reader mainly interested in the formulation of the proposed numerical scheme can skip them at first reading.


\section{Setting}\label{sec:setting}

In this section we define the discrete setting: the mesh, the vector operators on faces, and various polynomial spaces that appear in the construction.

\subsection{Mesh}

Given a set $X\subset\Real^3$, denote by $h_X$ its diameter, that is, the supremum of the distance between two points of $X$.
We consider meshes $\Mh\coloneq\Th\cup\Fh\cup\Eh\cup\Vh$, where:
  (i) $\Th$ is a finite collection of polyhedral elements that partition $\Omega$ and such that $h=\max_{T\in\Th}h_T>0$;
  (ii) $\Fh$ is a finite collection of planar faces;
  (iii) $\Eh$ is the set collecting the polygonal edges (line segments) of the faces;
  (iv) $\Vh$ is the set collecting the edge endpoints.
  It is assumed, in what follows, that $(\Th,\Fh)$ matches the conditions in \cite[Definition 1.4]{Di-Pietro.Droniou:20}.
  Notice that this notion of mesh is related to that of cellular (or CW) complex from algebraic topology \cite[Chapter 7]{Spanier:94}.
We additionally assume that the polytopes in $\Th\cup\Fh$ are simply connected and have connected boundaries that are Lipschitz-continuous (that is, each polytope can locally be represented as the epigraph, in the corresponding dimension 2 or 3, of a Lipschitz-continuous function).
We denote by $\Fhb\subset\Fh$ the set of boundary faces, that is, faces contained in $\partial\Omega$.

For all $F\in\Fh$, an orientation is set by prescribing a unit normal vector $\normal_F$. Similarly, each edge $E\in\Eh$ is endowed with a unit tangent vector $\tangent_E$ defining its orientation.
Given a mesh element $T\subset\Real^3$, we denote by $\FT\subset\Fh$ the set of faces contained in the boundary $\partial T$ of $T$.
For all $F\in\FT$, we denote by $\omega_{TF}\in\{-1,1\}$ the orientation of $F$ relative to $T$, that is, $\omega_{TF}=1$ if $\normal_F$ points out of $T$, $-1$ otherwise.
With this choice, $\omega_{TF}\normal_F$ is the unit vector normal to $F$ that points out of $T$.
Similarly, for a face $F$, we denote by $\EF$ the set of edges that lie on the boundary $\partial F$ of $F$.
The boundary of $F$ is oriented counter-clockwise with respect to $\normal_F$, and we denote by $\omega_{FE}\in\{-1,1\}$ the orientation of $\tangent_E$ opposite to $\partial F$: $\omega_{FE}=1$ if $\tangent_E$ points on $E$ in the opposite orientation to $\partial F$, $\omega_{FE}=-1$ otherwise.
For any polygon $F$ and any edge $E\in\EF$, we also denote by $\normal_{FE}$ the unit normal vector to $E$ lying in the plane of $F$ such that $(\tangent_E,\normal_{FE})$ forms a system of right-handed coordinates in the plane of $F$, which means that the system of coordinates $(\tangent_E,\normal_{FE},\normal_F)$ is right-handed.
It can be checked that $\omega_{FE}\normal_{FE}$ is the normal to $E$, in the plane where $F$ lies, pointing out of $F$.
In what follows, we will also need the set of edges of an element $T\subset\Real^3$, which we denote by $\ET$, and the set of vertices of a mesh entity $X\in\Mh$, which we denote by $\mathcal{V}_X$.
Finally, for any vertex $V\in\Vh$, we denote by $\bvec{x}_V$ the corresponding vector of coordinates.

\subsection{Vector operators on faces}

The DDR construction requires vector operators on faces.
Specifically, for any $F\in\Fh$, we respectively denote by $\GRAD_F$ and $\DIV_F$ the tangent gradient and divergence operators acting on smooth functions and, for any $r:F\to\Real$ smooth enough, we define the two-dimensional vector curl operator such that
\[
  \VROT_F r\coloneq \rotation{-\nicefrac\pi2}\left(\GRAD_F r\right),
\]
where $\rotation{-\nicefrac\pi2}$ is the rotation, in the oriented tangent space to $F$, of angle $-\frac\pi2$.
We will also need the two-dimensional scalar curl operator such that, for any $\bvec{v}:F\to\Real^2$ smooth enough,
\[
  \ROT_F\bvec{v}\coloneq\DIV_F \left(\rotation{-\nicefrac\pi2} \bvec{v}\right).
\]

\subsection{Polynomial spaces}\label{sec:setting:polynomial.spaces}

For given integers $\ell\ge -1$ and $n\ge 0$, we denote by $\mathbb{P}_n^\ell$ the space of $n$-variate polynomials of total degree $\le\ell$, with the convention that $\mathbb{P}_0^\ell\coloneq\Real$ for any $\ell$ and $\mathbb{P}_n^{-1}\coloneq\{0\}$ for any $n$.
Let $X$ be a polyhedron, a polygon (immersed in $\Real^3$), or a segment (again immersed in $\Real^3$). We denote by $\Poly{\ell}(X)$ the space spanned by the restriction to $X$ of functions in $\mathbb{P}_3^\ell$.
Denoting by $0\le n\le 3$ the dimension of $X$, $\Poly{\ell}(X)$ is isomorphic to $\mathbb{P}_n^\ell$ (the proof, quite simple, follows the ideas of \cite[Proposition 1.23]{Di-Pietro.Droniou:20}).
With a little abuse of notation, we denote both spaces with $\Poly{\ell}(X)$, and the exact meaning of this symbol should be inferred from the context.
We will also need the space $\Poly{0,\ell}(X)\coloneq\left\{q\in\Poly{\ell}(X)\st\int_X q=0\right\}$ spanned by functions in $\Poly{\ell}(X)$ with zero average over $X$.

For any $X\in\Th\cup\Fh\cup\Eh$, $\lproj{\ell}{X}:\Leb^1(X)\to\Poly{\ell}(X)$ is the $\Leb^2$-orthogonal projector such that, for any $q\in \Leb^1(X)$,
\[
  \int_X(\lproj{\ell}{X}q-q)r=0\qquad\forall r\in\Poly{\ell}(X).
\]
As a projector, $\lproj{\ell}{X}$ is polynomially consistent, that is, it maps any $r\in\Poly{\ell}(X)$ onto itself.
Optimal approximation properties for this projector have been proved in \cite{Di-Pietro.Droniou:17}; see also \cite{Di-Pietro.Droniou:17*1} for more general results on projectors on local polynomial spaces.
Letting $n$ be the dimension of $X$, we also denote by $\vlproj{\ell}{X}:\Leb^1(X)^n\to\Poly{\ell}(X)^n$ the vector version obtained applying the projector component-wise.

For any $F\in\Fh$ and any integer $\ell\ge -1$, we define the following relevant subspaces of $\Poly{\ell}(F)^2$:
\[
\begin{alignedat}{3}
  \Goly{\ell}(F)&\coloneq\GRAD_F\Poly{\ell+1}(F),
  &\qquad&
  \Goly{\ell}(F)^\perp&\coloneq\text{$\Leb^2$-orthogonal complement of $\Goly{\ell}(F)$ in $\Poly{\ell}(F)^2$},
  \\
  \Roly{\ell}(F)&\coloneq\VROT_F\Poly{\ell+1}(F),
  &\qquad&
  \Roly{\ell}(F)^\perp&\coloneq\text{$\Leb^2$-orthogonal complement of $\Roly{\ell}(F)$ in $\Poly{\ell}(F)^2$}.
\end{alignedat}
\]
The corresponding $\Leb^2$-orthogonal projectors are, with obvious notation, $\Gproj{\ell}{F}$, $\GOproj{\ell}{F}$, $\Rproj{\ell}{F}$, and $\ROproj{\ell}{F}$.
Similarly, for any  $T\in\Th$ and any integer $\ell\ge -1$ we introduce the following subspaces of $\Poly{\ell}(T)^3$:
\[
\begin{alignedat}{3}
  \Goly{\ell}(T)&\coloneq\GRAD\Poly{\ell+1}(T),
  &\qquad&
  \Goly{\ell}(T)^\perp&\coloneq\text{$\Leb^2$-orthogonal complement of $\Goly{\ell}(T)$ in $\Poly{\ell}(T)^3$},
  \\
  \Roly{\ell}(T)&\coloneq\CURL\Poly{\ell+1}(T),
  &\qquad&
  \Roly{\ell}(T)^\perp&\coloneq\text{$\Leb^2$-orthogonal complement of $\Roly{\ell}(T)$ in $\Poly{\ell}(T)^3$}.
\end{alignedat}
\]
The corresponding $\Leb^2$-orthogonal projectors are $\Gproj{\ell}{T}$, $\GOproj{\ell}{T}$, $\Rproj{\ell}{T}$, and $\ROproj{\ell}{T}$.

At the global level, we will need the space of broken polynomial functions of total degree $\le\ell$ defined by
\begin{equation}\label{eq:Poly.ell.Th}
  \Poly{\ell}(\Th)\coloneq\left\{
  q\in \Leb^2(\Omega)\st q_{|T}\in\Poly{\ell}(T)\quad\forall T\in\Th
  \right\}.
\end{equation}


\section{Global DDR sequence}\label{sec:sequence}

In this section, we define a DDR sequence mimicking the de Rham sequence \eqref{eq:continuous.sequence}.
Each space in the DDR sequence consists of vectors of polynomial functions attached to appropriate geometric entities of the mesh in order to imitate, through their single-valuedness, the continuity properties of the corresponding space in the continuous sequence.
The discrete vector operators in the DDR sequence are defined taking $\Leb^2$-orthogonal projections of \emph{full} operators, each mimicking an appropriate version of the Stokes formula.
The adjective full refers to the fact that these operators map on full polynomial spaces (and, correspondingly, enjoy optimal approximation properties).
Full operators that only appear in the discrete sequence through projections are identified by a dot.
The correspondence between continuous and discrete spaces and operators are summarised in Tables \ref{tab:correspondence:spaces} and \ref{tab:correspondence:operators}, respectively.
For the sake of brevity, we recall here only the main facts and refer to \cite{Di-Pietro.Droniou.ea:20} for a more detailed presentation of the local DDR sequence.

\begin{table} \centering
  \begin{tabular}{ccc}
    \toprule
    Continuous space & Discrete space & Definition \\
    \midrule
    $\Real$ & $\Real$ & --- \\
    $\Hil^1(\Omega)$ & $\Xgrad[h]$ & Eq. \eqref{eq:Xgrad.h} \\
    $\Hcurl{\Omega}$ & $\Xcurl[h]$ & Eq. \eqref{eq:Xcurl.h}  \\
    $\Hdiv{\Omega}$ & $\Xdiv[h]$ & Eq. \eqref{eq:Xdiv.h} \\
    $\Leb^2(\Omega)$ & $\Poly{k}(\Th)$ & Eq. \eqref{eq:Poly.ell.Th} \\
    \bottomrule
  \end{tabular}
  \caption{Correspondence between continuous and discrete spaces in the de Rham \eqref{eq:continuous.sequence} and DDR \eqref{eq:ddr.sequence} sequences.\label{tab:correspondence:spaces}}
\end{table}

\begin{table} \centering
  \begin{tabular}{ccc}
    \toprule
    Continuous operator & Discrete operator & Definition \\
    \midrule
    $\GRAD$ & $\uGh$ & Eq. \eqref{eq:uGh} \\
    $\CURL$ & $\uCh$ & Eq. \eqref{eq:uCh} \\
    $\DIV$ & $\Dh$ & Eq. \eqref{eq:Dh} \\
    \bottomrule
  \end{tabular}
  \caption{Correspondence between continuous and discrete vector operators in the de Rham \eqref{eq:continuous.sequence} and DDR \eqref{eq:ddr.sequence} sequences.\label{tab:correspondence:operators}}
\end{table}

As pointed out in \cite[Section 2]{Di-Pietro.Droniou.ea:20}, the exactness relations
\begin{equation}\label{eq:exactness}
  \Image\CURL=\Ker\DIV,\qquad\Image\DIV = \Leb^2(\Omega)
\end{equation}
play a key role in the well-posedness of problem \eqref{eq:weak}.
The main result of this section, proved in Theorem \ref{thm:exactness}, is a discrete counterpart of \eqref{eq:exactness} for the global DDR sequence.
In what follows, we fix an integer $k\ge 0$ corresponding to the polynomial degree of the sequence.

\subsection{Discrete $\Hil^1(\Omega)$ space and full gradient operators}\label{sec:disc.Hgrad}

The discrete counterpart of the space $\Hil^1(\Omega)$ is
\begin{equation}\label{eq:Xgrad.h}
\begin{aligned}
  \Xgrad[h]\coloneq\Big\{
  \underline{r}_h=\big({}&
  (r_T)_{T\in\Th}, (r_F)_{F\in\Fh}, (r_E)_{E\in\Eh}, (r_V)_{V\in\Vh}
  \big)\st
  \\
  {}&\text{$r_T\in\Poly{k-1}(T)$ for all $T\in\Th$,
    $r_F\in\Poly{k-1}(F)$ for all $F\in\Fh$,}
  \\
  {}&\text{$r_E\in\Poly{k-1}(E)$ for all $E\in\Eh$,    
  and $r_V\in\Real$ for all $V\in\Vh$}
  \Big\}.
\end{aligned}
\end{equation}
The restrictions of $\Xgrad[h]$ and $\underline{r}_h\in\Xgrad[h]$ to a mesh element, face, or edge $X\in\Th\cup\Fh\cup\Eh$ are denoted by, respectively, $\Xgrad[X]$ and $\underline{r}_X\in\Xgrad[X]$.
The interpolator on $\Xgrad[h]$ is $\uIgrad[h]:C^0(\overline{\Omega})\to\Xgrad[h]$ such that, for all $r\in C^0(\overline{\Omega})$,
\begin{equation}\label{eq:uIgradh}
  \uIgrad[h] r
  \coloneq\big(
  (\lproj{k-1}{T} r)_{T\in\Th},
  (\lproj{k-1}{F} r)_{F\in\Fh},
  (\lproj{k-1}{E} r)_{E\in\Eh},
  (r(\bvec{x}_V))_{V\in\Vh}
  \big),
\end{equation}
that is, the discrete element $\uIgrad[h] r$ representing $r$ is obtained taking the $\Leb^2$-projections of $r$ on the polynomial spaces composing $\Xgrad[h]$.

For any edge $E\in\Eh$ and any $\underline{r}_E=(r_E,(r_V)_{V\in\VE})\in\Xgrad[E]$, denote by $\trE\underline{r}_E$ the unique polynomial in $\Poly{k+1}(E)$ such that $(\trE\underline{r}_E)(\bvec{x}_V)=r_V$ for all $V\in\VE$ and $\lproj{k-1}{E}(\trE\underline{r}_E)=r_E$.
We define the (full) edge gradient $\GE:\Xgrad[E]\to\Poly{k}(E)$ setting
\begin{equation}\label{eq:GE}
\GE\underline{r}_E \coloneq (\trE\underline{r}_E)'\qquad\forall\underline{r}_E\in\Xgrad[E],
\end{equation}
where the derivative is taken along $E$ in the direction of $\tangent_E$.

\begin{remark}[Edge unknowns]\label{rem:continuous.poly.skeleton}
  Selecting a family $((r_E)_{E\in\Eh}, (r_V)_{V\in\Vh})$ of edge and vertex degrees of freedom is equivalent to selecting a function $q:\bigcup_{E\in\Eh}\overline{E}\to\Real$ on the edge skeleton that is polynomial of degree $\le k+1$ on each edge and continuous at the vertices (hence, $q$ is globally continuous on the edge skeleton). This function is simply given by $q_{|E}=\trE\underline{r}_E$ for all $E\in\Eh$.
\end{remark}

For any face $F\in\Fh$, we define the full face gradient $\GF:\Xgrad[F]\to\Poly{k}(F)^2$ such that, for all $\underline{r}_F\in\Xgrad[F]$,
\[
\int_F\GF\underline{r}_F\cdot\bvec{v}
= -\int_F r_F\DIV_F\bvec{v}
+ \sum_{E\in\EF}\omega_{FE}\int_E\trE\underline{r}_E(\bvec{v}\cdot\normal_{FE})
\qquad\forall \bvec{v}\in\Poly{k}(F)^2,
\]
along with the corresponding scalar potential $\trF:\Xgrad[F]\to\Poly{k+1}(F)$ such that, for all $\underline{r}_F\in\Xgrad[F]$,
\[
\trF\underline{r}_F \coloneq r_F + \trFtilde\underline{r}_F - \lproj{k-1}{F}(\trFtilde\underline{r}_F).
\]
Above, $\trFtilde:\Xgrad[F]\to\Poly{k+1}(F)$ is a face potential reconstruction that is consistent for polynomials of total degree $\le k+1$, that is, denoting by $\uIgrad[F]:C^0(\overline{F})\to\Xgrad[F]$ the restriction of the interpolator \eqref{eq:uIgradh} to $F$, $\trFtilde(\uIgrad[F] r) = r$ for all $r\in\Poly{k+1}(F)$.
Finally, for all $T\in\Th$, the full element gradient $\GT:\Xgrad[T]\to\Poly{k}(T)^3$ is such that, for all $\underline{r}_T\in\Xgrad[T]$,
\[
\int_T\GT\underline{r}_T\cdot\bvec{v}
= -\int_Tr_T\DIV\bvec{v}
+ \sum_{F\in\FT}\omega_{TF}\int_F\trF\underline{r}_F(\bvec{v}\cdot\normal_F)
\qquad\forall\bvec{v}\in\Poly{k}(T)^3.
\]

\subsection{Discrete $\Hcurl{\Omega}$ space and full curl operators}\label{sec:disc.Hcurl}

The role of the space $\Hcurl{\Omega}$ is played at the discrete level by
\begin{equation}\label{eq:Xcurl.h}
\begin{aligned}
  \Xcurl[h]\coloneq\Big\{
  \uvec{\upsilon}_h=\big({}&
  (\bvec{\upsilon}_{\cvec{R},T},\bvec{\upsilon}_{\cvec{R},T}^\perp)_{T\in\Th},
  (\bvec{\upsilon}_{\cvec{R},F},\bvec{\upsilon}_{\cvec{R},F}^\perp)_{F\in\Fh},
  (\upsilon_E)_{E\in\Eh}
  \big)\st
  \\
  {}&\text{
    $\bvec{\upsilon}_{\cvec{R},T}\in\Roly{k-1}(T)$ and $\bvec{\upsilon}_{\cvec{R},T}^\perp\in\Roly{k}(T)^\perp$ for all $T\in\Th$,
  }
  \\
  {}&\text{
    $\bvec{\upsilon}_{\cvec{R},F}\in\Roly{k-1}(F)$ and $\bvec{\upsilon}_{\cvec{R},F}^\perp\in\Roly{k}(F)^\perp$ for all $F\in\Fh$,
  }
  \\
  {}&\text{and $\upsilon_E\in\Poly{k}(E)$ for all $E\in\Eh$}
  \Big\}.
\end{aligned}
\end{equation}
The restrictions of $\Xcurl[h]$ and $\uvec{\upsilon}_h\in\Xcurl[h]$ to a mesh element or face $X\in\Th\cup\Fh$ are denoted by, respectively, $\Xcurl[X]$ and $\uvec{\upsilon}_X\in\Xcurl[X]$.

For any face $F\in\Fh$, we reconstruct the discrete (full) face curl $\CF:\Xcurl[F]\to\Poly{k}(F)$ approximating $\ROT_F$ such that, for all $\uvec{\upsilon}_F\in\Xcurl[F]$,
\begin{equation}\label{eq:CF}
  \int_F\CF\uvec{\upsilon}_F q
  = \int_F\bvec{\upsilon}_{\cvec{R},F}\cdot\VROT_F q
  - \sum_{E\in\EF}\omega_{FE}\int_E\upsilon_E q\qquad
  \forall q\in\Poly{k}(F),
\end{equation}
as well as the discrete tangential face potential $\trFt:\Xcurl[F]\to\Poly{k}(F)^2$ such that, for all $\uvec{\upsilon}_F\in\Xcurl[F]$, it holds, for all $(r,\bvec{\tau})\in\Poly{0,k+1}(F)\times\Roly{k}(F)^\perp$,
\begin{equation}\label{eq:trFt}
  \int_F\trFt\uvec{\upsilon}_F\cdot\left(\VROT_F r + \bvec{\tau}\right)
  \\
  = \int_F \CF\uvec{\upsilon}_F r
  + \sum_{E\in\EF}\omega_{FE}\int_E \upsilon_E r
  + \int_F\bvec{\upsilon}_{\cvec{R},F}^\perp\cdot\bvec{\tau}.
\end{equation}
Similarly, for any element $T\in\Th$, we define the discrete full element curl operator $\CT:\Xcurl\to\Poly{k}(T)^3$ such that, for all $\uvec{\upsilon}_T\in\Xcurl$,
\begin{equation}\label{eq:CT}
  \int_T\CT\uvec{\upsilon}_T\cdot\bvec{\tau}
  = \int_T\bvec{\upsilon}_{\cvec{R},T}\cdot\CURL\bvec{\tau}
  + \sum_{F\in\FT}\omega_{TF}\int_F\trFt\uvec{\upsilon}_F\cdot(\bvec{\tau}\times\normal_F)
  \qquad\forall\bvec{\tau}\in\Poly{k}(T)^3.
\end{equation}

\subsection{Discrete $\Hdiv{\Omega}$ space and full divergence operator}

The discrete counterpart of the space $\Hdiv{\Omega}$ is
\begin{equation}\label{eq:Xdiv.h}
\begin{aligned}
  \Xdiv[h]\coloneq\Big\{
  \uvec{w}_h=\big({}&
  (\bvec{w}_{\cvec{G},T},\bvec{w}_{\cvec{G},T}^\perp)_{T\in\Th},
  (w_F)_{F\in\Fh}
  \big)\st
  \\
  {}&\text{$\bvec{w}_{\cvec{G},T}\in\Goly{k-1}(T)$ and $\bvec{w}_{\cvec{G},T}^\perp\in\Goly{k}(T)^\perp$ for all $T\in\Th$,}
  \\
  {}&\text{and $w_F\in\Poly{k}(F)$ for all $F\in\Fh$}
  \Big\}.
\end{aligned}
\end{equation}
The restrictions of $\Xdiv[h]$ and $\uvec{w}_h\in\Xdiv[h]$ to a mesh element $T\in\Th$ are denoted by, respectively, $\Xdiv$ and $\uvec{w}_T\in\Xdiv$.

For any $T\in\Th$, we define the (full) discrete divergence reconstruction $\DT:\Xdiv\to\Poly{k}(T)$ such that, for any $\uvec{w}_T\in\Xdiv$,
\begin{equation}\label{eq:DT}
  \int_T\DT\uvec{w}_T q
  = -\int_T\bvec{w}_{\cvec{G},T}\cdot\GRAD q
  + \sum_{F\in\FT}\omega_{TF}\int_F w_F q
  \qquad\forall q\in\Poly{k}(T).
\end{equation}

\subsection{Global DDR sequence and exactness}

Define the discrete counterparts $\uGh:\Xgrad[h]\to\Xcurl[h]$ of the gradient operator, $\uCh:\Xcurl[h]\to\Xdiv[h]$ of the curl operator, and $\Dh:\Xdiv[h]\to\Poly{k}(\Th)$ of the divergence operator such that, for all $(\underline{r}_h,\uvec{\upsilon}_h,\uvec{w}_h)\in\Xgrad[h]\times\Xcurl[h]\times\Xdiv[h]$,
\begin{gather}\label{eq:uGh}
  \uGh\underline{r}_h
  \coloneq\big(
  (\Rproj{k-1}{T}\GT\underline{r}_T, \ROproj{k}{T}\GT\underline{r}_T)_{T\in\Th},
  (\Rproj{k-1}{F}\GF\underline{r}_F, \ROproj{k}{F}\GF\underline{r}_F)_{F\in\Fh},
  (\GE\underline{r}_E)_{E\in\Eh}
  \big),
  \\ \label{eq:uCh}
  \uCh\uvec{\upsilon}_h
  \coloneq\big(
  (\Gproj{k-1}{T}\CT\uvec{\upsilon}_T, \GOproj{k}{T}\CT\uvec{\upsilon}_T)_{T\in\Th},
  (\CF\uvec{\upsilon}_F)_{F\in\Fh}
  \big),
  \\ \label{eq:Dh}
  (\Dh\uvec{w}_h)_{|T}\coloneq \DT\uvec{w}_T\qquad\forall T\in\Th.
\end{gather}
Recalling the definition \eqref{eq:uIgradh} of the interpolator on $\Xgrad[h]$, the global DDR sequence reads
\begin{equation}\label{eq:ddr.sequence}
  \begin{tikzcd}
    \Real\arrow{r}[above=2pt]{\uIgrad[h]}
    & \Xgrad[h]\arrow{r}[above=2pt]{\uGh}
    & \Xcurl[h]\arrow{r}[above=2pt]{\uCh}
    & \Xdiv[h]\arrow{r}[above=2pt]{\Dh}
    & \Poly{k}(\Th)\arrow{r}[above=2pt]{0}
    & \{0\}.
  \end{tikzcd}
\end{equation}
\begin{remark}[Comparison with Finite and Virtual Element Methods]
   A thorough comparison between the DDR approach and standard Finite Elements on tetrahedral and hexahedral meshes has been carried out in \cite{Di-Pietro.Droniou.ea:20}; see, in particular, Tables 1 and 2 therein.
    
    Similarities and differences exist between the degrees of freedom of Virtual Element sequences and the polynomial components of the spaces in the DDR sequence \eqref{eq:ddr.sequence}.
    Specifically, contrary \eqref{eq:ddr.sequence}, the Virtual Element sequence of \cite{Beirao-da-Veiga.Brezzi.ea:16} is composed of spaces whose degree decreases by one at each application of the exterior derivative. In \cite{Beirao-da-Veiga.Brezzi.ea:18*1}, the authors focus on a lowest order version of the method, while, in \cite{Beirao-da-Veiga.Brezzi.ea:18*2}, they develop serendipity versions of Virtual Element spaces with fewer degrees of freedom with respect to those in the original sequence.
    As pointed out in \cite[Section 5.5]{Di-Pietro.Droniou:20}, there exists a duality between fully discrete approaches such as DDR or HHO and Virtual Elements in the sense that both interpretations can co-exist for a given scheme.
    From the analysis standpoint, working in a fully discrete framework can have advantages in some cases, linked, in particular, to the fact that it does not require to estimate the approximation properties of virtual spaces; see \cite{Di-Pietro.Droniou:18} and, in particular, Section 3.2 therein, devoted to the analysis of conforming and non-conforming Virtual Element methods.

    To close this remark, we emphasize that both the DDR and Virtual Element approaches can, in principle, be used in conjunction with Finite Elements on computational meshes that feature both standard and polyhedral elements.
\end{remark}
The following theorem establishes a discrete counterpart of the exactness relations \eqref{eq:exactness}, a crucial ingredient to prove the well-posedness of the discrete problem.
\begin{theorem}[Exactness]\label{thm:exactness}
  It holds
  \begin{align} \label{eq:Im.div=Poly.Th}
    \Image\Dh ={}& \Poly{k}(\Th),\\
      \label{eq:Im.curl=Ker.div}
    \Image\uCh ={}& \Ker\Dh.
  \end{align}
\end{theorem}
\begin{remark}[No voids assumption]
  The assumption that $\Omega$ does not enclose any void is only used to prove $\Ker\Dh\subset\Image\uCh$.
  The relations $\Image\Dh=\Poly{k}(\Th)$ and $\Image\uCh\subset\Ker\Dh$ hold for any polyhedral domain.
\end{remark}

\begin{proof}
{\bf 1. Proof of \eqref{eq:Im.div=Poly.Th}.}
We only have to show the inclusion $\Image\Dh\supset\Poly{k}(\Th)$. Let $q_h\in\Poly{k}(\Th)$. A classical result gives the existence of $\bvec{v}\in \Hil^1(\Omega)^3$ such that $\DIV\bvec{v}=q_h$ (see, e.g., \cite[Lemma 8.3]{Di-Pietro.Droniou:20} in the case where $q_h$ has a zero average over $\Omega$; the case of a generic $q_h$ follows easily since constant functions can be trivially written as divergences). Let $\uvec{w}_h\in\Xdiv[h]$ be the global interpolate of $\bvec{v}$, that is, $\bvec{w}_{\cvec{G},T}=\Gproj{k-1}{T}\bvec{v}$ and $\bvec{w}_{\cvec{G},T}^\perp=\GOproj{k}{T}\bvec{v}$ for all $T\in\Th$, while $w_F=\lproj{k}{F}(\bvec{v}\cdot\normal_F)$ for all $F\in\Fh$. Then, for all $T\in\Th$, $\uvec{w}_T$ is the local interpolate of $\bvec{v}$ in the element $T$ and by \cite[Lemma 25]{Di-Pietro.Droniou:20}, we have
\[
\DT\uvec{w}_T = \lproj{k}{T}(\DIV\bvec{v}_{|T})=\lproj{k}{T}q_{h|T}=q_{h|T}.
\]
This proves that $\Dh\uvec{w}_h=q_h$ and concludes the proof that $\Image\Dh=\Poly{k}(\Th)$.
\medskip\\
\noindent{\bf 2. Proof of \eqref{eq:Im.curl=Ker.div}.}
\smallskip\\
\underline{2.a) Proof that $\Image\uCh\subset \Ker\Dh$}.
Let $\uvec{w}_h\in\Image\uCh$.
For any $T\in\Th$, denote by $\uCT$ the restriction of $\uCh$ to $T$ such that, for all $\uvec{\upsilon}_T\in\Xcurl$,
\[
\uCT\uvec{\upsilon}_T\coloneq\big(
\Gproj{k-1}{T}\CT\uvec{\upsilon}_T, \GOproj{k}{T}\CT\uvec{\upsilon}_T,
(\CF\uvec{\upsilon}_F)_{F\in\FT}
\big).
\]
Then, for all $T\in\Th$, we have $\uvec{w}_T\in \Image\uCT=\Ker\DT$, the last equality following from the exactness of element-wise operators stated in \cite[Theorem 17]{Di-Pietro.Droniou.ea:20}. This shows that $\DT\uvec{w}_T=0$ for all $T\in\Th$, and thus that $\Dh\uvec{w}_h=0$ as required.
\smallskip\\
\underline{2.b) Proof that $\Ker\Dh\subset\Image\uCh$.}
Let $\uvec{w}_h\in\Ker\Dh$. We have to find $\uvec{\upsilon}_h\in\Xcurl[h]$ such that $\uCh\uvec{\upsilon}_h=\uvec{w}_h$. By definition of $\Dh$, for all $T\in\Th$ we have $\DT\uvec{w}_T=0$ and the local exactness stated in \cite[Theorem 17]{Di-Pietro.Droniou.ea:20} gives $\uvec{\upsilon}_T\in\Xcurl[T]$ such that $\uCT\uvec{\upsilon}_T=\uvec{w}_T$.
However, nothing ensures at that stage that the vectors $(\uvec{\upsilon}_T)_{T\in\Th}$ are the restrictions to the spaces $(\Xcurl[T])_{T\in\Th}$ of some $\uvec{\upsilon}_h\in\Xcurl[h]$; this only happens if the face and edge values of these local vectors $(\uvec{\upsilon}_T)_{T\in\Th}$ match between each pair of neighbouring elements.
\begin{figure}
  \centering
  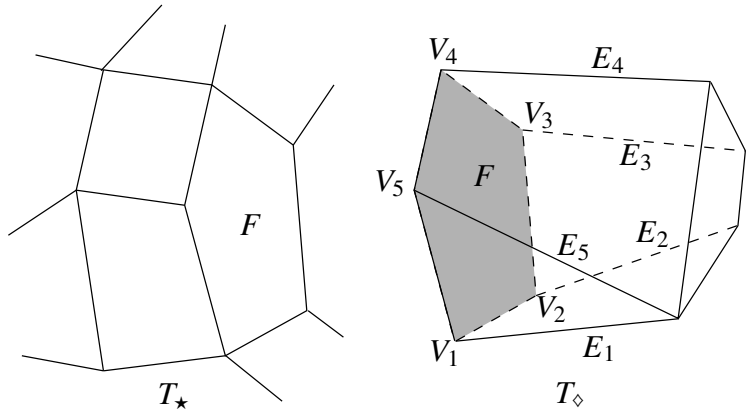
  \caption{Illustration of Operation A in the proof of Theorem \ref{thm:exactness}: adding a new element to the mesh.
  \label{fig:adding_element}
  }
\end{figure}
\begin{figure}
  \centering
  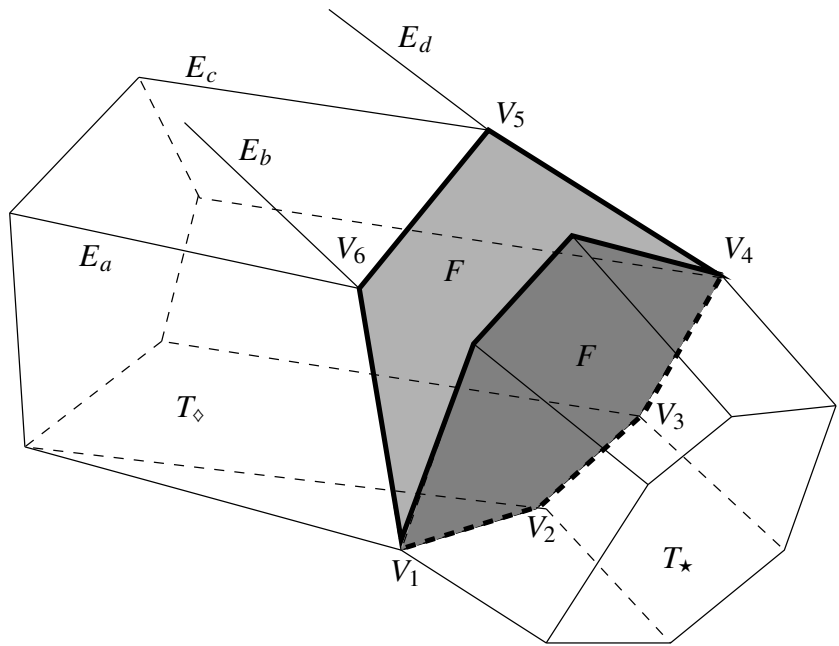
  \caption{Illustration of Operation B in the proof of Theorem \ref{thm:exactness}: gluing two faces in a mesh (note that, at the start of this operation, the faces are already glued along the connected path from $V_1$ to $V_4$).
  \label{fig:gluing_faces}
  }
\end{figure}

  To construct local vectors with matching interface values, we use an inductive approach. Since $\Omega$ does not enclose any void, the mesh $\Mh$ can be topologically assembled starting from a single element by a succession of the following two operations:
\begin{enumerate}
\item[A.] Adding a new element by gluing one of its faces to the face of another element already in the mesh; see Figure \ref{fig:adding_element}.
\item[B.] Gluing together two faces of elements already in the mesh, such that the edges along which the faces are already glued together form a connected path (which could be empty); see Figure \ref{fig:gluing_faces}.
\end{enumerate}
We note that, for a domain enclosing one or more voids, these two operations alone would not be sufficient to construct the mesh (see Remark \ref{rem:domain.hole}).

From the analytical point of view, we do not ``deform/move'' the elements, their faces, or edges to actually assemble the mesh; our elements/faces/edges are those already in the final mesh. However, we will interpret the topological aspects of this construction the following way: polynomial functions defined on faces/edges that are already glued together are single-valued (as in the global space $\Xcurl[h]$), while polynomial functions on faces/edges that are not yet ``glued'' have two values, one for each element on each side of the face we are gluing along. We nonetheless preserve the topological vocabulary of ``gluing'' faces together as it seems more intuitive and helps following the arguments in the proof.

The inductive construction of $\uvec{\upsilon}_h\in\Xcurl[h]$ such that $\Ch\uvec{\upsilon}_h=\uvec{w}_h$ starts from one element and follows the two operations described above.  The base case is already covered above: if $T\in\Th$, \cite[Theorem 17]{Di-Pietro.Droniou.ea:20} gives a pre-image $\uvec{\upsilon}_T$ through $\uCT$ of $\uvec{w}_T$. We therefore only have to consider the inductive step: starting from a submesh $\Mh[h,\star]$ of $\Mh$ (possibly with some unglued faces) and an element
\begin{equation}\label{eq:ind.star}
\uvec{\upsilon}_{h,\star}\in \Xcurl[h,\star]\quad\mbox{ such that }\quad\uCT\uvec{\upsilon}_{h,\star}=\uvec{w}_T\qquad\forall T\in\Th[h,\star],
\end{equation}
and considering a submesh $\Mh[h,\diamondsuit]$ of $\Mh$ built from $\Mh[h,\star]$ through one of the operations A or B, we need to establish the existence of
\begin{equation}\label{eq:ind.diamond}
\uvec{\upsilon}_{h,\diamondsuit}\in \Xcurl[h,\diamondsuit]\quad\mbox{ such that }\quad\uCT\uvec{\upsilon}_{h,\diamondsuit}=\uvec{w}_T\qquad\forall T\in\Th[h,\diamondsuit].
\end{equation}
Above, for $\circ\in\{\star,\diamondsuit\}$, $\Xcurl[h,\circ]$ and $\Th[h,\circ]$ denote, respectively, the discrete curl space and set of elements associated to $\Mh[h,\circ]$.
In what follows, for any $T\in\Th$, we denote by $\Mh[T]\coloneq \{T\}\cup\FT\cup\ET\cup\VT$ the submesh associated with $T$.
\smallskip

\noindent{\bf Operation A: \emph{Adding a new element by gluing one of its faces}.} Let us call $T_\diamondsuit$ the element added to $\Mh[h,\star]$, and $F$ the face along which we glue it; let $T_\star$ be the element in $\Mh[h,\star]$ to which $T_\diamondsuit$ is glued. In the step-by-step procedure below, we construct an extension $\uvec{\upsilon}_{h,\diamondsuit}$ of $\uvec{\upsilon}_{h,\star}$ to $\Mh[h,\diamondsuit]=\Mh[h,\star]\cup\Mh[T_\diamondsuit]$ such that its values on $F$ viewed from $T_\diamondsuit$ and from $T_\star$ match (which ensures that $\uvec{\upsilon}_{h,\diamondsuit}\in\Xcurl[h,\diamondsuit]$), and such that $\uCT[T_\diamondsuit]\uvec{\upsilon}_{h,\diamondsuit}=\uvec{w}_{T_\diamondsuit}$. Since $\uvec{\upsilon}_{h,\diamondsuit}=\uvec{\upsilon}_{h,\star}$ on each $T\in\Th[h,\star]$, the inductive assumption \eqref{eq:ind.star} will then prove that \eqref{eq:ind.diamond} holds.

\begin{itemize}[leftmargin=1em]
\item \emph{Pre-image in $T_\diamondsuit$}. Since $\uvec{w}_{T_\diamondsuit}\in\Ker\DT[T_\diamondsuit]$, the local exactness of \cite[Theorem 17]{Di-Pietro.Droniou.ea:20} gives $\uvec{\zeta}_{T_\diamondsuit}\in\Xcurl[T_\diamondsuit]$ that satisfies $\uCT[T_\diamondsuit]\uvec{\zeta}_{T_\diamondsuit}=\uvec{w}_{T_\diamondsuit}$. 

\item \emph{The difference of the pre-images on $F$ is a gradient}. Restricting the two relations $\uCT[T_\star]\uvec{\upsilon}_{h,\star}=\uvec{w}_{T_\star}$ and $\uCT[T_\diamondsuit]\uvec{\zeta}_{T_\diamondsuit}=\uvec{w}_{T_\diamondsuit}$ to the face $F$ in common between $T_\diamondsuit$ and $T_\star$, we have $\CF\uvec{\zeta}_F=\CF\uvec{\upsilon}_{F,\star}$ (where $\uvec{\upsilon}_{F,\star}$ is the restriction to $F$ of $\uvec{\upsilon}_{h,\star}$). Hence, $\uvec{\zeta}_F-\uvec{\upsilon}_{F,\star}\in \Ker\CF=\Image\uvec{G}_F^k$ (see \cite[Theorem 8]{Di-Pietro.Droniou.ea:20}), where $\uvec{G}_F^k:\Xgrad[F]\to \Xcurl[F]$ is the restriction to $F$ of the global gradient $\uGh$ defined by \eqref{eq:uGh}. There exists thus $\underline{q}_{F,\diamondsuit}\in\Xgrad[F]$ such that $\uvec{\zeta}_F+\uvec{G}_F^k\underline{q}_{F,\diamondsuit}=\uvec{\upsilon}_{F,\star}$. 

\item \emph{Extension of $\underline{q}_{F,\diamondsuit}$ to $\Xgrad[\diamondsuit]$}. We extend $\underline{q}_{F,\diamondsuit}$ into $\underline{q}_{T_\diamondsuit}\in\Xgrad[T_\diamondsuit]$ the following way: letting $q_{\partial F,\diamondsuit}$ be the continuous piecewise polynomial function on $\partial F$ corresponding to the boundary values of $\underline{q}_{F,\diamondsuit}$ (see Remark \ref{rem:continuous.poly.skeleton}), we prolong $q_{\partial F,\diamondsuit}$ by continuity to $\mathcal E_{T_\diamondsuit}$ as a linear function along each edge $E_i$ (see the notations in Figure \ref{fig:adding_element}), from its value at the vertex $V_i$ to 0 at the other end of $E_i$, and then set it to zero on the remaining edges. We then set all face values (except on $F$) and the element value of $\underline{q}_{T_\diamondsuit}$ to zero. 

\item \emph{Conclusion}. With the above extension, the vector $\uvec{\upsilon}_{T_\diamondsuit}:=\uvec{\zeta}_{T_\diamondsuit}+\uvec{G}_{T_\diamondsuit}^k\underline{q}_{T_\diamondsuit}\in\Xcurl[T_\diamondsuit]$ satisfies $\uCT[T_\diamondsuit]\uvec{\upsilon}_{T_\diamondsuit}=\uvec{w}_{T_\diamondsuit}$ (because $\Image \uvec{G}_{T_\diamondsuit}^k\subset \Ker\uCT[T_\diamondsuit]$ by \cite[Theorem 17]{Di-Pietro.Droniou.ea:20}) and, by construction, the values of $\uvec{\upsilon}_{h,\star}$ and $\uvec{\upsilon}_{T_\diamondsuit}$ on $F$ coincide. This completes the construction of the extension $\uvec{\upsilon}_{h,\diamondsuit}$ of $\uvec{\upsilon}_{h,\star}$ in the case of Operation A.
\end{itemize}

\smallskip

\noindent{\bf Operation B: \emph{Gluing together two mesh faces, already glued together along a connected path of their edges}.}
No new element is added to the mesh, but two faces of neighbouring elements $T_\star$ and $T_\diamondsuit$ are glued together along one of their faces $F$. In this situation, the element $\uvec{\upsilon}_{h,\star}$ given by \eqref{eq:ind.star} has, on $F$, two values $\uvec{\upsilon}_{F,\star},\uvec{\upsilon}_{F,\diamondsuit}\in\Xcurl[F]$, viewed from $T_\star$ and $T_\diamondsuit$ respectively, which coincide only along the edges that are already glued together (e.g., those between the vertices $V_1,V_2,V_3,V_4$ in Figure \ref{fig:gluing_faces}). We have to find a modification $\uvec{\upsilon}_{h,\diamondsuit}$ of $\uvec{\upsilon}_{h,\star}$ that preserves the relation $\uCT\uvec{\upsilon}_{h,\diamondsuit}=\uvec{w}_T$ for all $T\in\Th[h,\star]=\Th[h,\diamondsuit]$, and is single-valued on $F$. This is done in the following procedure by modifying the vector $\uvec{\upsilon}_{F,\diamondsuit}$ so that it coincides with $\uvec{\upsilon}_{F,\star}$; this modification will, however, have repercussions on other values of $\uvec{\upsilon}_{h,\star}$, that need to be properly tracked.

\begin{itemize}[leftmargin=1em]
\item \emph{Notations for glued/non-glued edges and vertices of $F$}. Let $\EF^{\rm g}$ be the set of edges of $F$ already glued between $T_\star$ and $T_\diamondsuit$, and $\EF^{\rm ng}\coloneq\EF\setminus\EF^{\rm g}$ the edges that are not glued; similarly, $\VF^{\rm g}$ and $\VF^{\rm ng}\coloneq\VF\setminus\VF^{\rm g}$ respectively denote the vertices of $F$ that are glued and not glued between $T_\star$ and $T_\diamondsuit$. 

In Figure \ref{fig:gluing_faces}, $\EF^{\rm g}$ is made of $[V_1,V_2]$, $[V_2,V_3]$ and $[V_3,V_4]$, $\EF^{\rm ng}$ comprises $[V_4,V_5]$, $[V_5,V_6]$ and $[V_6,V_1]$, while $\VF^{\rm g}=\{V_1,V_2,V_3,V_4\}$ and $\VF^{\rm ng}=\{V_5,V_6\}$ (remember that there is only one face $F$, its representation as two faces in this figure is just a mnemotechnic way to remember that some functions can be double-valued on the face or some of its edges). 

\item \emph{The difference of pre-images on $F$ is the gradient of a vector that vanishes along the connected path}. As in Operation A, we have $\CF(\uvec{\upsilon}_{F,\star}-\uvec{\upsilon}_{F,\diamondsuit})=w_F-w_F=0$, and there exists thus $\underline{q}_{F,\diamondsuit}\in\Xgrad[F]$ such that $\uvec{\upsilon}_{F,\star}-\uvec{\upsilon}_{F,\diamondsuit}=\uvec{G}_{F}^k\underline{q}_{F,\diamondsuit}$. By definition \eqref{eq:GE} of the boundary gradient, the derivative of $q_{\partial F,\diamondsuit}$ on each edge $E\in\EF$ is $(\uvec{\upsilon}_{F,\star}-\uvec{\upsilon}_{F,\diamondsuit})_{|E}$, and therefore vanishes on the connected path of edges in $\EF^{\rm g}$. The function $q_{\partial F,\diamondsuit}$ is thus constant along this path and, since adding to $\underline{q}_{F,\diamondsuit}$ the interpolate in $\Xgrad[F]$ of a constant function does not change the relation $\uvec{\upsilon}_{F,\star}-\uvec{\upsilon}_{F,\diamondsuit}=\uvec{G}_{F}^k\underline{q}_{F,\diamondsuit}$ (due to \cite[Proposition 4.1]{Di-Pietro.Droniou.ea:20}), we can assume that $q_{\partial F,\diamondsuit}=0$ on the edges in $\EF^{\rm g}$.

\item \emph{Extension of $\underline{q}_{F,\diamondsuit}$}.
Still following ideas developed in Operation A above, the boundary function $q_{\partial F,\diamondsuit}$ is extended to all edges starting from a vertex in $\VF^{\rm ng}$: the extension is linear from the value of  $q_{\partial F,\diamondsuit}$ at the considered vertex to zero at the other edge endpoint (in Figure \ref{fig:gluing_faces}, this extends $q_{\partial F,\diamondsuit}$ to $E_a$, $E_b$, $E_c$, $E_d$ -- notice that some of these edges do not belong to $\Eh[T_\diamondsuit]$). None of these edges belongs to $\Eh[T_\star]$, since they originate from a vertex in $\VF^{\rm ng}$ at which $T_\diamondsuit$ and $T_\star$ are not connected.

Recalling that $q_{\partial F,\diamondsuit}=0$ at the vertices in $\VF^{\rm g}$, we can further extend this function by zero on all remaining edges (that is, edges that do not contain any vertex in $\VF^{\rm ng}$). This yields a continuous piecewise polynomial function on the complete edge skeleton of $\Mh[h,\star]$, that is single-valued on all edges except those in $\EF^{\rm ng}$, and that vanishes on all the edges of $T_\star$. Setting all face values (except $F$ when viewed from $T_\diamondsuit$) and  all element values to zero, we obtain $\underline{q}_{h,\diamondsuit}$ that is single-valued on the edge skeleton (except the edges in $\EF^{\rm ng}$), single-valued on the faces (except $F$), vanishes on $T_\star$, and satisfies $\underline{q}_{T,\diamondsuit}\in\Xgrad[T]$ for all $T\in\Th[h,\star]$. By definition of $\uGh$, the vector $\uGh\underline{q}_{h,\diamondsuit}$ can be defined element-wise, and is single-valued on all faces except $F$, and on all edges except those in $\EF^{\rm ng}$.

\item \emph{Conclusion}. Setting $\uvec{\upsilon}_{h,\diamondsuit}\coloneq\uvec{\upsilon}_{h,\star}-\uGh\underline{q}_{h,\diamondsuit}$, we obtain by construction an element that is single-valued on all faces, including $F$ (since $\underline{q}_{h,\diamondsuit}=0$ on $T_\star$, the addition of $\uGh\underline{q}_{h,\diamondsuit}$ does not modify $\uvec{\upsilon}_{F,\star}$, while it adjusts $\uvec{\upsilon}_{F,\diamondsuit}$ to match $\uvec{\upsilon}_{F,\star}$). Hence, $\uvec{\upsilon}_{h,\diamondsuit}\in\Xcurl[h,\diamondsuit]$. Moreover, since $\Image\uvec{G}_T^k\subset \Ker\uCT$ for all $T\in\Th[h,\diamondsuit]$, from \eqref{eq:ind.star} we deduce that \eqref{eq:ind.diamond} holds.
\end{itemize}
\end{proof}

\begin{remark}[Domain with voids]\label{rem:domain.hole}
Consider assembling, following similar steps as in Operations A and B above, a mesh for an open set enclosing a void, e.g., a one-layer mesh of a domain contained between two concentric (polyhedral approximations of) balls. At some stage, to close the layer, we would end up in the situation described in Figure \ref{fig:invalid_gluing} (in which the ball enclosed by the domain is below the part of the mesh depicted): all the mesh except one element, $T_\diamondsuit$, has been fully assembled and this last element has to be glued to the existing mesh. On the left of this picture, it has already been glued to the elements $T_1$ and $T_2$, which can be done using Operations A and B above. Then, applying another Operation B to glue it to $T_3$, we would end up in a situation described on the right of the picture: $T_\diamondsuit$ and $T_\star$ are then only glued along the two vertical edges of $F$, which do not form a connected path. The consequence is that, when attempting to apply Operation B as in the proof of Theorem \ref{thm:exactness} to finalise the gluing of $T_\diamondsuit$ and $T_\star$ along $F$, the function $q_{\partial F,\diamondsuit}$ could only be chosen to vanish on one of the vertical edges, and not necessarily the other; its extension to the edges skeleton would therefore have non-zero values on the edges of $T_\star$, and the resulting element $\uGh\underline{q}_{h,\diamondsuit}$ would modify the value of $\uvec{\upsilon}_{F,\star}$ on $T_\star$, thus preventing $\uvec{G}_F^k\underline{q}_{F,\diamondsuit}$ from properly gluing $\uvec{\upsilon}_{F,\star}$ and $\uvec{\upsilon}_{F,\diamondsuit}$.

\begin{figure}
  \centering
  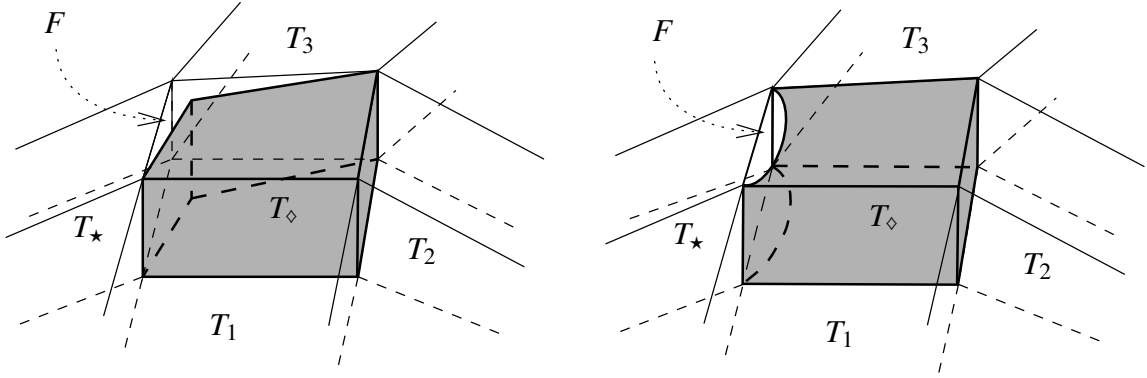
  \caption{Example of gluing appearing when assembling a mesh of an open set enclosing a void (here, the open set is a layer between two concentric balls, the inner one being below in the figures above).
  \label{fig:invalid_gluing}
  }
\end{figure}
\end{remark}


\section{DDR-based discretisation}\label{sec:discretisation}

In this section, we formulate  the DDR-based discretisation of problem \eqref{eq:weak}.
The key ingredients are reconstructions of vector potentials and discrete $\Leb^2$-products on the spaces $\Xcurl[h]$ and $\Xdiv[h]$.
The vector potential reconstructions are obtained element-wise by mimicking the Stokes formula with the role of the exterior derivative played by the appropriate vector operator reconstruction.
The discrete $\Leb^2$-products consist of two terms, one in charge of consistency based on the vector potential reconstructions, and the other in charge of stability.
The latter is obtained by penalising in a least-square sense the difference between projections of the vector potential reconstruction and the polynomial functions in the space. We notice that the approach consisting in reconstructing potentials as functions over the domain and using them to design, through standard $\Leb^2$-products, the consistent contributions to the discrete $\Leb^2$-products, enables us to seamlessly take into account the physical parameter $\mu$ in the discretisation \eqref{eq:weak}. From hereon, we assume that the permeability is constant over each element, and we denote by $\mu_T\in [\underline{\mu},\overline{\mu}]$ its value in $T\in\Th$. This choice is made for simplicity of presentation, as accounting for a permeability that varies inside each element would raise no additional challenge.
The more general case of locally varying permeability is numerically demonstrated in Section \ref{sec:numerical.examples}.

\subsection{Discrete vector potential reconstructions and $\Leb^2$-products}\label{sec:discretisation:potentials.l2-products}

Let a mesh element $T\in\Th$ be fixed.
The vector potential reconstruction on $\Xcurl[T]$ is $\PcurlT:\Xcurl\to\Poly{k}(T)^3$ such that, for all $\uvec{\upsilon}_T\in\Xcurl$,
\begin{equation}\label{eq:pcurl}
  \PcurlT\uvec{\upsilon}_T\coloneq\hPcurlT\uvec{\upsilon}_T - \Rproj{k-1}{T}(\hPcurlT\uvec{\upsilon}_T) + \bvec{\upsilon}_{\cvec{R},T},
\end{equation}
where $\hPcurlT\uvec{\upsilon}_T$ satisfies, for all $(\bvec{v},\bvec{\tau})\in\Goly{k+1}(T)^\perp\times\Roly{k}(T)^\perp$,
\begin{equation}\label{eq:PcurlT}
  \int_T\hPcurlT\uvec{\upsilon}_T\cdot\left(
  \CURL\bvec{v} + \bvec{\tau}
  \right)
  = \int_T\CT\uvec{\upsilon}_T\cdot\bvec{v}
  -\sum_{F\in\FT}\omega_{TF}\int_F\trFt\uvec{\upsilon}_T\cdot(\bvec{v}\times\normal_F)
  + \int_T\bvec{\upsilon}_{\cvec{R},T}^\perp\cdot\bvec{\tau}.
\end{equation}
Based on this potential reconstruction, we define the discrete $\Leb^2$-product such that, for all $\uvec{\upsilon}_h,\uvec{\zeta}_h\in\Xcurl[h]$,
\begin{equation}\label{eq:def.inner.curl}
\begin{gathered}
  (\uvec{\upsilon}_h,\uvec{\zeta}_h)_{\mu,\CURL,h}\coloneq\sum_{T\in\Th} (\uvec{\upsilon}_T,\uvec{\zeta}_T)_{\mu,\CURL,T}
  \\
  \text{
  with
  $(\uvec{\upsilon}_T,\uvec{\zeta}_T)_{\mu,\CURL,T}\coloneq \int_T\mu_T\PcurlT\uvec{\upsilon}_T\cdot\PcurlT\uvec{\zeta}_T + \mu_T\mathrm{s}_{\CURL,T}(\uvec{\upsilon}_T,\uvec{\zeta}_T)$
  for all $T\in\Th$.
  }
\end{gathered}
\end{equation}
In the above expression, $\mathrm{s}_{\CURL,T}:\Xcurl[T]\times\Xcurl[T]\to\Real$ is a stabilisation bilinear form that can be taken such that
\[
\begin{aligned}
  \mathrm{s}_{\CURL,T}(\uvec{\upsilon}_T,\uvec{\zeta}_T)
  &\coloneq
  \sum_{F\in\FT} h_F\int_F (\Rproj{k-1}{F}\PcurlT\uvec{\upsilon}_T - \bvec{\upsilon}_{\cvec{R},F})\cdot(\Rproj{k-1}{F}\PcurlT\uvec{\zeta}_T - \bvec{\zeta}_{\cvec{R},F})
  \\
  &\quad
  +\sum_{F\in\FT} h_F\int_F (\ROproj{k}{F}\PcurlT\uvec{\upsilon}_T - \bvec{\upsilon}_{\cvec{R},F}^\perp)\cdot(\ROproj{k}{F}\PcurlT\uvec{\zeta}_T - \bvec{\zeta}_{\cvec{R},F}^\perp)
  \\
  &\quad
  + \sum_{E\in\ET} h_E^2\int_E (\PcurlT\uvec{\upsilon}_T\cdot\tangent_E - \upsilon_E) (\PcurlT\uvec{\zeta}_T\cdot\tangent_E - \zeta_E).
\end{aligned}
\]

Given $T\in\Th$, the vector potential reconstruction on $\Xdiv[T]$ is $\PdivT:\Xdiv\to\Poly{k}(T)^3$ such that for all $\uvec{w}_T\in\Xdiv$ it holds, for all $(q,\bvec{v})\in\Poly{0,k+1}(T)\times\Goly{k}(T)^\perp$,
\begin{equation}\label{eq:PdivT}
  \int_T\PdivT\uvec{w}_T\cdot(\GRAD q+\bvec{v})
  = -\int_T\DT\uvec{w}_T q
  + \sum_{F\in\FT}\omega_{TF}\int_F w_F q
  + \int_T\bvec{w}_{\cvec{G},T}^\perp\cdot\bvec{v}.
\end{equation}
Based on this potential reconstruction, we define the $\Leb^2$-product such that, for all $\uvec{w}_h,\uvec{v}_h\in\Xdiv[h]$,
\begin{subequations}\label{eq:l2-product:Xdivh.glob}
  \begin{gather} \label{eq:l2-product:Xdivh}
    (\uvec{w}_h,\uvec{v}_h)_{\DIV,h}\coloneq\sum_{T\in\Th}(\uvec{w}_T,\uvec{v}_T)_{\DIV,T}
    \\ \label{eq:l2-product:XdivT}
    \text{
      with
      $(\uvec{w}_T,\uvec{v}_T)_{\DIV,T}\coloneq
      \int_T\PdivT\uvec{w}_T\cdot\PdivT\uvec{v}_T + \mathrm{s}_{\DIV,T}(\uvec{w}_T,\uvec{v}_T)$
      for all $T\in\Th$.
    }
  \end{gather}
\end{subequations}
Above, $\mathrm{s}_{\DIV,T}:\Xdiv[T]\times\Xdiv[T]\to\Real$ is a stabilisation bilinear form that can be taken such that
\begin{equation}\label{eq:sdivT}
  \begin{aligned}
    \mathrm{s}_{\DIV,T}(\uvec{w}_T,\uvec{v}_T)
    &\coloneq
    \int_T (\Gproj{k-1}{T}\PdivT\uvec{w}_T - \bvec{w}_{\cvec{G},T})\cdot(\Gproj{k-1}{T}\PdivT\uvec{v}_T - \bvec{v}_{\cvec{G},T})
    \\
    &\quad
    + \sum_{F\in\FT} h_F\int_F (\PdivT\uvec{w}_T\cdot\normal_F - w_F)(\PdivT\uvec{v}_T\cdot\normal_F - v_F).
  \end{aligned}
\end{equation}
The following result establishes, for all $T\in\Th$, a link between the potential reconstruction $\PdivT$ applied to the restriction to $T$ of the discrete curl operator defined by \eqref{eq:uCh}, and the full element curl operator defined by \eqref{eq:CT}.
\begin{proposition}[Link between $\PdivT$ and $\CT$]
  For all $T\in\Th$, it holds
  \begin{equation}\label{eq:PdivT.uCT=CT}
    \PdivT(\uCT\uvec{\zeta}_T)=\CT\uvec{\zeta}_T\qquad\forall\uvec{\zeta}_T\in\Xcurl.
  \end{equation}
\end{proposition}

\begin{proof}
  Let a mesh element $T\in\Th$ be fixed.
  Writing \eqref{eq:PdivT} for $\uvec{w}_T=\uCT\uvec{\zeta}_T$ and $\bvec{v}=\bvec{0}$, we infer that it holds, for all $q\in\Poly{0,k+1}(T)$,
  \[
  \int_T\PdivT(\uCT\uvec{\zeta}_T)\cdot\GRAD q
  = -\int_T\cancel{\DT(\uCT\uvec{\zeta}_T)} q
  + \sum_{F\in\FT}\omega_{TF}\int_F\CF\uvec{\zeta}_Tq
  = \int_T\CT\uvec{\zeta}_T\cdot\GRAD q,
  \]
  where we have used the fact that $\Ker\DT=\Image\uCT$ (cf.\ \cite[Theorem 17]{Di-Pietro.Droniou.ea:20}) in the cancellation, and \cite[Eq. (5.21)]{Di-Pietro.Droniou.ea:20} to conclude.
  Hence, $\Gproj{k}{T}\big(\PdivT(\uCT\uvec{\zeta}_T)\big)=\Gproj{k}{T}\big(\CT\uvec{\zeta}_T\big)$.
  On the other hand, \eqref{eq:PdivT} with $q=0$ and $\bvec{v}$ spanning $\Goly{k}(T)^\perp$ implies $\GOproj{k}{T}\big(\PdivT(\uCT\uvec{\zeta}_T)\big)=\GOproj{k}{T}\big(\CT\uvec{\zeta}_T\big)$.
  Combining these facts with the orthogonal decomposition $\Poly{k}(T)^3=\Goly{k}(T)\oplus\Goly{k}(T)^\perp$ yields \eqref{eq:PdivT.uCT=CT} and concludes the proof.
\end{proof}

\subsection{Discrete problem}

Define the discrete bilinear forms $a_h:\Xcurl[h]\times\Xcurl[h]\to\Real$, $b_h:\Xcurl[h]\times\Xdiv[h]\to\Real$, and $c_h:\Xdiv[h]\times\Xdiv[h]\to\Real$ such that, for all $\uvec{\upsilon}_h,\uvec{\zeta}_h\in\Xcurl[h]$ and all $\uvec{w}_h,\uvec{v}_h\in\Xdiv[h]$,
\[
  \mathrm{a}_h(\uvec{\upsilon}_h,\uvec{\zeta}_h)\coloneq (\uvec{\upsilon}_h,\uvec{\zeta}_h)_{\mu,\CURL,h},\quad
  \mathrm{b}_h(\uvec{\zeta}_h,\uvec{v}_h)\coloneq (\uCh\uvec{\zeta}_h,\uvec{v}_h)_{\DIV,h},\quad
  \mathrm{c}_h(\uvec{w}_h,\uvec{v}_h)\coloneq\int_\Omega\Dh\uvec{w}_h~\Dh\uvec{v}_h.
\]
The discrete problem reads:
Find $\uvec{H}_h\in\Xcurl[h]$ and $\uvec{A}_h\in\Xdiv[h]$ such that
\begin{equation}\label{eq:discrete}
  \begin{alignedat}{2}
    \mathrm{a}_h(\uvec{H}_h,\uvec{\zeta}_h) - \mathrm{b}_h(\uvec{\zeta}_h,\uvec{A}_h) &= -\sum_{F\in\Fhb}\int_F\bvec{g}\cdot\trFt\uvec{\zeta}_h
    &\qquad&\forall\uvec{\zeta}_h\in\Xcurl[h],
    \\
    \mathrm{b}_h(\uvec{H}_h,\uvec{v}_h) + \mathrm{c}_h(\uvec{A}_h,\uvec{v}_h) &= \int_\Omega\bvec{J}\cdot\Pdivh\uvec{v}_h
    &\qquad&\forall\uvec{v}_h\in\Xdiv[h].
  \end{alignedat}
\end{equation}

\begin{remark}[Characterisation of {$\mathrm{b}_h$}]
  Expanding, in the definition of $\mathrm{b}_h$, the discrete $\Leb^2$-inner product on $\Xdiv[h]$ according to \eqref{eq:l2-product:Xdivh.glob} and recalling \eqref{eq:PdivT.uCT=CT}, we obtain the following equivalent expression, which can be used for its practical implementation:
  For all $(\uvec{\zeta}_h,\uvec{w}_h)\in\Xcurl[h]\times\Xdiv[h]$,
  \[
    \mathrm{b}_h(\uvec{\zeta}_h,\uvec{w}_h)
    =\sum_{T\in\Th}\left[
    \int_T\CT\uvec{\zeta}_T\cdot\PdivT\uvec{w}_T
    + \sum_{F\in\FT} h_F\int_F (\CT\uvec{\zeta}_T\cdot\normal_F - \CF\uvec{\zeta}_F)(\PdivT\uvec{w}_T\cdot\normal_F - w_F)
    \right].
  \]
\end{remark}

\begin{remark}[Comparison with Finite Elements]\label{rem:comparison}
  Even on standard meshes, the proposed method does not coincide, in general, with the Finite Element approximation of degree $k$.
  Bearing in mind \cite[Table 2]{Di-Pietro.Droniou.ea:20}, the number of degrees of freedom is slightly higher on tetrahedra and significantly smaller on hexahedra.
\end{remark}

\subsection{Well-posedness analysis}\label{sec:discrete:well-posedness}

The discrete problem \eqref{eq:discrete} can be recast as: Find $(\uvec{H}_h,\uvec{A}_h)\in\Xcurl[h]\times\Xdiv[h]$ such that
\[
\mathcal A_h((\uvec{H}_h,\uvec{A}_h),(\uvec{\zeta}_h,\uvec{v}_h))=\mathcal L_h(\uvec{\zeta}_h,\uvec{v}_h)\qquad\forall (\uvec{\zeta}_h,\uvec{v}_h)\in 
\Xcurl[h]\times\Xdiv[h],
\]
where the bilinear form $\mathcal A_h: (\Xcurl[h]\times\Xdiv[h])^2\to \Real$ and the linear form $\mathcal L_h:\Xcurl[h]\times\Xdiv[h]\to\Real$
are such that, for all $((\uvec{\upsilon}_h,\uvec{w}_h),(\uvec{\zeta}_h,\uvec{v}_h))\in (\Xcurl[h]\times\Xdiv[h])^2$,
\begin{equation}\label{eq:def.Ah}
 \mathcal A_h((\uvec{\upsilon}_h,\uvec{w}_h),(\uvec{\zeta}_h,\uvec{v}_h))\coloneq
\mathrm{a}_h(\uvec{\upsilon}_h,\uvec{\zeta}_h) - \mathrm{b}_h(\uvec{\zeta}_h,\uvec{w}_h)
+\mathrm{b}_h(\uvec{\upsilon}_h,\uvec{v}_h) + \mathrm{c}_h(\uvec{w}_h,\uvec{v}_h)
\end{equation}
and
\[
\mathcal L_h(\uvec{\zeta}_h,\uvec{v}_h)\coloneq
-\sum_{F\in\Fhb}\int_F\bvec{g}\cdot\trFt\uvec{\zeta}_h+\int_\Omega\bvec{J}\cdot\Pdivh\uvec{v}_h.
\]

In the following, we consider a regular sequence $(\Mh)_{h>0}$ of polyhedral meshes, meaning that $(\Th,\Fh)_{h>0}$ matches the requirements of \cite[Definition 1.9]{Di-Pietro.Droniou:20}. The well-posedness analysis requires a discrete Poincar\'e inequality for $\uCh$.
This inequality can be established under the following assumption, which only requires a proper control the averages of the edge unknowns. Proposition \ref{prop:poincare.curl.simply.connected} in the appendix shows that this assumption is satisfied if $\Omega$ is simply connected. Hereafter, we denote by $\norm[\mu,\CURL,h]{{\cdot}}$ and $\norm[\DIV,h]{{\cdot}}$ the $\Leb^2$-like norms respectively associated with the inner products $(\cdot,\cdot)_{\mu,\CURL,h}$ and $(\cdot,\cdot)_{\DIV,h}$.

\begin{assumption}[Poincar\'e inequality for edge averages]\label{assum:poincare.edges}
There is an inner product on $\Xcurl[h]$ whose norm is equivalent (uniformly in $h$) to $\norm[\mu,\CURL,h]{{\cdot}}$ and such that, letting $(\Ker\uCh)^\perp$ be the orthogonal complement of $\Ker\uCh$ in $\Xcurl[h]$ for this inner product, there exists $\alpha>0$ (not depending on $h$) satisfying
\begin{equation}\label{eq:poincare.edges}
\left(\sum_{E\in\Eh}h_E^2|E|(\overline{\upsilon}_E)^2\right)^{\frac12}\le \alpha\norm[\DIV,h]{\uCh\uvec{\upsilon}_h}\qquad\forall \uvec{\upsilon}_h\in(\Ker\uCh)^\perp,
\end{equation}
where $|E|$ denotes the length of $E\in\Eh$ and $\overline{\upsilon}_E$ is the average value on $E$ of $\upsilon_E$.
\end{assumption}

The uniform inf--sup property of $\mathcal A_h$ is established in the following discrete versions of the $\Hcurl{\Omega}$ and $\Hdiv{\Omega}$ norms:
\[
\norm[\mu,\CURL,1,h]{\uvec{\zeta}_h}\coloneq\left( \norm[\mu,\CURL,h]{\uvec{\zeta}_h}^2 + \norm[\DIV,h]{\uCh\uvec{\zeta}_h}^2\right)^{\frac12}\qquad\forall\uvec{\zeta}_h\in\Xcurl[h],
\]
and
\[
\norm[\DIV,1,h]{\uvec{v}_h}\coloneq\left( \norm[\DIV,h]{\uvec{v}_h}^2 + \norm[\Omega]{\Dh\uvec{v}_h}^2\right)^{\frac12}\qquad\forall\uvec{v}_h\in\Xdiv[h].
\]
Here and in the following, $\norm[X]{{\cdot}}$ denotes the $L^2$-norm on $X=\Omega$, $X=\partial\Omega$, or $X\in\Mh$.

\begin{theorem}[Inf-sup condition for $\mathcal{A}_h$]\label{thm:stability}
  Let $\Omega\subset\Real^3$ be an open connected polyhedral domain that does not enclose any void (i.e., its second Betti number is zero), and let $(\Mh)_{h>0}$ be a regular polyhedral mesh sequence. Then, under Assumption \ref{assum:poincare.edges}, there exists $\beta>0$ depending only on $\Omega$, the mesh regularity parameter, $\mu$, and $\alpha$, but not depending on $h$, such that, for all $(\uvec{\upsilon}_h,\uvec{w}_h)\in\Xcurl[h]\times\Xdiv[h]$,
  \begin{equation}\label{eq:inf.sup.Ah}
    \sup_{(\uvec{\zeta}_h,\uvec{v}_h)\in \Xcurl[h]\times\Xdiv[h]\setminus\{0\}}\frac{\mathcal A_h((\uvec{\upsilon}_h,\uvec{w}_h),(\uvec{\zeta}_h,\uvec{v}_h))}{\norm[\mu,\CURL,1,h]{\uvec{\zeta}_h}+\norm[\DIV,1,h]{\uvec{v}_h}}\ge \beta \left(\norm[\mu,\CURL,1,h]{\uvec{\upsilon}_h}+\norm[\DIV,1,h]{\uvec{w}_h}\right).
  \end{equation}
\end{theorem}

\begin{proof}
See Section \ref{sec:proof.stability}.
\end{proof}

\begin{corollary}[Well-posedness of the discrete problem]\label{cor:well-posedness}
Under the assumptions of Theorem \ref{thm:stability}, there is a unique solution $(\uvec{H}_h,\uvec{A}_h)\in\Xcurl[h]\times\Xdiv[h]$ to \eqref{eq:discrete} and there exists $C>0$ not depending on $h$ such that
\begin{equation}\label{eq:well.posed}
  \norm[\mu,\CURL,1,h]{\uvec{H}_h}+\norm[\DIV,1,h]{\uvec{A}_h}\le C \left(
  \norm[\partial\Omega]{\bvec{g}}+\norm[\Omega]{\bvec{J}}
  \right).
\end{equation}
\end{corollary}

\begin{proof}
  Using the boundedness of $\trFt$ stated in \eqref{eq:bd.CF.trFt}, the norm equivalence \eqref{eq:equiv.Xcurl}, and the definition of $\norm[\DIV,1,h]{{\cdot}}$ (which implies $\sum_{T\in\Th}\norm[T]{\PdivT\uvec{v}_T}^2\le \norm[\DIV,1,h]{\uvec{v}_h}^2$), we obtain an upper bound of the dual norm of $\mathcal L_h$ (for the norm $\Xcurl[h]\times\Xdiv[h]\ni(\uvec{\zeta}_h,\uvec{v}_h)\mapsto\norm[\mu,\CURL,1,h]{\uvec{\zeta}_h}+\norm[\DIV,1,h]{\uvec{v}_h}\in\Real$) in terms of the right-hand side of \eqref{eq:well.posed}. The conclusion then follows from Theorem \ref{thm:stability} and \cite[Proposition A.4]{Di-Pietro.Droniou:20}.
\end{proof}

\subsection{Numerical assessment of the convergence rate}\label{sec:numerical.examples}

  The goal of this section is to numerically assess the convergence rate of the method. Our focus is thus on academic test cases for which an analytical solution is available.
  More physical tests and an assessment of the performance are postponed to a future, engineering-oriented paper.

\subsubsection{Setting}

The DDR method \eqref{eq:discrete} has been implemented within the \texttt{HArDCore3D} C++ framework (see \url{https://github.com/jdroniou/HArDCore}), using linear algebra facilities from the \texttt{Eigen3} library (see \url{http://eigen.tuxfamily.org}) and, for the resolution of the sparse linear systems, \texttt{Intel MKL PARDISO} (see \url{https://software.intel.com/en-us/mkl}) library.

In order to numerically assess the convergence properties of the DDR method, we consider the following manufactured exact solution on the unit cube $\Omega=(0,1)^3$:
\[
\bvec{H}(\bvec{x})=\frac{3\pi}{\mu(\bvec{x})}\begin{pmatrix}
\sin(\pi x_1)\cos(\pi x_2)\sin(\pi x_3) \\
0 \\
-\cos(\pi x_1)\cos(\pi x_2)\sin(\pi x_3)
\end{pmatrix},\quad
\bvec{A}(\bvec{x})=\begin{pmatrix}
\cos(\pi x_1)\sin(\pi x_2)\sin(\pi x_3) \\
-2\sin(\pi x_1)\cos(\pi x_2)\sin(\pi x_3) \\
\sin(\pi x_1)\sin(\pi x_2)\cos(\pi x_3)
\end{pmatrix},
\]
with expressions for the boundary datum $\bvec{g}$ and the current density $\bvec{J}$ inferred from the expression of $\bvec{A}$ and \eqref{eq:strong:ampere}, respectively.
We consider two cases:
\[
\text{$\mu=1$ (unit permeability)\quad and\quad $\mu(\bvec{x})=1+x_1+x_2+x_3$ (variable permeability).}
\]

Define the following interpolate of the exact solution:
\[
\begin{gathered}
  \hat{\uvec{H}}_h\coloneq\big(
  (\Rproj{k-1}{T}\bvec{H},\ROproj{k}{T}\bvec{H})_{T\in\Th},
  (\Rproj{k-1}{F}\bvec{H}_{{\rm t},F},\ROproj{k}{F}\bvec{H}_{{\rm t},F})_{F\in\Fh},
  (\lproj{k}{E}(\bvec{H}\cdot\tangent_E))_{E\in\Eh}
  \big)\in\Xcurl[h],
  \\
  \hat{\uvec{A}}_h\coloneq\big(
  (\Gproj{k-1}{T}\bvec{A},\GOproj{k}{T}\bvec{A})_{T\in\Th},
  (\lproj{k}{T}(\bvec{A}\cdot\normal_F))_{F\in\Fh}
  \big)\in\Xdiv[h],
\end{gathered}
\]
where, for all $F\in\Fh$, $\bvec{H}_{{\rm t},F}\coloneq\normal_F\times(\bvec{H}_{|F}\times\normal_F)$ denotes the tangential component of $\bvec{H}$ over $F$.
We consider the energy norm of the error defined as
\[
\norm[\rm en]{(\uvec{H}_h-\hat{\uvec{H}}_h,\uvec{A}_h-\hat{\uvec{A}}_h)}
\coloneq\left[
    \mathrm{a}_h(\uvec{H}_h-\hat{\uvec{H}}_h,\uvec{H}_h-\hat{\uvec{H}}_h)
    + \mathrm{c}_h(\uvec{A}_h-\hat{\uvec{A}}_h,\uvec{A}_h-\hat{\uvec{A}}_h)
    \right]^{\frac12}.
\]
Numerical approximations of the solution are computed on Cartesian, tetrahedral, and Voronoi mesh sequences (see Figure \ref{fig:meshes}) and polynomial degrees $k$ ranging from 0 to 3.
\begin{figure}\centering
  \begin{minipage}{0.275\textwidth}
    \includegraphics[width=0.90\textwidth]{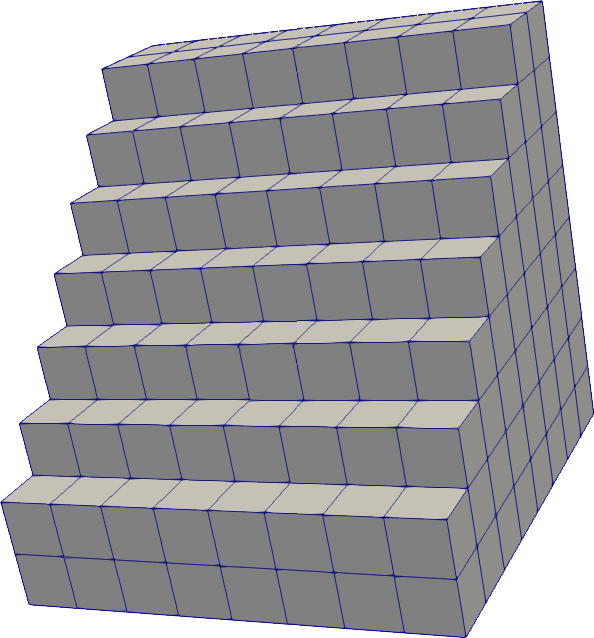}
    \subcaption{``Cubic-Cells'' mesh}
  \end{minipage}
  \hspace{0.25cm}
  \begin{minipage}{0.275\textwidth}
    \includegraphics[width=0.90\textwidth]{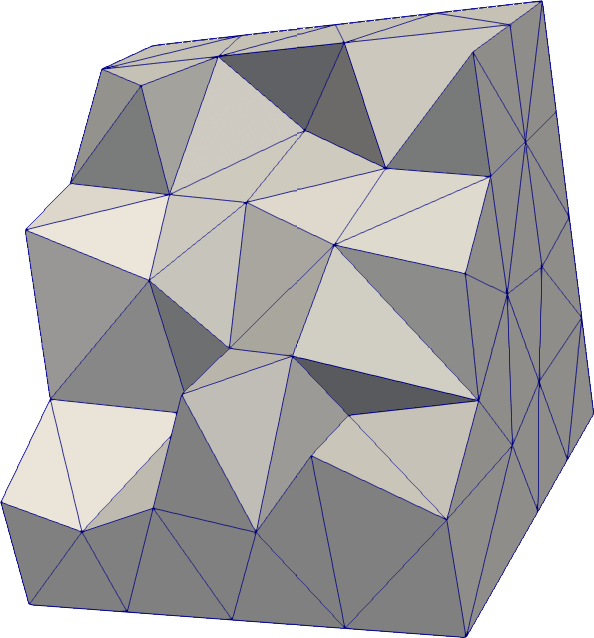}
    \subcaption{``Tetgen-Cube-0'' mesh}
  \end{minipage}
  \vspace{0.25cm}\\
  \begin{minipage}{0.275\textwidth}
    \includegraphics[width=0.90\textwidth]{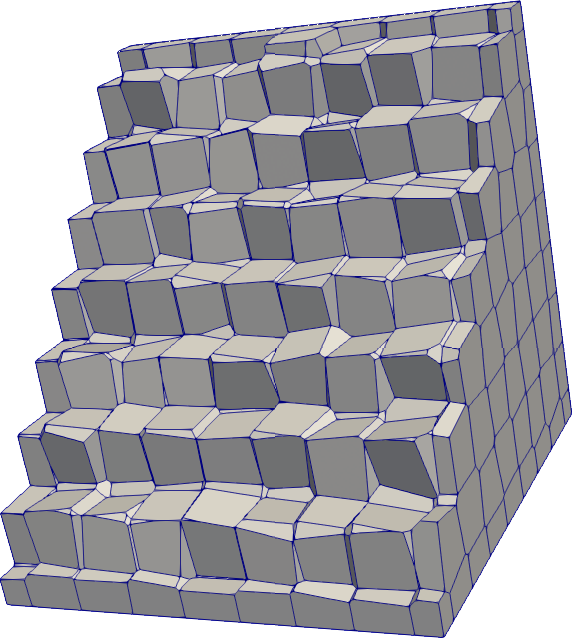}
    \subcaption{``Voro-small-0'' mesh}
  \end{minipage}
  \hspace{0.25cm}
  \begin{minipage}{0.275\textwidth}
    \includegraphics[width=0.90\textwidth]{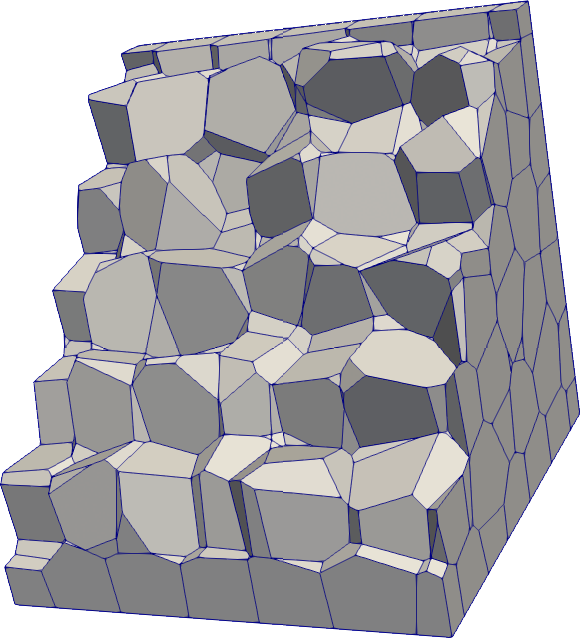}
    \subcaption{``Voro-small-1'' mesh}
  \end{minipage}
  \caption{Members of the refined mesh families used in the numerical example of Section \ref{sec:numerical.examples}.\label{fig:meshes}}
\end{figure}
%


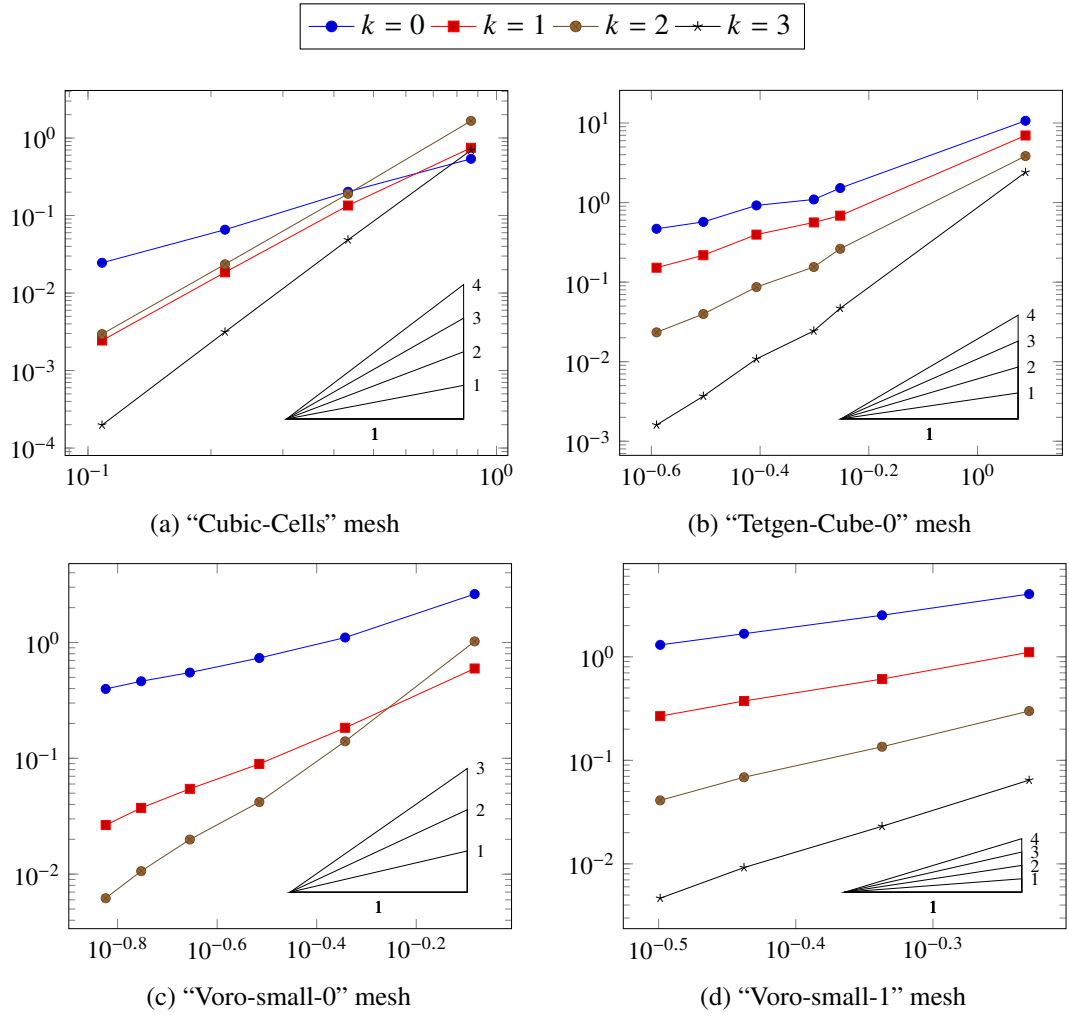
\begin{figure}\centering
  \ref{conv.cubic-cells}
  \vspace{0.50cm}\\
  \begin{minipage}{0.45\textwidth}
    \begin{tikzpicture}[scale=0.85]
      \begin{loglogaxis}[legend columns=-1, legend to name=conv.cubic-cells]        
        \addplot table[x=MeshSize,y=EnergyError] {outputs/unit-permeability/Cubic-Cells_0/data.dat};
        \logLogSlopeTriangle{0.90}{0.4}{0.1}{1}{black};
        \addplot table[x=MeshSize,y=EnergyError] {outputs/unit-permeability/Cubic-Cells_1/data.dat};
        \logLogSlopeTriangle{0.90}{0.4}{0.1}{2}{black};
        \addplot table[x=MeshSize,y=EnergyError] {outputs/unit-permeability/Cubic-Cells_2/data_rates.dat};
        \logLogSlopeTriangle{0.90}{0.4}{0.1}{3}{black};
        \addplot table[x=MeshSize,y=EnergyError] {outputs/unit-permeability/Cubic-Cells_3/data_rates.dat};
        \logLogSlopeTriangle{0.90}{0.4}{0.1}{4}{black};        
        \legend{$k=0$, $k=1$, $k=2$, $k=3$};
      \end{loglogaxis}            
    \end{tikzpicture}
    \subcaption{``Cubic-Cells'' mesh}
  \end{minipage}
  \begin{minipage}{0.45\textwidth}
    \begin{tikzpicture}[scale=0.85]
      \begin{loglogaxis}
        \addplot table[x=MeshSize,y=EnergyError] {outputs/unit-permeability/Tetgen-Cube-0_0/data.dat};
        \logLogSlopeTriangle{0.90}{0.4}{0.1}{1}{black};
        \addplot table[x=MeshSize,y=EnergyError] {outputs/unit-permeability/Tetgen-Cube-0_1/data.dat};
        \logLogSlopeTriangle{0.90}{0.4}{0.1}{2}{black};
        \addplot table[x=MeshSize,y=EnergyError] {outputs/unit-permeability/Tetgen-Cube-0_2/data_rates.dat};
        \logLogSlopeTriangle{0.90}{0.4}{0.1}{3}{black};
        \addplot table[x=MeshSize,y=EnergyError] {outputs/unit-permeability/Tetgen-Cube-0_3/data_rates.dat};
        \logLogSlopeTriangle{0.90}{0.4}{0.1}{4}{black};
      \end{loglogaxis}            
    \end{tikzpicture}
    \subcaption{``Tetgen-Cube-0'' mesh}
  \end{minipage}
  \vspace{0.25cm}\\
  \begin{minipage}{0.45\textwidth}
    \begin{tikzpicture}[scale=0.85]
      \begin{loglogaxis}
        \addplot table[x=MeshSize,y=EnergyError] {outputs/unit-permeability/Voro-small-0_0/data.dat};
        \logLogSlopeTriangle{0.90}{0.4}{0.1}{1}{black};
        \addplot table[x=MeshSize,y=EnergyError] {outputs/unit-permeability/Voro-small-0_1/data.dat};
        \logLogSlopeTriangle{0.90}{0.4}{0.1}{2}{black};
        \addplot table[x=MeshSize,y=EnergyError] {outputs/unit-permeability/Voro-small-0_2/data_rates.dat};
        \logLogSlopeTriangle{0.90}{0.4}{0.1}{3}{black};
      \end{loglogaxis}            
    \end{tikzpicture}
    \subcaption{``Voro-small-0'' mesh}
  \end{minipage}
  \begin{minipage}{0.45\textwidth}
    \begin{tikzpicture}[scale=0.85]
      \begin{loglogaxis}
        \addplot table[x=MeshSize,y=EnergyError] {outputs/unit-permeability/Voro-small-1_0/data.dat};
        \logLogSlopeTriangle{0.90}{0.4}{0.1}{1}{black};
        \addplot table[x=MeshSize,y=EnergyError] {outputs/unit-permeability/Voro-small-1_1/data.dat};
        \logLogSlopeTriangle{0.90}{0.4}{0.1}{2}{black};
        \addplot table[x=MeshSize,y=EnergyError] {outputs/unit-permeability/Voro-small-1_2/data.dat};
        \logLogSlopeTriangle{0.90}{0.4}{0.1}{3}{black};
        \addplot table[x=MeshSize,y=EnergyError] {outputs/unit-permeability/Voro-small-1_3/data_rates.dat};
        \logLogSlopeTriangle{0.90}{0.4}{0.1}{4}{black};
      \end{loglogaxis}            
    \end{tikzpicture}
    \subcaption{``Voro-small-1'' mesh}
  \end{minipage}
  \caption{Energy error $\norm[\rm en]{(\uvec{H}_h-\hat{\uvec{H}}_h,\uvec{A}_h-\hat{\uvec{A}}_h)}$ versus mesh size $h$ for the numerical example of Section \ref{sec:numerical.examples} (unit permeability).\label{fig:conv}}
\end{figure}

Starting with the unit permeability test case, we display in Figures \ref{fig:conv} the error versus the mesh size $h$.
The observed convergence rate is of $h^{k+1}$, with a slight degradation on the ``Voronoi-small-0'' mesh sequence, which can be ascribed to the fact that the regularity parameter increases upon refinement in this case (see the discussion in \cite[Section 5.1.8.2]{Di-Pietro.Droniou:20} for an assessment of the regularity of this mesh sequence, and its numerical impact in the context of HHO schemes).
A full theoretical justification of the fact that the scheme converges with order $k+1$ is postponed to a future work.


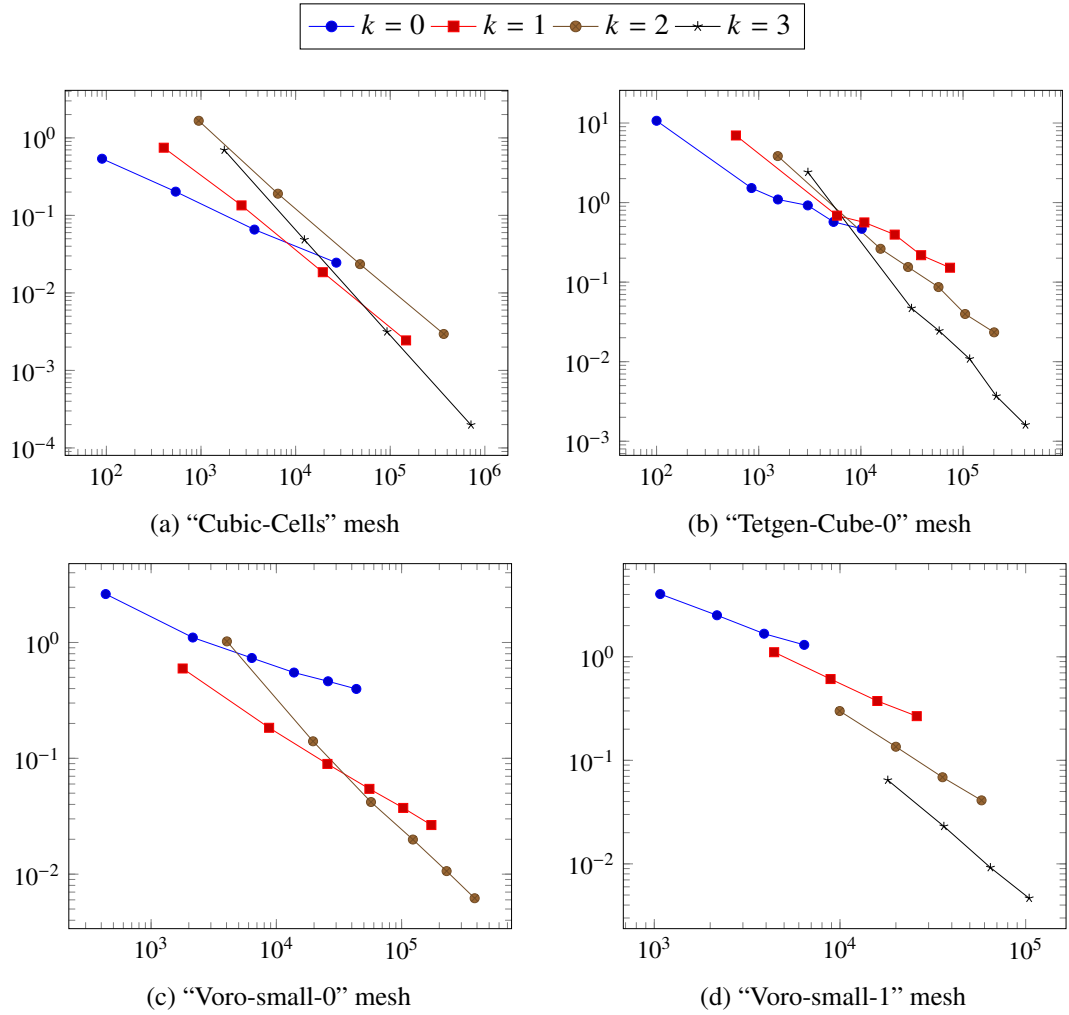
\begin{figure}\centering
  \ref{conv.cubic-cells.dofs}
  \vspace{0.50cm}\\
  \begin{minipage}{0.45\textwidth}
    \begin{tikzpicture}[scale=0.85]
      \begin{loglogaxis}[legend columns=-1, legend to name=conv.cubic-cells.dofs]        
        \addplot table[x expr=\thisrow{DimXCurl}+\thisrow{DimXDiv}
          ,y=EnergyError] {outputs/unit-permeability/Cubic-Cells_0/data.dat};
        \addplot table[x expr=\thisrow{DimXCurl}+\thisrow{DimXDiv}
          ,y=EnergyError] {outputs/unit-permeability/Cubic-Cells_1/data.dat};
        \addplot table[x expr=\thisrow{DimXCurl}+\thisrow{DimXDiv}
          ,y=EnergyError] {outputs/unit-permeability/Cubic-Cells_2/data_rates.dat};
        \addplot table[x expr=\thisrow{DimXCurl}+\thisrow{DimXDiv}
          ,y=EnergyError] {outputs/unit-permeability/Cubic-Cells_3/data_rates.dat};
        \legend{$k=0$, $k=1$, $k=2$, $k=3$};
      \end{loglogaxis}            
    \end{tikzpicture}
    \subcaption{``Cubic-Cells'' mesh}
  \end{minipage}
  \begin{minipage}{0.45\textwidth}
    \begin{tikzpicture}[scale=0.85]
      \begin{loglogaxis}
        \addplot table[x expr=\thisrow{DimXCurl}+\thisrow{DimXDiv}
          ,y=EnergyError] {outputs/unit-permeability/Tetgen-Cube-0_0/data.dat};
        \addplot table[x expr=\thisrow{DimXCurl}+\thisrow{DimXDiv}
          ,y=EnergyError] {outputs/unit-permeability/Tetgen-Cube-0_1/data.dat};
        \addplot table[x expr=\thisrow{DimXCurl}+\thisrow{DimXDiv}
          ,y=EnergyError] {outputs/unit-permeability/Tetgen-Cube-0_2/data_rates.dat};
        \addplot table[x expr=\thisrow{DimXCurl}+\thisrow{DimXDiv}
          ,y=EnergyError] {outputs/unit-permeability/Tetgen-Cube-0_3/data_rates.dat};
      \end{loglogaxis}            
    \end{tikzpicture}
    \subcaption{``Tetgen-Cube-0'' mesh}
  \end{minipage}
  \vspace{0.25cm}\\
  \begin{minipage}{0.45\textwidth}
    \begin{tikzpicture}[scale=0.85]
      \begin{loglogaxis}
        \addplot table[x expr=\thisrow{DimXCurl}+\thisrow{DimXDiv}
          ,y=EnergyError] {outputs/unit-permeability/Voro-small-0_0/data.dat};
        \addplot table[x expr=\thisrow{DimXCurl}+\thisrow{DimXDiv}
          ,y=EnergyError] {outputs/unit-permeability/Voro-small-0_1/data.dat};
        \addplot table[x expr=\thisrow{DimXCurl}+\thisrow{DimXDiv}
          ,y=EnergyError] {outputs/unit-permeability/Voro-small-0_2/data_rates.dat};
      \end{loglogaxis}            
    \end{tikzpicture}
    \subcaption{``Voro-small-0'' mesh}
  \end{minipage}
  \begin{minipage}{0.45\textwidth}
    \begin{tikzpicture}[scale=0.85]
      \begin{loglogaxis}
        \addplot table[x expr=\thisrow{DimXCurl}+\thisrow{DimXDiv}
          ,y=EnergyError] {outputs/unit-permeability/Voro-small-1_0/data.dat};
        \addplot table[x expr=\thisrow{DimXCurl}+\thisrow{DimXDiv}
          ,y=EnergyError] {outputs/unit-permeability/Voro-small-1_1/data.dat};
        \addplot table[x expr=\thisrow{DimXCurl}+\thisrow{DimXDiv}
          ,y=EnergyError] {outputs/unit-permeability/Voro-small-1_2/data.dat};
        \addplot table[x expr=\thisrow{DimXCurl}+\thisrow{DimXDiv}
          ,y=EnergyError] {outputs/unit-permeability/Voro-small-1_3/data_rates.dat};
      \end{loglogaxis}            
    \end{tikzpicture}
    \subcaption{``Voro-small-1'' mesh}
  \end{minipage}
  \caption{Energy error $\norm[\rm en]{(\uvec{H}_h-\hat{\uvec{H}}_h,\uvec{A}_h-\hat{\uvec{A}}_h)}$ versus number of degrees of freedom for the numerical example of Section \ref{sec:numerical.examples} (unit permeability).\label{fig:conv.dofs}}
\end{figure}

To assess the impact of the degree $k$, we plot in Figure \ref{fig:conv.dofs} the energy norm of the error as a function of the total number of degrees of freedom.
For all the considered mesh families, the convergence rate increases as expected with the polynomial degree.
For the standard Cartesian and tetrahedral meshes, a trade-off is present between the mesh size and the approximation degree.
Specifically, on the Cartesian mesh family, the choice $k=0$ is advantageous for less than $\approx2\cdot10^3$ degrees of freedom, while on the tetrahedral mesh family this threshold is $\approx7\cdot10^3$ degrees of freedom.
We notice however that, on generic polyhedral meshes such as those obtained by Voronoi tessellation, increasing the polynomial degree appears to always be advantageous. Complex geometries can be more efficiently meshed using generic polyhedral elements (which result in fewer elements than meshes of tetrahedra, for example); the tests presented here justify the practical interest of developing and using arbitrary-order methods on generic polyhedral meshes.
Finally, it is worth noticing that these results can be further improved resorting to static condensation to eliminate the element (and, possibly, face) unknowns by the local computation of a Schur complement, which we have however not done in the current implementation.


\begin{figure}\centering
  \ref{conv.cubic-cells:var.perm}
  \vspace{0.50cm}\\
  \begin{minipage}{0.45\textwidth}
    \begin{tikzpicture}[scale=0.85]
      \begin{loglogaxis}[legend columns=-1, legend to name=conv.cubic-cells:var.perm]        
        \addplot table[x=MeshSize,y=EnergyError] {outputs/variable-permeability/Cubic-Cells_0/data_rates.dat};
        \logLogSlopeTriangle{0.90}{0.4}{0.1}{1}{black};
        \addplot table[x=MeshSize,y=EnergyError] {outputs/variable-permeability/Cubic-Cells_1/data_rates.dat};
        \logLogSlopeTriangle{0.90}{0.4}{0.1}{2}{black};
        \addplot table[x=MeshSize,y=EnergyError] {outputs/variable-permeability/Cubic-Cells_2/data_rates.dat};
        \logLogSlopeTriangle{0.90}{0.4}{0.1}{3}{black};
        \addplot table[x=MeshSize,y=EnergyError] {outputs/variable-permeability/Cubic-Cells_3/data_rates.dat};
        \logLogSlopeTriangle{0.90}{0.4}{0.1}{4}{black};        
        \legend{$k=0$, $k=1$, $k=2$, $k=3$};
      \end{loglogaxis}            
    \end{tikzpicture}
    \subcaption{``Cubic-Cells'' mesh}
  \end{minipage}
  \begin{minipage}{0.45\textwidth}
    \begin{tikzpicture}[scale=0.85]
      \begin{loglogaxis}
        \addplot table[x=MeshSize,y=EnergyError] {outputs/variable-permeability/Tetgen-Cube-0_0/data_rates.dat};
        \logLogSlopeTriangle{0.90}{0.4}{0.1}{1}{black};
        \addplot table[x=MeshSize,y=EnergyError] {outputs/variable-permeability/Tetgen-Cube-0_1/data_rates.dat};
        \logLogSlopeTriangle{0.90}{0.4}{0.1}{2}{black};
        \addplot table[x=MeshSize,y=EnergyError] {outputs/variable-permeability/Tetgen-Cube-0_2/data_rates.dat};
        \logLogSlopeTriangle{0.90}{0.4}{0.1}{3}{black};
        \addplot table[x=MeshSize,y=EnergyError] {outputs/variable-permeability/Tetgen-Cube-0_3/data_rates.dat};
        \logLogSlopeTriangle{0.90}{0.4}{0.1}{4}{black};
      \end{loglogaxis}            
    \end{tikzpicture}
    \subcaption{``Tetgen-Cube-0'' mesh}
  \end{minipage}
  \vspace{0.25cm}\\
  \begin{minipage}{0.45\textwidth}
    \begin{tikzpicture}[scale=0.85]
      \begin{loglogaxis}
        \addplot table[x=MeshSize,y=EnergyError] {outputs/variable-permeability/Voro-small-0_0/data_rates.dat};
        \logLogSlopeTriangle{0.90}{0.4}{0.1}{1}{black};
        \addplot table[x=MeshSize,y=EnergyError] {outputs/variable-permeability/Voro-small-0_1/data_rates.dat};
        \logLogSlopeTriangle{0.90}{0.4}{0.1}{2}{black};
        \addplot table[x=MeshSize,y=EnergyError] {outputs/variable-permeability/Voro-small-0_2/data_rates.dat};
        \logLogSlopeTriangle{0.90}{0.4}{0.1}{3}{black};
      \end{loglogaxis}            
    \end{tikzpicture}
    \subcaption{``Voro-small-0'' mesh}
  \end{minipage}
  \begin{minipage}{0.45\textwidth}
    \begin{tikzpicture}[scale=0.85]
      \begin{loglogaxis}
        \addplot table[x=MeshSize,y=EnergyError] {outputs/variable-permeability/Voro-small-1_0/data_rates.dat};
        \logLogSlopeTriangle{0.90}{0.4}{0.1}{1}{black};
        \addplot table[x=MeshSize,y=EnergyError] {outputs/variable-permeability/Voro-small-1_1/data_rates.dat};
        \logLogSlopeTriangle{0.90}{0.4}{0.1}{2}{black};
        \addplot table[x=MeshSize,y=EnergyError] {outputs/variable-permeability/Voro-small-1_2/data_rates.dat};
        \logLogSlopeTriangle{0.90}{0.4}{0.1}{3}{black};
        \addplot table[x=MeshSize,y=EnergyError] {outputs/variable-permeability/Voro-small-1_3/data_rates.dat};
        \logLogSlopeTriangle{0.90}{0.4}{0.1}{4}{black};
      \end{loglogaxis}            
    \end{tikzpicture}
    \subcaption{``Voro-small-1'' mesh}
  \end{minipage}
  \caption{Energy error $\norm[\rm en]{(\uvec{H}_h-\hat{\uvec{H}}_h,\uvec{A}_h-\hat{\uvec{A}}_h)}$ versus mesh size $h$ for the numerical example of Section \ref{sec:numerical.examples} (variable permeability).\label{fig:conv:var.perm}}
\end{figure}
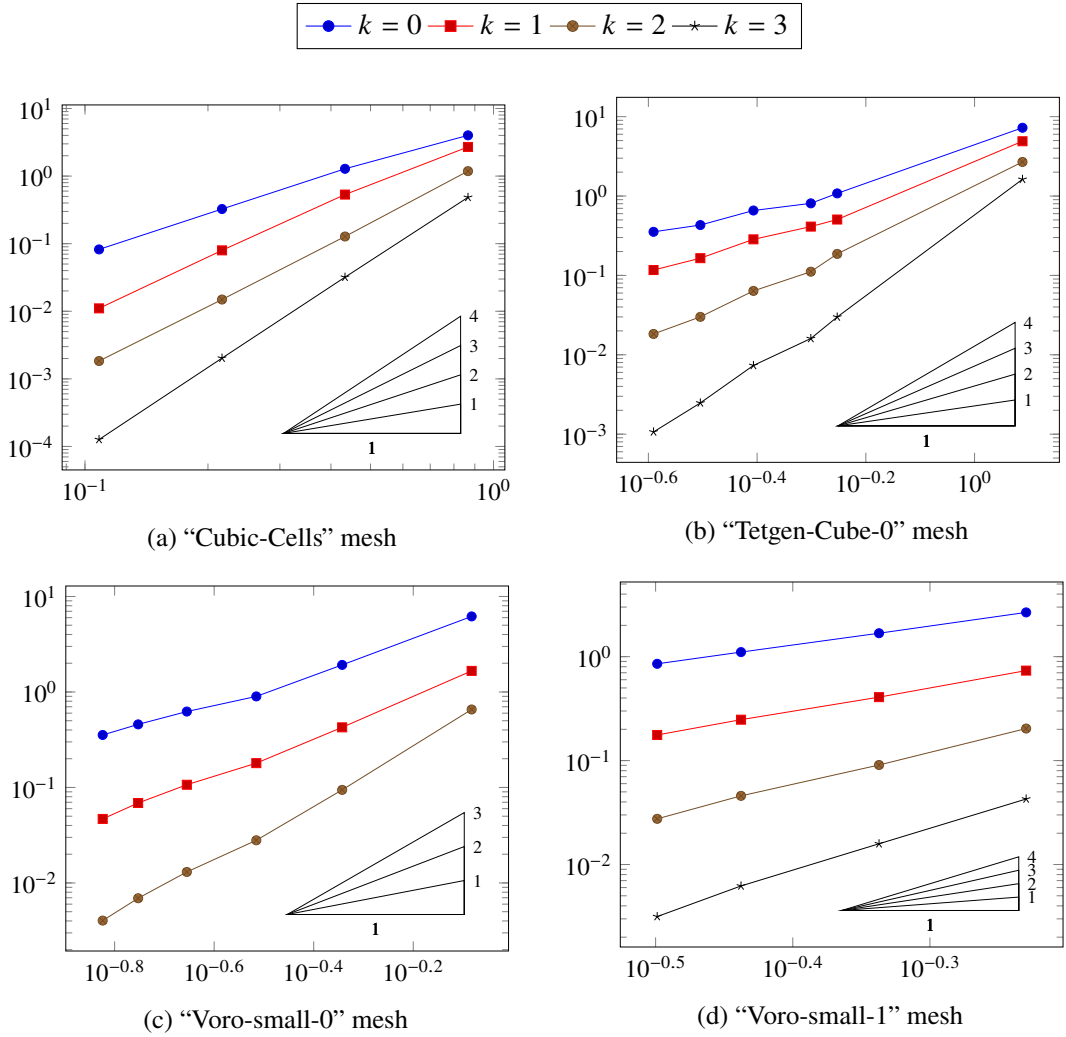

  Moving to the variable-permeability case, the results collected in Figure \ref{fig:conv:var.perm} show essentially the same behaviour as for the unit permeability case in terms of error versus mesh size (compare with Figure \ref{fig:conv}).
  Having allowed the permeability to vary inside each element, we had to increase the degree of exactness for quadrature rules inside each element in the computation of the $\mu$-weighted discrete $L^2$-product in $\Xcurl[h]$ in order to retain optimal convergence rates.


\section{Implementation}\label{sec:implementation}

In this section, we discuss the implementation of the DDR-based method \eqref{eq:discrete}.
We start with the identification of suitable bases for the local polynomial spaces introduced in Section \ref{sec:setting:polynomial.spaces}, then move to the implementation of the discrete vector operators defined in Section \ref{sec:sequence} along with the vector potentials and discrete $\Leb^2$-products of Section \ref{sec:discretisation:potentials.l2-products}.
In what follows, we use the C++ convention that numbering starts from 0.
Vectors and matrices are denoted with simple and bold sans-serif font, respectively.
Intervals of integers are denoted using double brackets so that, e.g., for any $n,m\in\Natural$ with $n<m$, $\llbracket n,m\llbracket$ denotes the set of integers greater or equal than $n$ and strictly smaller than $m$.

Note that the principles described here are very close to those used in the implementation of other polytopal methods; in particular, the reader will find many similarities with the implementation of the Hybrid High-Order method for the Poisson problem described in \cite[Appendix B]{Di-Pietro.Droniou:20}.

\subsection{Bases for the local polynomial spaces}

Let a mesh element $T\in\Th$ and an integer $\ell\ge 0$ be given, set $\NPT{\ell}\coloneq\dim\big(\Poly{\ell}(T)\big)={\ell+3\choose 3}$, and denote by
\[
\BPoly{\ell}\coloneq\left\{
\varphi_{\mathcal{P},T}^i\st i\in\llbracket 0, \NPT{\ell}\llbracket
\right\}
\]
a basis for the space $\Poly{\ell}(T)$, such that $\varphi_{\mathcal{P},T}^0$ is constant and $\int_T\varphi_{\mathcal{P},T}^i=0$ for all $i\in\llbracket 1,\NPT{\ell}\llbracket$.
For ease of presentation, the basis is selected in a hierarchical way, that is, $\BPoly{\ell}\subset \BPoly{\ell+1}$
for all $\ell\ge 0$.
For a more thorough discussion on the choice of the basis $\BPoly{\ell}$, we refer the reader to \cite[Section B.1.1]{Di-Pietro.Droniou:20}.
Here, we limit ourselves to noticing that the choice of $\BPoly{\ell}$ can have a sizeable impact on the conditioning of the discrete problem, especially on distorted meshes.
In practice, especially when non-isotropic elements are present, choosing $\BPoly{\ell}$ orthonormal with respect to the usual $\Leb^2(T)$-inner product can mitigate these issues.
This is the choice made in the numerical examples of Section \ref{sec:numerical.examples}.
A basis
\[
\BPolyd{\ell}\coloneq\left\{
\bvec{\varphi}_{\bvec{\mathcal{P}},T}^i\st i\in\llbracket 0, 3\NPT{\ell}\llbracket
\right\}
\]
for $\Poly{\ell}(T)^3$ can be obtained by tensorisation of $\BPoly{\ell}$ setting $\bvec{\varphi}_{\bvec{\mathcal{P}},T}^i\coloneq\varphi_{\mathcal{P},T}^{(i \% \NPT{\ell})}\bvec{e}_{(i\div\NPT{\ell})}$ for all $i\in\llbracket 0, 3\NPT{\ell}\llbracket$, where $\div$ and $\%$ denote, respectively, the integer division and the modulo operations and, for $j\in\llbracket 0,2\rrbracket$, $\bvec{e}_j$ denotes the $(j+1)$th vector of the canonical basis of $\Real^3$.
For future use, we notice that it holds
\begin{equation}\label{eq:curl.varphi.PTi}
  \CURL\bvec{\varphi}_{\bvec{\mathcal{P}},T}^i
  =\left(\GRAD\varphi_{\mathcal{P},T}^{(i\%\NPT{\ell})}\right)\times \bvec{e}_{(i\div\NPT{\ell})}.
\end{equation}
\medskip

Letting $\NGT{\ell}\coloneq\dim\big(\Goly{\ell}(T)\big)=\NPT{\ell+1}-1$, and recalling that $\GRAD:\Poly{0,\ell+1}(T)\xrightarrow{\cong}\Goly{\ell}(T)$ is an isomorphism, a basis $\BGoly{\ell}$ for the space $\Goly{\ell}(T)$ is obtained setting
\[
\BGoly{\ell}\coloneq\left\{
\bvec{\varphi}_{\bvec{\mathcal{G}},T}^i\coloneq\GRAD\varphi_{\mathcal{P},T}^{i+1}\st
i\in\llbracket 0, \NGT{\ell}\llbracket
\right\}.
\]
In order to find a basis for $\Goly{\ell}(T)^\perp$, the $\Leb^2(T)^3$-orthogonal complement of $\Goly{\ell}(T)$ in $\Poly{\ell}(T)^3$, define the matrix
\[
\MAT{A}_{\bvec{\mathcal{G}},T}^\ell\coloneq\left[
  \int_T\bvec{\varphi}_{\bvec{\mathcal{G}},T}^i\cdot\bvec{\varphi}_{\bvec{\mathcal{P}},T}^j
  \right]_{(i,j)\in\llbracket 0, \NGT{\ell}\llbracket\times\llbracket 0, 3 \NPT{\ell}\llbracket}
\in\Real^{\NGT{\ell}\times 3\NPT{\ell}}.
\]
The right null space of this matrix is formed by the vectors $\VEC{V}=\left[V_j\right]_{j\in\llbracket 0,3\NPT{\ell}\llbracket}\in\Real^{3\NPT{\ell}}$ such that
\[
\MAT{A}_{\bvec{\mathcal{G}},T}^\ell\VEC{V}=\VEC{0}.
\]
Hence, $\VEC{V}$ is in the right null space of $\MAT{A}_{\bvec{\mathcal{G}},T}^\ell$ if and only if $\bvec{v}\coloneq\sum_{j=0}^{3\NPT{\ell}-1}V_j\bvec{\varphi}_{\bvec{\mathcal{P}},T}^j$ satisfies
\[
\left(
\int_T\bvec{\varphi}_{\bvec{\mathcal{G}},T}\cdot\bvec{v}=0\qquad\forall \bvec{\varphi}_{\bvec{\mathcal{G}},T}\in\BGoly{\ell}
\right)
\iff\bvec{v}\in\Goly{\ell}(T)^\perp.
\]
A basis $\BGoly{\ell,\perp}$ for $\Goly{\ell}(T)^\perp$ is thus easily obtained once a basis for the right null space of the matrix $\MAT{A}_{\bvec{\mathcal{G}},T}^\ell$ has been found, which can be done in \texttt{Eigen3} using the method \verb|FullPivLU::kernel|.
As the matrix $\MAT{A}_{\bvec{\mathcal{G}},T}^\ell$ has full rank, this basis is composed of $\NGOT{\ell}\coloneq 3 \NPT{\ell} - \NGT{\ell}$ vectors, each containing the coefficients of the expansion on $\BPolyd{\ell}$ of a basis function $\bvec{\varphi}_{\bvec{\mathcal{G}},T}^{i,\ell,\perp}$ of $\BGoly{\ell,\perp}$, for $i\in \llbracket 0,\NGOT{\ell}\llbracket$.
\medskip

Recalling that $\CURL:\Goly{\ell+1}(T)^\perp\xrightarrow{\cong}\Roly{\ell}(T)$ is an isomorphism, a basis $\BRoly{\ell}$ for $\Roly{\ell}(T)$ is simply obtained taking the curl of the elements of $\bvec{\mathfrak{G}}_T^{\ell+1,\perp}$, that is,
\[
\BRoly{\ell}\coloneq\left\{
\bvec{\varphi}_{\bvec{\mathcal{R}},T}^i\coloneq\CURL\bvec{\varphi}_{\bvec{\mathcal{G}},T}^{i,\ell+1,\perp}\st i\in\llbracket 0,\NGOT{\ell+1}\llbracket
\right\}.
\]
Notice that the curl of $\bvec{\varphi}_{\bvec{\mathcal{G}},T}^{i,\ell+1,\perp}\in\BGOoly{\ell+1}$ is easily computable from \eqref{eq:curl.varphi.PTi} using the expansion of this function in $\BPolyd{\ell+1}$.
Finally, a basis $\BROoly{\ell}$ for $\Roly{\ell}(T)^\perp$ is obtained computing a basis for the right null space of the matrix
\[
\MAT{A}_{\bvec{\mathcal{R}},T}^\ell\coloneq\left[
  \int_T\bvec{\varphi}_{\bvec{\mathcal{R}},T}^i\cdot\bvec{\varphi}_{\bvec{\mathcal{P}},T}^j
  \right]_{(i,j)\in\llbracket 0,\NRT{\ell}\llbracket\times\llbracket0, 3\NPT{\ell}\llbracket},
\]
where $\NRT{\ell}\coloneq\NGOT{\ell+1}$ is the dimension of $\Roly{\ell}(T)$.

For any $F\in\Fh$, bases for the spaces $\Poly{\ell}(F)$, $\Poly{\ell}(F)^2$, $\Goly{\ell}(F)$, $\Goly{\ell}(F)^\perp$, $\Roly{\ell}(F)$, $\Roly{\ell}(F)^\perp$ can be obtained in a similar way considering a local orthogonal system of coordinates.

\subsection{Vector operators and potentials}

The construction of the discrete vector operators and potentials requires the solution of local problems on mesh elements and faces.
For the sake of simplicity, we only detail the construction of the discrete divergence operator defined by \eqref{eq:DT} and of the corresponding potential defined by \eqref{eq:PdivT} for a given mesh element $T\in\Th$.
The construction of the discrete curl operators and of the corresponding vector potentials on mesh faces and elements follows similar principles.

According to the discussion in the previous section, a basis $\BXdivT$ for $\Xdiv$ can be obtained taking the Cartesian product of the bases for the spaces that compose $\Xdiv$, that is,
\begin{equation}\label{eq:BXdivT}
  \BXdivT\coloneq\BGoly{k-1}\times\BGOoly{k}\times\bigtimes_{F\in\FT}\BPoly[F]{k}.
\end{equation}
Let $\uvec{v}_T=(\bvec{v}_{\cvec{G},T},\bvec{v}_{\cvec{G},T}^\perp,(v_F)_{F\in\FT})\in\Xdiv$ be given, and denote by $\UVEC{V}_T$ the corresponding vector of degrees of freedom, partitioned as follows:
\[
\UVEC{V}_T=\begin{bmatrix}
\VEC{V}_{\cvec{G},T} \\ \VEC{V}_{\cvec{G},T}^\perp \\ \VEC{V}_{F_1} \\ \vdots \\ \VEC{V}_{F_{\card(\FT)}}
\end{bmatrix}\in\Real^{\NDIV}.
\]
Above, letting $\NPF{k}\coloneq\dim\big(\Poly{k}(F)\big)={k+2\choose 2}$, we have set $\NDIV\coloneq\dim\big(\Xdiv\big)=\NGT{k-1}+\NGOT{k}+\card(\FT)\NPF{k}$ and the subvectors
$\VEC{V}_{\cvec{G},T}\in\Real^{\NGT{k-1}}$,
$\VEC{V}_{\cvec{G},T}^\perp\in\Real^{\NGOT{k}}$,
and $\VEC{V}_F\in\Real^{\NPF{k}}$, $F\in\FT$,
collect the coefficients of the expansions of $\bvec{v}_{\cvec{G},T}$, $\bvec{v}_{\cvec{G},T}^\perp$ and $v_F$ in $\BGoly{k-1}$, $\BGOoly{k}$, and $\BPoly[F]{k}$, respectively.

Denoting by $\VEC{D}_T\in\Real^{\NPT{k}}$ the vector collecting the coefficients of the expansions of $\DT\uvec{v}_T$ in $\BPoly{k}$, the algebraic realisation of \eqref{eq:DT} is
\begin{equation}\label{eq:DT:algebraic}
  \MAT{M}_{D,T}\VEC{D}_T
  = -\MAT{B}_{D,T}\VEC{V}_{\cvec{G},T} + \sum_{F\in\FT}\omega_{TF}\MAT{B}_{D,F}\VEC{V}_F,
\end{equation}
with
\[
\begin{aligned}
  \MAT{M}_{D,T}
  &\coloneq\begin{bmatrix}
  \int_T\varphi_{\mathcal{P},T}^i\varphi_{\mathcal{P},T}^j
  \end{bmatrix}_{(i,j)\in\llbracket 0,\NPT{k}\llbracket^2},
  \\
  \MAT{B}_{D,T}
  &\coloneq\begin{bmatrix}
  \int_T\GRAD\varphi_{\mathcal{P},T}^i\cdot\bvec{\varphi}_{\cvec{G},T}^j
  \end{bmatrix}_{(i,j)\in\llbracket0,\NPT{k}\llbracket\times\llbracket 0,\NGT{k}\llbracket},
  \\
  \MAT{B}_{D,F}
  &\coloneq\begin{bmatrix}
  \int_F\varphi_{\mathcal{P},T}^i\varphi_{\mathcal{P},F}^j
  \end{bmatrix}_{(i,j)\in\llbracket0,\NPT{k}\llbracket\times\llbracket0,\NPF{k}\llbracket}.
\end{aligned}
\]
The linear operator $\DT:\Xdiv\to\Poly{k}(T)$ is thus represented, in the selected bases for $\Xdiv$ and $\Poly{k}(T)$, by the matrix $\MAT{D}_T\in\Real^{\NPT{k}\times\NDIV}$ whose $i$th column is the solution of the algebraic problem \eqref{eq:DT:algebraic} for $\UVEC{V}_T=\VEC{e}_i$, with $\VEC{e}_i$ denoting the $i$th vector of the canonical basis of $\Real^{\NDIV}$.
In the spirit of \cite[Appendix B]{Di-Pietro.Droniou:20}, these conditions can be translated into an explicit single equation on $\MAT{D}_T$.

The computation of the vector potential on $\Xdiv$ follows similar principles.
With the same notations as above, and additionally denoting by $\VEC{P}_{\DIV,T}\in\Real^{3\NPT{k}}$ the vector collecting the coefficients of the expansion of $\PdivT\uvec{v}_T$ in $\BPolyd{k}$, the algebraic realisation of \eqref{eq:PdivT} reads
\begin{equation}\label{eq:PdivT:algebraic}
  \begin{bmatrix}
    \MAT{M}_{P,T} \\
    \MAT{M}_{P,T}^{\perp}
  \end{bmatrix}
  \VEC{P}_{\DIV,T}
  = \begin{bmatrix}
    -\MAT{B}_{P,T}\MAT{D}_T\UVEC{V}_T + \sum_{F\in\FT}\omega_{TF}\MAT{B}_{P,F}\VEC{V}_F \\
    \MAT{B}_{P,T}^\perp\VEC{V}_{\cvec{G},T}^\perp
  \end{bmatrix},
\end{equation}
where
\[
\begin{aligned}
  \MAT{M}_{P,T} &\coloneq
  \begin{bmatrix}
    \int_T\GRAD\varphi_{\mathcal{P},T}^i\cdot\bvec{\varphi}_{\cvec{P},T}^j
  \end{bmatrix}_{(i,j)\in\llbracket1,\NPT{k+1}\llbracket\times\llbracket0,3\NPT{k}\llbracket},
  \\
  \MAT{M}_{P,T}^\perp &\coloneq
  \begin{bmatrix}
    \int_T\bvec{\varphi}_{\cvec{G},T}^{i,k,\perp}\cdot\bvec{\varphi}_{\cvec{P},T}^j
  \end{bmatrix}_{(i,j)\in\llbracket0,\NGOT{k}\llbracket\times\llbracket0,3\NPT{k}\llbracket},
  \\
  \MAT{B}_{P,T} &\coloneq
  \begin{bmatrix}
    \int_T\varphi_{\mathcal{P},T}^i\varphi_{\mathcal{P},T}^j
  \end{bmatrix}_{(i,j)\in\llbracket1,\NPT{k+1}\llbracket\times\llbracket 0,\NPT{k}\llbracket},
  \\
  \MAT{B}_{P,T}^\perp &\coloneq
  \begin{bmatrix}
    \int_T\bvec{\varphi}_{\cvec{G},T}^{i,k,\perp}\cdot\bvec{\varphi}_{\cvec{G},T}^{j,k,\perp}
  \end{bmatrix}_{(i,j)\in\llbracket0,\NGOT{k}\llbracket^2},
  \\
  \MAT{B}_F &\coloneq
  \begin{bmatrix}
    \int_F\varphi_{\mathcal{P},T}^i\varphi_{\mathcal{P},F}^j
  \end{bmatrix}_{(i,j)\in\llbracket1,\NPT{k+1}\llbracket\times\llbracket0,\NPF{k}\llbracket}
  \qquad\forall F\in\FT.
\end{aligned}
\]
Hence, $\PdivT:\Xdiv\to\Poly{k}(T)^3$ is represented by the matrix $\MAT{P}_{\DIV,T}\in\Real^{3\NPT{k}\times\NDIV}$ whose $i$th column is the solution of the algebraic problem \eqref{eq:PdivT:algebraic} for $\UVEC{V}_T=\VEC{e}_i$, $i\in\llbracket 0,\NDIV\llbracket$.

\subsection{Discrete $\Leb^2$-products and bilinear forms}

The last ingredient for the implementation of the DDR-based scheme \eqref{eq:discrete} are the discrete $\Leb^2$-products in $\Xcurl[h]$ and $\Xdiv[h]$.
For the sake of simplicity, we will focus here on the local $\Leb^2$-product in $\Xdiv$ defined by \eqref{eq:l2-product:XdivT}.
The corresponding global $\Leb^2$-product is obtained assembling the local contributions element by element in the usual (Finite Element) way.
The $\Leb^2$-product in $\Xcurl$ is constructed following the same general ideas.

The matrix representing $(\cdot,\cdot)_{\DIV,T}$ in the basis \eqref{eq:BXdivT} for $\Xdiv$ is
\[
\MAT{L}_{\DIV,T}\coloneq
\MAT{P}_{\DIV,T}^\trans\MAT{M}_{\DIV,T}\MAT{P}_{\DIV,T} + \MAT{S}_{\DIV,T}
\in\Real^{\NDIV\times\NDIV},
\]
where $\MAT{M}_{\DIV,T}\in\Real^{3\NPT{k+1}\times3\NPT{k+1}}$ is the mass matrix of $\Poly{k}(T)^3$, while $\MAT{S}_{\DIV,T}$ is the matrix counterpart of the stabilisation bilinear form $\mathrm{s}_{\DIV,T}$ defined by \eqref{eq:sdivT}.
This stabilisation bilinear form penalises in a least-square sense the difference between (projections of) the potential reconstruction and the corresponding components of the vector of polynomials in $\Xdiv$.
In the spirit of \cite[Section B.2.2]{Di-Pietro.Droniou:20}, the difference operator $\Xdiv[T]\ni\uvec{v}_T\mapsto\Gproj{k-1}{T}(\PdivT\uvec{v}_T)-\bvec{v}_{\cvec{G},T}\in\Goly{k-1}(T)$ is represented, in the selected bases for $\Xdiv[T]$ and $\Goly{k-1}(T)$, by the matrix
\begin{equation}\label{eq:Delta.div.T}
  \bvec{\Delta}_{\DIV,T}\coloneq
  \bvec{\Pi}_{\cvec{G},T}^{k-1}\MAT{P}_{\DIV,T}
  - \begin{bmatrix}
    \MAT{I}_{\NGT{k-1}}\MAT{0}\cdots\MAT{0}
  \end{bmatrix}
  \in\Real^{\NGT{k-1}\times\NDIV},
\end{equation}
where, denoting by $\MAT{M}_{\cvec{G},T}$ the mass matrix of $\Goly{k-1}(T)$,
\[
\bvec{\Pi}_{\cvec{G},T}^{k-1}\coloneq\MAT{M}_{\cvec{G},T}^{-1}\begin{bmatrix}
  \int_T\bvec{\varphi}_{\cvec{G},T}^i\cdot\bvec{\varphi}_{\cvec{P},T}^j
\end{bmatrix}_{(i,j)\in\llbracket0,\NGT{k-1}\llbracket\times\llbracket0,3\NPT{k}\llbracket}
\]
is the matrix representing the $\Leb^2$-orthogonal projector $\Gproj{k-1}{T}$ with domain restricted to $\Poly{k}(T)^3$, while the identity matrix $\MAT{I}_{\NGT{k-1}}$ in \eqref{eq:Delta.div.T} occupies the columns corresponding to the component $\bvec{v}_{\cvec{G},T}$.
Similarly, for all $F\in\FT$, the face difference operator $\Xdiv[T]\ni\uvec{v}_T\mapsto\PdivT\uvec{v}_T\cdot\normal_F-v_F\in\Poly{k}(F)$ is represented by the matrix
\begin{equation}\label{eq:Delta.div.F}
  \bvec{\Delta}_{\DIV,F}\coloneq
  \MAT{T}_F\MAT{P}_{\DIV,T}-\begin{bmatrix}
  \MAT{0}\cdots\MAT{I}_{\NPF{k}}\cdots\MAT{0}
  \end{bmatrix},
\end{equation}
where the matrix $\MAT{T}_F$ representing the normal trace operator applied to functions in $\Poly{k}(T)^3$ is such that, denoting by $\MAT{M}_{\mathcal{P},F}$ the mass matrix of $\Poly{k}(F)$,
\[
\MAT{T}_F\coloneq\MAT{M}_{\mathcal{P},F}^{-1}\begin{bmatrix}
\int_F\varphi_{\mathcal{P},F}^i(\bvec{\varphi}_{\cvec{P},T}^j\cdot\normal_F)
\end{bmatrix}_{(i,j)\in\llbracket0,\NPF{k}\llbracket\times\llbracket0,3\NPT{k}\llbracket},
\]
while the identity matrix $\MAT{I}_{\NPF{k}}$ in \eqref{eq:Delta.div.F} occupies the columns corresponding to the component $v_F$.
Finally, denoting, for all $F\in\FT$, by $\MAT{M}_{\mathcal{P},F}$ the mass matrix of $\Poly{k}(F)$, we set
\[
\MAT{S}_{\DIV,T}
\coloneq
\bvec{\Delta}_{\DIV,T}^\trans\MAT{M}_{\cvec{G},T}\bvec{\Delta}_{\DIV,T}
+ \sum_{F\in\FT} h_F\bvec{\Delta}_{\DIV,F}^\trans\MAT{M}_{\mathcal{P},F}\bvec{\Delta}_{\DIV,F}.
\]


\appendix 

\section{Proof of the inf--sup estimate (Theorem \ref{thm:stability})}\label{sec:inf-sup}

Throughout this section, it is assumed that the mesh $\Mh$ belongs to a regular sequence in the sense made precise in Section \ref{sec:discrete:well-posedness}, and that the permeability $\mu$ is piecewise constant on this mesh.
Mesh regularity implies, in particular, that the diameter $h_T$ of a mesh element $T\in\Th$ is comparable to the diameters $h_F$ and $h_E$ of any of its faces $F\in\FT$ or edges $E\in\ET$, uniformly for all meshes in the sequence.
Assumption \ref{assum:poincare.edges}, on the other hand, is only useful for Theorem \ref{thm:poincare.Ch} and, as a consequence, for the proof of the inf--sup estimate.
For the sake of conciseness, we write $a\lesssim b$ as a shorthand for $a\le Cb$ with $C$ independent of $h$, of the chosen element/face/edge, and of the chosen functions involved in the quantities $a,b$ (so $C$ depends only on $\Omega$, $\mu$, the mesh regularity parameter and, when appropriate, on $\alpha$ in Assumption \ref{assum:poincare.edges}). The notation $a\approx b$ means that $a\lesssim b$ and $b\lesssim a$. 

We first establish a few results on the interpolators, operators, and potentials linked to the portion of the DDR sequence \eqref{eq:ddr.sequence} relevant to our purpose. Similar results can also be proved the remaining portion of the sequence. The analysis is more easily carried out using $\Leb^2(\Omega)^3$-like norms on $\Xcurl[h]$ and $\Xdiv[h]$ that are equivalent to, but not coincident with, $\norm[\mu,\CURL,h]{{\cdot}}$ and $\norm[\DIV,h]{{\cdot}}$, respectively.
Specifically, we let
\begin{subequations}
\label{eq:def.tnorm.curl}
  \begin{gather}\label{eq:tnorm.curl.h}
    \text{
      $\tnorm[\CURL,h]{\uvec{\zeta}_h}\coloneq\left(
      \sum_{T\in\Th}\tnorm[\CURL,T]{\uvec{\zeta}_T}^2
      \right)^{\nicefrac12}$
      for all $\uvec{\zeta}_h\in\Xcurl[h]$
    }
    \\ \label{eq:tnorm.curl.T}
    \text{
      with
      $\tnorm[\CURL,T]{\uvec{\zeta}_T}\coloneq\left(
      \norm[T]{\bvec{\zeta}_{\cvec{R},T}\!+\bvec{\zeta}_{\cvec{R},T}^\perp}^2
      +\!\!\!\sum_{F\in\FT}\!\!h_F\tnorm[\CURL,F]{\uvec{\zeta}_F}^2
      \right)^{\nicefrac12}$
      for all $T\in\Th$ and all $\uvec{\zeta}_T\!\!\in\!\Xcurl[T]$
    }
    \\ \label{eq:tnorm.curl.F}
    \text{
      and
      $\tnorm[\CURL,F]{\uvec{\zeta}_F}\coloneq\left(
      \norm[F]{\bvec{\zeta}_{\cvec{R},F}+\bvec{\zeta}_{\cvec{R},F}^\perp}^2
      +\!\! \sum_{E\in\EF}\!\! h_E\norm[E]{\zeta_E}^2
      \right)^{\nicefrac12}$
      for all $F\in\Fh$ and all $\uvec{\zeta}_F\in\Xcurl[F]$,
    }
  \end{gather}
\end{subequations}
and
\begin{equation}\label{eq:def.tnorm.div}
\begin{gathered}
  \text{
    $\tnorm[\DIV,h]{\uvec{v}_h}\coloneq\left(
    \sum_{T\in\Th}\tnorm[\DIV,T]{\uvec{v}_T}^2
    \right)^{\frac12}$
    for all $\uvec{v}_h\in\Xdiv[h]$
  }
  \\
  \text{
    with $\tnorm[\DIV,T]{\uvec{v}_T}\coloneq\left(
    \norm[T]{\bvec{v}_{\cvec{G},T} + \bvec{v}_{\cvec{G},T}^\perp}^2
    +\!\! \sum_{F\in\FT}\!\!h_F\norm[F]{v_F}^2
    \right)^{\nicefrac12}$
    for all $T\in\Th$ and all $\uvec{v}_T\in\Xdiv$.
  }
\end{gathered}
\end{equation}
We note that, by orthogonality, in the expressions above we have $\norm[T]{\bvec{\zeta}_{\cvec{R},T}+\bvec{\zeta}_{\cvec{R},T}^\perp}^2=\norm[T]{\bvec{\zeta}_{\cvec{R},T}}^2+\norm[T]{\bvec{\zeta}_{\cvec{R},T}^\perp}^2$, $\norm[F]{\bvec{\zeta}_{\cvec{R},F}+\bvec{\zeta}_{\cvec{R},F}^\perp}^2=\norm[F]{\bvec{\zeta}_{\cvec{R},F}}^2+\norm[F]{\bvec{\zeta}_{\cvec{R},F}^\perp}^2$ and $\norm[T]{\bvec{v}_{\cvec{G},T} + \bvec{v}_{\cvec{G},T}^\perp}^2=\norm[T]{\bvec{v}_{\cvec{G},T}}^2+\norm[T]{\bvec{v}_{\cvec{G},T}^\perp}^2$.
The powers of the face and edge diameters in front of the various contributions to the norms $\tnorm[\CURL,h]{{\cdot}}$ and $\tnorm[\DIV,h]{{\cdot}}$ are selected so as to ensure that all the terms have the same scaling.

\subsection{Boundedness of curl operators}

\begin{lemma}[Isomorphism property of {$\CURL$} on elements]\label{lem:CURL.iso}
For all $\ell\ge 0$ and all $T\in\Th$, the mapping $\CURL:\Goly{\ell+1}(T)^\perp\xrightarrow{\cong}\Roly{\ell}(T)$ is an isomorphism and
\begin{equation}\label{eq:est.r.CURL}
\norm[T]{\bvec{r}}\lesssim h_T\norm[T]{\CURL\bvec{r}}\qquad\forall \bvec{r}\in \Goly{\ell+1}(T)^\perp.
\end{equation}
\end{lemma}

\begin{proof}
The fact that $\CURL$ establishes an isomorphism between $\Goly{\ell+1}(T)^\perp$ and $\Roly{\ell}(T)$ follows from \cite[Corollary 7.3]{Arnold:18}. We therefore focus on proving \eqref{eq:est.r.CURL}. 
Using a scaling argument as in, e.g., \cite[Lemma 1.28]{Di-Pietro.Droniou:20}, we can assume that $T$ has diameter $1$, is contained in the unit ball $\ball{1}$ centred at $0$, and contains the ball $\ball{\varrho}$ centered at 0 and of size $\varrho\gtrsim 1$. Let $\bvec{r}\in \Goly{\ell+1}(T)^\perp$, which can be considered as a polynomial $\widetilde{\bvec{r}}$ on $\ball{1}$. Set $\bvec{q}\coloneq\CURL\widetilde{\bvec{r}}\in\Roly{\ell}(\ball{1})$. The mapping $\CURL:\Goly{\ell+1}(\ball{1})^\perp\to\Roly{\ell}(\ball{1})$ is an isomorphism, and its inverse is therefore continuous with a norm that only depends on these spaces -- that is, only on $\ell$ and $d$. There exists therefore $\widehat{\bvec{r}}\in\Goly{\ell+1}(\ball{1})^\perp\subset\Poly{\ell+1}(\ball{1})^3$ such that $\CURL\widehat{\bvec{r}}=\bvec{q}=\CURL\widetilde{\bvec{r}}$ and
\begin{equation}\label{eq:est.wr}
\norm[\ball{1}]{\widehat{\bvec{r}}}\lesssim \norm[\ball{1}]{\bvec{q}}.
\end{equation}
In particular, $\CURL(\widehat{\bvec{r}}_{|T}-\bvec{r})=(\CURL\widehat{\bvec{r}})_{|T}-\CURL\bvec{r}=\bvec{q}_{|T}-\CURL\bvec{r}=\bvec{0}$, and thus $\widehat{\bvec{r}}_{|T}-\bvec{r}$ lies in the kernel of $\CURL$ on $\Poly{\ell+1}(T)^3$, which is $\Goly{\ell+1}(T)$. Taking the projection on $\Goly{\ell+1}(T)^\perp$, to which $\bvec{r}$ belongs, we infer $\GOproj{\ell+1}{T}\widehat{\bvec{r}}_{|T}=\GOproj{\ell+1}{T}\bvec{r}=\bvec{r}$.
The estimate \eqref{eq:est.wr} then yields
\[
\norm[T]{\bvec{r}}=\norm[T]{\GOproj{\ell+1}{T}\widehat{\bvec{r}}_{|T}}\le\norm[T]{\widehat{\bvec{r}}}\le
\norm[\ball{1}]{\widehat{\bvec{r}}}\lesssim \norm[\ball{1}]{\bvec{q}}\lesssim \norm[\ball{\varrho}]{\bvec{q}}\le\norm[T]{\bvec{q}},
\]
the second last inequality following from \cite[Eq. (1.43)]{Di-Pietro.Droniou:20} and the fact that $\bvec{q}\in\Poly{\ell}(T)^3$. Since we are in a situation where $T$ has diameter $h_T=1$, this concludes the proof. \end{proof}

\begin{proposition}[Boundedness of curl operators and vector potentials]\label{prop:bd.curl}
It holds
\begin{gather}\label{eq:bd.CF.trFt}
  \text{
    $\norm[F]{\CF \uvec{\upsilon}_F}\lesssim h_F^{-1}\tnorm[\CURL,F]{\uvec{\upsilon}_F}$ and
    $\norm[F]{\trFt\uvec{\upsilon}_F}\lesssim \tnorm[\CURL,F]{\uvec{\upsilon}_F}$
  }\quad\forall F\in\Fh\,,\quad\forall \uvec{\upsilon}_F\in \Xcurl[F],
  \\ \label{eq:bd.uCT.PcurlT}
  \text{
    $\tnorm[\DIV,T]{\uCT \uvec{\upsilon}_T}\lesssim h_T^{-1}\tnorm[\CURL,T]{\uvec{\upsilon}_T}$ and 
    $\norm[T]{\PcurlT\uvec{\upsilon}_T}\lesssim \tnorm[\CURL,T]{\uvec{\upsilon}_T}$
  }\quad\forall T\in\Th\,,\quad\forall\uvec{\upsilon}_T\in \Xcurl[T].
\end{gather}
\end{proposition}

\begin{proof}
Let $q\in\Poly{k}(F)$ and apply the Cauchy--Schwarz inequality along with the inverse and discrete trace inequalities of \cite[Sections 1.2.5 and 1.2.6]{Di-Pietro.Droniou:20} to \eqref{eq:CF} to write
\[
\int_F \CF\uvec{\upsilon}_F\, q\lesssim \norm[F]{\bvec{\upsilon}_{\cvec{R},F}}h_F^{-1}\norm[F]{q}+
\sum_{E\in\EF}\norm[E]{\upsilon_E} h_F^{-\frac12}\norm[F]{q}.
\]
Taking the supremum over the set $\left\{q\in\Poly{k}(F)\st\norm[F]{q}\le 1\right\}$ and recalling that $h_F\approx h_E$ by mesh regularity leads to 
\begin{equation}\label{est:CF.final}
\norm[F]{\CF\uvec{\upsilon}_F}\lesssim h_F^{-1}\left(\norm[F]{\bvec{\upsilon}_{\cvec{R},F}}+\sum_{E\in\EF}h_E^{\frac12}\norm[E]{\upsilon_E} \right)\lesssim h_F^{-1}\tnorm[\CURL,F]{\uvec{\upsilon}_F}.
\end{equation}
This proves the estimate on $\CF\uvec{\upsilon}_F$ in \eqref{eq:bd.CF.trFt}.

Let $\bvec{\sigma}\in\Poly{k}(F)^2$ and let $r\in\Poly{0,k+1}(F)$ be such that $\VROT_F r=\Rproj{k}{F}\bvec{\sigma}$, that is, $\GRAD_F r=\rotation{\nicefrac{\pi}{2}}(\Rproj{k}{F}\bvec{\sigma})$. The local Poincar\'e--Wirtinger inequality \cite[Remark 1.46]{Di-Pietro.Droniou:20} then yields 
\begin{equation}\label{eq:est.r}
  \norm[F]{r}\lesssim h_F\norm[F]{\Rproj{k}{F}\bvec{\sigma}}
  \le h_F\norm[F]{\bvec{\sigma}},
\end{equation}
the last inequality being a consequence of the $\Leb^2(F)^2$-boundedness of $\Rproj{k}{F}$.
Applying the definition \eqref{eq:trFt} of $\trFt\uvec{\upsilon}_F$ to this $r$ and $\bvec{\tau}=\ROproj{k}{F}\bvec{\sigma}$, we obtain, with the help of Cauchy--Schwarz and discrete trace inequalities
\begin{align*}
\int_F\trFt\uvec{\upsilon}_F\cdot\bvec{\sigma}\lesssim{}& \norm[F]{\CF\uvec{\upsilon}_F}\norm[F]{r}+\sum_{E\in\EF}
\norm[E]{\upsilon_E}h_F^{-\frac12}\norm[F]{r}+\norm[F]{\bvec{\upsilon}_{\cvec{R},F}^\perp}\norm[F]{\bvec{\sigma}}\\
\lesssim{}&\tnorm[\CURL,F]{\uvec{\upsilon}_F}\norm[F]{\bvec{\sigma}}+\sum_{E\in\EF}
\norm[E]{\upsilon_E}h_E^{\frac12}\norm[F]{\bvec{\sigma}}+\norm[F]{\bvec{\upsilon}_{\cvec{R},F}^\perp}\norm[F]{\bvec{\sigma}}
\end{align*}
where the conclusion follows from \eqref{est:CF.final} along with \eqref{eq:est.r} and $h_F\approx h_E$. The estimate on $\trFt\uvec{\upsilon}_F$ follows by taking the supremum over $\bvec{\sigma}\in\Poly{k}(F)^2$ such that $\norm[F]{\bvec{\sigma}}\le 1$.

The estimates \eqref{eq:bd.uCT.PcurlT} are obtained applying the same arguments as above, using the boundedness of $\trFt$ stated in \eqref{eq:bd.CF.trFt} and invoking Lemma \ref{lem:CURL.iso} in lieu of the Poincar\'e--Wirtinger inequality to ensure that, for a given $\bvec{\sigma}\in\Poly{k}(T)^3$, the polynomial $\bvec{v}\in \Goly{k+1}(T)^\perp$ to be used in \eqref{eq:PcurlT} and such that $\CURL\bvec{v}=\Rproj{k}{T}\bvec{\sigma}$ satisfies $\norm[T]{\bvec{v}}\lesssim h_T\norm[T]{\bvec{\sigma}}$.
\end{proof}

\subsection{Equivalence of norms}

In line with the notations in \eqref{eq:def.tnorm.curl} and \eqref{eq:def.tnorm.div}, we denote by $\norm[\mu,\CURL,T]{{\cdot}}$ (resp.\ $\norm[\DIV,T]{{\cdot}}$) the restriction to $\Xcurl[T]$ (resp.\ $\Xdiv[T]$) of $\norm[\mu,\CURL,h]{{\cdot}}$ (resp.\ $\norm[\DIV,h]{{\cdot}}$).

\begin{proposition}[Equivalence of norms on discrete spaces]
For $\clubsuit=T\in\Th$ or $\clubsuit=h$, it holds
\begin{align}\label{eq:equiv.Xcurl}
\norm[\mu,\CURL,\clubsuit]{\uvec{\upsilon}_\clubsuit}\approx \tnorm[\CURL,\clubsuit]{\uvec{\upsilon}_\clubsuit}\qquad\forall \uvec{\upsilon}_\clubsuit\in\Xcurl[\clubsuit],
\\
\label{eq:equiv.Xdiv}
\norm[\DIV,\clubsuit]{\uvec{w}_\clubsuit}\approx \tnorm[\DIV,\clubsuit]{\uvec{w}_\clubsuit}\qquad\forall \uvec{w}_\clubsuit\in\Xdiv[\clubsuit].
\end{align}
\end{proposition}

\begin{proof}
We only prove \eqref{eq:equiv.Xcurl}, the equivalence \eqref{eq:equiv.Xdiv} being obtained similarly. The case $\clubsuit=h$ follows summing over $T$ the squares of the cases $\clubsuit=T$, so we focus on this latter situation. Let $T\in\Th$ and $\uvec{\upsilon}_T\in\Xcurl[T]$. The definition \eqref{eq:def.inner.curl} of the inner product to which $\norm[\mu,\CURL,T]{{\cdot}}$ corresponds and the $\Leb^2(F)^2$-orthogonality of $\Rproj{k-1}{F}(\PcurlT\uvec{\upsilon}_T)-\bvec{\upsilon}_{\cvec{R},F}\in\Roly{k-1}(F)$ and $\ROproj{k}{F}(\PcurlT\uvec{\upsilon}_T)-\bvec{\upsilon}_{\cvec{R},F}^\perp\in\Roly{k}(F)^\perp$ give
\begin{align}
\norm[\mu,\CURL,T]{\uvec{\upsilon}_T}^2={}&\mu_T\norm[T]{\PcurlT\uvec{\upsilon}_T}^2 + \mu_T\sum_{F\in\FT} h_F\norm[F]{(\Rproj{k-1}{F}+\ROproj{k}{F})\PcurlT\uvec{\upsilon}_T-(\bvec{\upsilon}_{\cvec{R},F}+\bvec{\upsilon}_{\cvec{R},F}^\perp)}^2\nonumber\\
&+\mu_T\sum_{E\in\ET}h_E^2\norm[E]{\PcurlT\uvec{\upsilon}_T\cdot\tangent_E-\upsilon_E}^2.
\label{eq:expr.norm.curl}
\end{align}

Applying \eqref{eq:PcurlT} with $\bvec{v}=\bvec{0}$ yields $\ROproj{k}{T}(\hPcurlT\uvec{\upsilon}_T)=\bvec{\upsilon}_{\cvec{R},T}^\perp$. Taking the projection $\ROproj{k}{T}$ of \eqref{eq:pcurl}, we infer that $\ROproj{k}{T}(\PcurlT\uvec{\upsilon}_T)=\bvec{\upsilon}_{\cvec{R},T}^\perp$. By \eqref{eq:pcurl}, we also have $\Rproj{k-1}{T}(\PcurlT\uvec{\upsilon}_T)=\bvec{\upsilon}_{\cvec{R},T}$. Hence, $\bvec{\upsilon}_{\cvec{R},T}+\bvec{\upsilon}_{\cvec{R},T}^\perp=(\Rproj{k-1}{T}+\ROproj{k}{T})\PcurlT\uvec{\upsilon}_T$. Using $h_F\approx h_E$ for all $E\in\EF$, we infer
\begin{align*}
\tnorm[\CURL,T]{\uvec{\upsilon}_T}^2\approx{}&\norm[T]{\bvec{\upsilon}_{\cvec{R},T}+\bvec{\upsilon}_{\cvec{R},T}^\perp}^2+\sum_{F\in\FT}h_F\norm[F]{\bvec{\upsilon}_{\cvec{R},F}+\bvec{\upsilon}_{\cvec{R},F}^\perp}^2+\sum_{E\in\ET}h_E^2\norm[E]{\upsilon_E}^2\\
\lesssim{}& \norm[T]{(\Rproj{k-1}{T}+\ROproj{k}{T})\PcurlT\uvec{\upsilon}_T}^2+\sum_{F\in\FT}h_F\norm[F]{\bvec{\upsilon}_{\cvec{R},F}+\bvec{\upsilon}_{\cvec{R},F}^\perp-(\Rproj{k-1}{F}+\ROproj{k}{F})\PcurlT\uvec{\upsilon}_T}^2\\
&+\sum_{E\in\ET}h_E^2\norm[E]{\upsilon_E-\PcurlT\uvec{\upsilon}_T\cdot\tangent_E}^2+\sum_{F\in\FT}h_F\norm[F]{(\Rproj{k-1}{F}+\ROproj{k}{F})\PcurlT\uvec{\upsilon}_T}^2\\
&+\sum_{E\in\ET}h_E^2\norm[E]{\PcurlT\uvec{\upsilon}_T\cdot\tangent_E}^2\\
\lesssim{}&\norm[T]{\PcurlT\uvec{\upsilon}_T}^2+\sum_{F\in\FT}h_F\norm[F]{\bvec{\upsilon}_{\cvec{R},F}+\bvec{\upsilon}_{\cvec{R},F}^\perp-(\Rproj{k-1}{F}+\ROproj{k}{F})\PcurlT\uvec{\upsilon}_T}^2\\
&+\sum_{E\in\ET}h_E^2\norm[E]{\upsilon_E-\PcurlT\uvec{\upsilon}_T\cdot\tangent_E}^2,
\end{align*}
where the first inequality follows introducing $(\Rproj{k-1}{F}+\ROproj{k}{F})\PcurlT\uvec{\upsilon}_T$ in the face terms, $\PcurlT\uvec{\upsilon}_T\cdot\tangent_E$ in the edge terms, and using triangle inequalities, while the second inequality is obtained invoking the boundedness of the $\Leb^2$-projectors $(\Rproj{k-1}{T}+\ROproj{k}{T})$ and $(\Rproj{k-1}{F}+\ROproj{k}{F})$, and discrete trace inequalities on each face $F\in\FT$ and edge $E\in\ET$. Recalling \eqref{eq:expr.norm.curl}, we infer $\tnorm[\CURL,T]{\uvec{\upsilon}_T}^2\lesssim \norm[\mu,\CURL,T]{\uvec{\upsilon}_T}^2$.

To establish the converse estimate, we start from \eqref{eq:expr.norm.curl} and use triangle inequalities to write
\begin{align*}
\norm[\mu,\CURL,T]{\uvec{\upsilon}_T}^2\lesssim{}& \norm[T]{\PcurlT\uvec{\upsilon}_T}^2 + \sum_{F\in\FT} h_F\norm[F]{(\Rproj{k-1}{F}+\ROproj{k}{F})\PcurlT\uvec{\upsilon}_T}^2+\sum_{E\in\ET}h_E^2\norm[E]{\PcurlT\uvec{\upsilon}_T\cdot\tangent_E}^2\\
&+\sum_{F\in\FT} h_F\norm[F]{\bvec{\upsilon}_{\cvec{R},F}+\bvec{\upsilon}_{\cvec{R},F}^\perp}^2+\sum_{E\in\ET}h_E^2\norm[E]{\upsilon_E}^2\\
\lesssim{}& \norm[T]{\PcurlT\uvec{\upsilon}_T}^2+\tnorm[\CURL,T]{\uvec{\upsilon}_T}^2\lesssim \tnorm[\CURL,T]{\uvec{\upsilon}_T}^2,
\end{align*}
where the second inequality follows from the boundedness of $\Leb^2$-projectors $(\Rproj{k-1}{F}+\ROproj{k}{F})$ together with discrete trace inequalities, while the conclusion is obtained invoking \eqref{eq:bd.uCT.PcurlT}. \end{proof}

\subsection{Preliminaries to the Poincar\'e inequalities for $\uCh$}

\begin{lemma}[Poincar\'e inequality for $\uCh$ with averages of edge unknowns]\label{lem:poincare.with.edges}
For any $\uvec{\upsilon}_h\in\Xcurl[h]$, there exists $\uvec{\zeta}_h\in \Xcurl[h]$ such that
\begin{equation}\label{eq:dec.zeta.ker}
\uvec{\upsilon}_h+\uvec{\zeta}_h\in\ker \Ch\quad\mbox{ and }\quad \norm[\mu,\CURL,h]{\uvec{\zeta}_h}\lesssim 
\left(\sum_{T\in\Th}h_T^2\tnorm[\DIV,T]{\uCT\uvec{\upsilon}_h}^2\right)^{\frac12} + \left(\sum_{E\in\Eh}h_E^2|E|(\overline{\upsilon}_E)^2\right)^{\frac12},
\end{equation}
where, for any $E\in\Eh$, $\overline{\upsilon}_E$ is the average value of $\upsilon_E$ on $E$.

As a consequence, if $(\Ker\uCh)^\perp$ is the ortho\-gonal complement of $\Ker\uCh$ in $\Xcurl[h]$ for an inner product whose norm is, uniformly in $h$, equivalent to $\norm[\mu,\CURL,h]{{\cdot}}$, it holds
\begin{equation}\label{eq:est.face.el.Xcurl}
\norm[\mu,\CURL,h]{\uvec{\upsilon}_h}\lesssim \tnorm[\DIV,h]{\uCh\uvec{\upsilon}_h} + \left(
\sum_{E\in\Eh}h_E^2|E|(\overline{\upsilon}_E)^2\right)^{\frac12}\qquad\forall \uvec{\upsilon}_h\in(\Ker\uCh)^\perp.
\end{equation}
\end{lemma}

\begin{proof}
Let $\uvec{\upsilon}_h\in\Xcurl[h]$. The vector $\uvec{\zeta}_h$ in \eqref{eq:dec.zeta.ker} is constructed under the form 
\[
\uvec{\zeta}_h=\big((\bvec{\zeta}_{\cvec{R},T},\bvec{0})_{T\in\Th},(\bvec{\zeta}_{\cvec{R},F},\bvec{0})_{F\in\Fh},(\zeta_E)_{E\in\Eh}\big),
\]
starting with the edge components, then constructing the face components, and finally the element components.
We fix $\zeta_E=-\overline{\upsilon}_E$ for all $E\in\Eh$. This readily gives, denoting by $\mathcal S_h(\uvec{\upsilon}_h)$ the right-hand side of \eqref{eq:dec.zeta.ker},
\begin{equation}\label{eq:est.zetaE}
\sum_{E\in\Eh}h_E^2\norm[E]{\zeta_E}^2\le\mathcal S_h(\uvec{\upsilon}_h)^2.
\end{equation}

Recalling \eqref{eq:CF} and enforcing, for all $F\in\Fh$, $\CF(\uvec{\upsilon}_F+\uvec{\zeta}_F)=0$, we obtain
\begin{equation}\label{eq:def.zetaF}
\int_F (\bvec{\upsilon}_{\cvec{R},F}+\bvec{\zeta}_{\cvec{R},F})\cdot\VROT_Fq-\sum_{E\in\EF}\omega_{FE}\int_E(\upsilon_E+\zeta_E)q=0
\qquad\forall q\in\Poly{k}(F).
\end{equation}
For all $E\in\Eh$, by construction of $\zeta_E$, we have $\int_E(\upsilon_E+\zeta_E)=0$. Hence, the above equation is automatically satisfied if $q$ is constant, and we only have to impose it for $q\in\Poly{0,k}(F)$. Since $\VROT_F:\Poly{0,k}(F)\xrightarrow{\cong}\Roly{k-1}(F)$ is an isomorphism, this defines $\bvec{\zeta}_{\cvec{R},F}\in\Roly{k-1}(F)$ uniquely. Moreover, by definition \eqref{eq:CF} of $\CF\uvec{\upsilon}_F$, the terms involving $\bvec{\upsilon}_{\cvec{R},F}$ and $\upsilon_E$ can be replaced in \eqref{eq:def.zetaF} and we have
\begin{equation}\label{eq:def.zetaF.2}
\int_F \bvec{\zeta}_{\cvec{R},F}\cdot\VROT_Fq=\sum_{E\in\EF}\omega_{FE}\int_E\zeta_Eq - \int_F \CF\uvec{\upsilon}_Fq
\qquad\forall q\in\Poly{k}(F).
\end{equation}
Working as around \eqref{eq:est.r} with $\bvec{\sigma}=\bvec{\zeta}_{\cvec{R},F}$, which satisfies $\Rproj{k}{F}\bvec{\sigma}=
\bvec{\zeta}_{\cvec{R},F}$, we find $q\in\Poly{k}(F)$ such that $\VROT_Fq=\bvec{\zeta}_{\cvec{R},F}$ and $\norm[F]{q}
\lesssim h_F\norm[F]{\bvec{\zeta}_{\cvec{R},F}}$. Plugging this $q$ into \eqref{eq:def.zetaF.2} and using Cauchy--Schwarz and discrete trace inequalities leads to
\[
\norm[F]{\bvec{\zeta}_{\cvec{R},F}}\lesssim \left(\sum_{E\in\EF}h_E^2\norm[E]{\zeta_E}^2\right)^{\frac12}h_F^{-\frac12}
+h_F\norm[F]{\CF\uvec{\upsilon}_F}.
\]
Squaring this inequality, multiplying by $h_F$, summing over $F\in\Fh$, and using $\card\left\{G\in \Fh\st E\in\mathcal{E}_{G}\right\}\lesssim 1$, we infer that
\begin{equation}
\sum_{F\in\Fh}h_F\norm[F]{\bvec{\zeta}_{\cvec{R},F}}^2\lesssim \sum_{F\in\Fh}h_F^3\norm[F]{\CF\uvec{\upsilon}_F}^2+\sum_{E\in\Eh}h_E^2\norm[E]{\zeta_E}^2\lesssim \mathcal S_h(\uvec{\upsilon}_h)^2,
\label{eq:est.zetaF}
\end{equation}
where the conclusion follows from $h_F^2\le h_T^2$, \eqref{eq:est.zetaE} and the definition of $\tnorm[\DIV,T]{\uCh\uvec{\upsilon}_h}$. 

By \cite[Eq. (5.22)]{Di-Pietro.Droniou.ea:20}, and recalling that $\CF(\uvec{\upsilon}_F+\uvec{\zeta}_F)=0$ for all $F\in\Fh$, we have, for all $T\in\Th$,
\[
\int_T \Gproj{k-1}{T}\CT(\uvec{\upsilon}_T+\uvec{\zeta}_T)\cdot\GRAD r=\sum_{F\in\FT}\omega_{TF}\int_F\CF(\uvec{\upsilon}_F+\uvec{\zeta}_F)r=0\qquad\forall r\in\Poly{k}(T).
\]
This shows that, whatever the choice of $(\bvec{\zeta}_{\cvec{R},T})_{T\in\Th}$, we have $\Gproj{k-1}{T}\CT(\uvec{\upsilon}_T+\uvec{\zeta}_T)=0$ for all $T\in\Th$. To ensure that $\uCh(\uvec{\upsilon}_h+\uvec{\zeta}_h)=0$ we therefore only have to show that $\GOproj{k}{T}\CT(\uvec{\upsilon}_T+\uvec{\zeta}_T)=0$ for all $T\in\Th$, which reduces to, recalling the definition \eqref{eq:CT} of $\CT$,
\[
\int_T(\bvec{\upsilon}_{\cvec{R},T}+\bvec{\zeta}_{\cvec{R},T})\cdot \CURL\bvec{\tau}+\sum_{F\in\FT}\omega_{TF}\int_F\trFt(\uvec{\upsilon}_F+\uvec{\zeta}_F)\cdot(\bvec{\tau}\times\normal_F)=0\qquad\forall\bvec{\tau}\in\Goly{k}(T)^\perp.
\]
Since $\CURL:\Goly{k}(T)^\perp\xrightarrow{\cong}\Roly{k-1}(T)$ is an isomorphism, this equation defines a unique $\bvec{\zeta}_{\cvec{R},T}\in\Roly{k-1}(T)$. Moreover, using the definition \eqref{eq:CT} of $\CT\uvec{\upsilon}_T$ to replace the terms involving $\bvec{\upsilon}_{\cvec{R},T}$ and $(\trFt\uvec{\upsilon}_F)_{F\in\FT}$, we have
\[
\int_T\bvec{\zeta}_{\cvec{R},T}\cdot \CURL\bvec{\tau}=-\sum_{F\in\FT}\omega_{TF}\int_F\trFt\uvec{\zeta}_F\cdot(\bvec{\tau}\times\normal_F)-\int_T \GOproj{k}{T}\CT\uvec{\upsilon}_T\cdot\bvec{\tau}\qquad\forall\bvec{\tau}\in\Goly{k}(T)^\perp.
\]
Invoking Lemma \ref{lem:CURL.iso}, we select $\bvec{\tau}\in\Goly{k}(T)^\perp$ such that $\CURL\bvec{\tau}=\bvec{\zeta}_{\cvec{R},T}$ and $\norm[T]{\bvec{\tau}}\lesssim h_T\norm[T]{\bvec{\zeta}_{\cvec{R},T}}$. Using Cauchy--Schwarz and discrete trace inequalities, we obtain
\[
\norm[T]{\bvec{\zeta}_{\cvec{R},T}}\lesssim \left(\sum_{F\in\FT}h_F\norm[F]{\trFt\uvec{\zeta}_F}^2\right)^{\frac12} + h_T\norm[T]{\GOproj{k}{T}\CT\uvec{\upsilon}_T}.
\]
Squaring, summing over $T\in\Th$, using the definition of $\uCh$, the boundedness of $\trFt$ stated in \eqref{eq:bd.CF.trFt}, and recalling \eqref{eq:est.zetaF} and \eqref{eq:est.zetaE}, we infer 
\begin{equation}\label{eq:est.zetaT}
\sum_{T\in\Th}\norm[T]{\bvec{\zeta}_{\cvec{R},T}}^2\lesssim \mathcal S_h(\uvec{\upsilon}_h)^2.
\end{equation}
The proof of \eqref{eq:dec.zeta.ker} is concluded gathering \eqref{eq:est.zetaE}, \eqref{eq:est.zetaF} and \eqref{eq:est.zetaT}, and using the norm equivalence \eqref{eq:equiv.Xcurl}.

\medskip

We now establish \eqref{eq:est.face.el.Xcurl}. Let $N(\cdot)$ be the norm associated with the inner product for which $\uvec{\upsilon}_h\perp\Ker\uCh$, and $P^\perp:\Xcurl[h]\to(\Ker\uCh)^\perp$ be the orthogonal projector for this inner product.
Taking $\uvec{\zeta}_h$ given by \eqref{eq:dec.zeta.ker}, we have $P^\perp(\uvec{\upsilon}_h+\uvec{\zeta}_h)=0$, which gives $\uvec{\upsilon}_h=-P^\perp\uvec{\zeta}_h$ since $\uvec{\upsilon}_h\in(\Ker\uCh)^\perp$. Hence, $N(\uvec{\upsilon}_h)\le N(\uvec{\zeta}_h)$ and thus, by the assumed equivalence of $N(\cdot)$, we infer $\norm[\mu,\CURL,h]{\uvec{\upsilon}_h}\lesssim\tnorm[\CURL,h]{\uvec{\zeta}_h}$. The proof is completed by noticing that the right-hand side of \eqref{eq:dec.zeta.ker} is bounded above by the right-hand side of \eqref{eq:est.face.el.Xcurl}, since $h_T\lesssim 1$. \end{proof}

Lemma \ref{lem:poincare.with.edges} highlights the interest of Assumption \ref{assum:poincare.edges} when establishing a Poincar\'e inequality for $\uCh$. The following proposition gives a situation when this assumption is satisfied.

\begin{proposition}[Assumption \ref{assum:poincare.edges} on simply connected domains]\label{prop:poincare.curl.simply.connected}
If $\Omega$ is simply connected, then Assumption \ref{assum:poincare.edges} holds with $\alpha$ not depending on $h$.
\end{proposition}

\begin{proof}
  We have to prove \eqref{eq:poincare.edges}.
  To this end, we leverage the discrete Poincar\'e inequality proved in \cite[Lemma 2.2]{Bonelle.Ern:15} in the context of Compatible Discrete Operators, which are linked to the DDR sequence for $k=0$.
    We next recall this result.
  Let
\[
\Xcurlz[h] \coloneq \left\{(\nu_E)_{E\in\Eh}\,:\,\nu_E\in\Real\quad\forall E\in\Eh\right\}
\quad\text{ and }\quad
\Xdivz[h]\coloneq\left\{(r_F)_{F\in\Fh}\,:\,r_F\in\Real\quad\forall F\in\Fh\right\}
\]
be the spaces of edge and face constant values, respectively.
These spaces correspond to $\underline{\bvec{X}}_{\CURL,h}^0$ and $\underline{\bvec{X}}_{\DIV,h}^0$, respectively, but are expressed in a notation more similar to \cite{Bonelle.Ern:15} to help the reader navigate this reference.
We define the projectors $\Pi^0_{\CURL,h}:\Xcurl[h]\to\Xcurlz[h]$ and $\Pi^0_{\DIV,h}:\Xdiv[h]\to\Xdivz[h]$ such that
\[
\Pi^0_{\CURL,h}\uvec{\upsilon}_h\coloneq(\overline{\upsilon}_E)_{E\in\Eh}\quad\forall \uvec{\upsilon}_h\in\Xcurl[h]
\quad\mbox{ and }\quad \Pi^0_{\DIV,h}\uvec{w}_h\coloneq(\overline{w}_F)_{F\in\Fh}\quad\forall\uvec{w}_h\in\Xdiv[h],
\]
where $\overline{w}_F$ is the average of $w_F$ on $F$. Let $\CURLz:\Xcurlz[h]\to\Xdivz[h]$ be the discrete curl defined by
\begin{equation}\label{eq:CURLz}
  \CURLz\underline{\nu}_h\coloneq\Big(-\frac{1}{|F|}\sum_{E\in\EF}\omega_{FE}|E|\nu_E\Big)_{F\in\Fh}\qquad\forall \underline{\nu}_h=(\nu_E)_{E\in\Eh}\in\Xcurlz[h],
\end{equation}
with $|F|$ denoting the area of $F\in\Fh$ and $|E|$ the length of $E$.
Making $q=1$ in \eqref{eq:CF} shows that the following commutation property holds:
\begin{equation}\label{eq:commut.proj}
  \Pi^0_{\DIV,h}\big(\uCh\uvec{\upsilon}_h\big)
  = \CURLz\big(\Pi^0_{\CURL,h}\uvec{\upsilon}_h\big)
  \qquad\forall\uvec{\upsilon}_h\in\Xcurl[h].
\end{equation}
It is inferred from \cite[Lemma 2.2]{Bonelle.Ern:15} that
\begin{equation}\label{eq:poincare.zero}
\tnorm[\Eh]{\underline{\nu}_h}^2:=
\sum_{E\in\Eh}h_E^2|E|\nu_E^2\lesssim \sum_{F\in\Fh}h_F|F|(\CURLz\underline{\nu}_h)_F^2\qquad\forall\underline{\nu}_h\in(\Ker\CURLz)^{\perp_{\Eh}},
\end{equation}
where the orthogonal complement $(\Ker\CURLz)^{\perp_{\Eh}}$ is taken for a certain inner product $\langle\cdot,\cdot\rangle_{\Eh}$ on $\Xcurlz[h]$ whose norm is equivalent (uniformly in $h$) to $\tnorm[\Eh]{{\cdot}}$ defined above.
Notice that the norms appearing in \eqref{eq:poincare.zero} use a slightly different local length scale with respect to the ones defined in \cite[Eq.\ (2.15)]{Bonelle.Ern:15}. By mesh regularity, these choices are equivalent up to a constant that depends only on the mesh regularity parameter.
Our approach to leverage this lowest-order Poincar\'e inequality in order to prove \eqref{eq:poincare.edges} consists in defining on $\Xcurl[h]$ an inner product, whose norm is equivalent to $\tnorm[\CURL,h]{{\cdot}}$, such that $\Pi^0_{\CURL,h}(\Ker\uCh)^\perp\subset (\Ker\CURLz)^{\perp_{\Eh}}$.
A preliminary step consists in constructing an extension operator from $\Xcurlz$ to $\Xcurl[h]$ that sends $\Ker\CURLz$ into $\Ker\uCh$.
\medskip\\
{\bf 1. Extension operator.}
The extension operator $\ext:\Xcurlz[h]\to\Xcurl[h]$ is constructed following similar ideas as in the proof of Lemma \ref{lem:poincare.with.edges}:
For all $\underline{\nu}_h\in\Xcurlz[h]$, set $\ext\underline{\nu}_h\coloneq\uvec{\zeta}_h=((\bvec{\zeta}_{\cvec{R},T},\bvec{0})_{T\in\Th},(\bvec{\zeta}_{\cvec{R},F},\bvec{0})_{F\in\Fh},(\nu_E)_{E\in\Eh})$ such that
\begin{alignat}{3}
  \label{eq:def.Emu.F}
  \forall F\in\Fh&\qquad&
  \int_F \bvec{\zeta}_{\cvec{R},F}\cdot\VROT_F q - \sum_{E\in\FE}\omega_{FE}\int_E\nu_E q &= 0
  &\qquad&\forall q\in\Poly{0,k}(F),
  \\
  \label{eq:def.Emu.T}
  \forall T\in\Th&\qquad&
  \int_T\bvec{\zeta}_{\cvec{R},T}\cdot\CURL \bvec{\tau}+\sum_{F\in\FT}\omega_{TF}\int_F \trFt\uvec{\zeta}_F\cdot(\bvec{\tau}\times\normal_F) &= 0
  &\qquad&\forall \bvec{\tau}\in\Goly{k}(T)^\perp.
\end{alignat}
Reasoning as in the proof of Lemma \ref{lem:poincare.with.edges}, these equations uniquely define $(\bvec{\zeta}_{\cvec{R},F})_{F\in\Fh}$ and $(\bvec{\zeta}_{\cvec{R},T})_{T\in\Th}$, which satisfy
\begin{equation}\label{eq:E.iso:intermediate}
\sum_{T\in\Th}\norm[T]{\bvec{\zeta}_{\cvec{R},T}}^2+\sum_{F\in\Fh}h_F\norm[F]{\bvec{\zeta}_{\cvec{R},F}}^2\lesssim
\tnorm[\Eh]{\underline{\nu}_h}^2.
\end{equation}
We obviously have 
\begin{equation}\label{eq:E.Pi0}
  \Pi^0_{\CURL,h}\big(\ext\underline{\nu}_h\big)=\underline{\nu}_h
  \qquad\forall \underline{\nu}_h\in\Xcurlz[h]
\end{equation}
since the edge components of $\ext\underline{\nu}_h=\uvec{\zeta}_h$ are $(\nu_E)_{E\in\Eh}$.
Combining this remark with \eqref{eq:E.iso:intermediate}, we infer 
\begin{equation}\label{eq:E.iso}
\tnorm[\CURL,h]{\ext\underline{\nu}_h}\approx \tnorm[\Eh]{\underline{\nu}_h}\qquad\forall \underline{\nu}_h\in \Xcurlz[h].
\end{equation}

Let us prove that
\begin{equation}\label{eq:E.ker}
\ext\underline{\nu}_h\in \Ker\uCh\qquad\forall \underline{\nu}_h\in\Ker\CURLz.
\end{equation}
Let $\underline{\nu}_h\in\Ker\CURLz$.
By definition \eqref{eq:CURLz} of $\CURLz$, it holds $\sum_{E\in\EF}\omega_{FE}|E|\nu_E=0$ for all $F\in \Fh$, and thus \eqref{eq:def.Emu.F} is actually satisfied for all $q\in\Poly{k}(F)$. This shows that $\CF\big(\ext\underline{\nu}_h\big)=0$ for all $F\in\Fh$. Using again the argument, based on \cite[Eq. (5.22)]{Di-Pietro.Droniou.ea:20}, in the proof of Lemma \ref{lem:poincare.with.edges}, we infer from \eqref{eq:def.Emu.T} that $\uCT\big(\ext\underline{\nu}_h\big)=\uvec{0}$ for all $T\in\Th$.
This proves \eqref{eq:E.ker}.
\medskip\\
{\bf 2. Construction of the inner product on  $\Xcurl[h]$.}
Let us now define the inner product $[\cdot,\cdot]$ on $\Xcurl[h]$ such that, for all $\forall (\uvec{\upsilon}_h,\uvec{\sigma}_h)\in\Xcurl[h]\times\Xcurl[h]$,
\[
[\uvec{\upsilon}_h,\uvec{\sigma}_h]=\langle \Pi^0_{\CURL,h}\uvec{\upsilon}_h,\Pi^0_{\CURL,h}\uvec{\sigma}_h\rangle_{\Eh}
+((\Id-\ext\Pi^0_{\CURL,h})\uvec{\upsilon}_h,(\Id-\ext\Pi^0_{\CURL,h})\uvec{\sigma}_h)_{\CURL,h}.
\]
Letting $N(\cdot)$ be the norm associated with $[\cdot,\cdot]$ and recalling that the norm associated with $\langle\cdot,\cdot\rangle_{\Eh}$ is equivalent to $\tnorm[\Eh]{{\cdot}}$, we have, for all $\uvec{\upsilon}_h\in\Xcurl[h]$,
\[
\begin{aligned}
N(\uvec{\upsilon}_h)^2={}&\tnorm[\Eh]{\Pi^0_{\CURL,h}\uvec{\upsilon}_h}^2+\norm[\mu,\CURL,h]{(\Id-\ext\Pi^0_{\CURL,h})\uvec{\upsilon}_h}^2\\
\approx{}& \norm[\mu,\CURL,h]{\ext\Pi^0_{\CURL,h}\uvec{\upsilon}_h}^2+\norm[\mu,\CURL,h]{(\Id-\ext\Pi^0_{\CURL,h})\uvec{\upsilon}_h}^2,
\end{aligned}
\]
where the equivalence follows from \eqref{eq:E.iso} and \eqref{eq:equiv.Xcurl}.
Triangle inequalities combined with \eqref{eq:E.iso} and the straightforward estimate $\tnorm[\Eh]{\Pi^0_{\CURL}\uvec{\upsilon}_h}\lesssim \tnorm[\CURL,h]{\uvec{\upsilon}_h}\approx \norm[\mu,\CURL,h]{\uvec{\upsilon}_h}$ then show that $N(\cdot)$ is equivalent to $\norm[\mu,\CURL,h]{{\cdot}}$, as required.
\medskip\\
{\bf 3. Conclusion.}
Take $\uvec{\upsilon}_h\in(\Ker\uCh)^\perp$, where the orthogonal is taken for the inner product $[\cdot,\cdot]$. Let $\underline{\nu}_h\in\Ker\CURLz$.
Then, $\ext\underline{\nu}_h\in\Ker\uCh$ (see \eqref{eq:E.ker}) and $(\Id-\ext\Pi^0_{\CURL,h})\ext\underline{\nu}_h=\ext\underline{\nu}_h-\ext\Pi^0_{\CURL,h}(\ext\underline{\nu}_h)=0$ (see \eqref{eq:E.Pi0}).
Hence,
\[
0=[\uvec{\upsilon}_h,\ext\underline{\nu}_h]=\langle \Pi^0_{\CURL,h}\uvec{\upsilon}_h,\Pi^0_{\CURL,h}\ext\underline{\nu}_h\rangle_{\Eh}=\langle \Pi^0_{\CURL,h}\uvec{\upsilon}_h,\underline{\nu}_h\rangle_{\Eh}.
\]
In other words, $\Pi^0_{\CURL,h}\uvec{\upsilon}_h\in(\Ker\CURLz)^{\perp_{\Eh}}$. Invoking then \eqref{eq:poincare.zero} on $\underline{\nu}_h=\Pi^0_{\CURL,h}\uvec{\upsilon}_h$, recalling the definition of $\Pi^0_{\CURL,h}$, and using the commutation property \eqref{eq:commut.proj}, we infer
\[
\sum_{E\in\Eh}h_E^2|E|(\overline{\upsilon}_E)^2\lesssim \sum_{F\in\Fh}h_F|F|\left(\overline{\CF\uvec{\upsilon}_F}\right)_F^2
\le\sum_{F\in\Fh}h_F\norm[F]{\CF\uvec{\upsilon}_F}^2,
\]
the second estimate following by Jensen's inequality since $\left(\overline{\CF\uvec{\upsilon}_F}\right)_F$ is the average of $\CF\uvec{\upsilon}_F$ on $F$. The proof of \eqref{eq:poincare.edges} is completed recalling the definition \eqref{eq:def.tnorm.div} of $\tnorm[\DIV,h]{{\cdot}}$ and the norm equivalence \eqref{eq:equiv.Xdiv}. \end{proof}

\subsection{Boundedness of interpolator}

The global interpolator on $\Xdiv[h]$ is given by: For all $\bvec{v}\in \Hil^1(\Omega)^3$,
\[
\uIdiv[h]\bvec{v}\coloneq\big(
(\Gproj{k-1}{T}\bvec{v},\GOproj{k}{T}\bvec{v})_{T\in\Th},(\lproj{k}{F}(\bvec{v}\cdot\normal_F)
\big)_{F\in\Fh})\in\Xdiv[h].
\]

\begin{lemma}[Boundedness of {$\uIdiv[h]$}]\label{lem:bd.interp.div}
It holds
\begin{equation}\label{eq:bd.uIdiv}
\tnorm[\DIV,h]{\uIdiv[h]\bvec{v}}\lesssim \norm[\Hil^1(\Omega)^3]{\bvec{v}}\qquad\forall\bvec{v}\in \Hil^1(\Omega)^3.
\end{equation}
\end{lemma}

\begin{proof}
By $\Leb^2(T)^3$-boundedness of the orthogonal projector $\Gproj{k-1}{T}+\GOproj{k}{T}$, we have 
\begin{equation}\label{eq:bd.uIdiv.1}
\norm[T]{\Gproj{k-1}{T}\bvec{v}+\GOproj{k}{T}\bvec{v}}\le \norm[T]{\bvec{v}}.
\end{equation}
The $\Leb^2(F)$-boundedness of $\lproj{k}{F}$ and the local continuous trace inequality \cite[Lemma 1.31]{Di-Pietro.Droniou:20} yield
\begin{equation}\label{eq:bd.uIdiv.2}
 h_F^{\frac12}\norm[F]{\lproj{k}{F}(\bvec{v}\cdot\normal_F)}\le h_F^{\frac12}\norm[F]{\bvec{v}}\le \norm[T]{\bvec{v}}+h_T\norm[T]{\GRAD\bvec{v}}\lesssim \norm[\Hil^1(T)^3]{\bvec{v}}.
\end{equation}
Estimate \eqref{eq:bd.uIdiv} follows raising \eqref{eq:bd.uIdiv.1} and \eqref{eq:bd.uIdiv.2} to the square and summing, respectively, over $T\in\Th$ and $F\in \Fh$.
\end{proof}

\subsection{Poincar\'e inequalities}

\begin{theorem}[Isomorphism property and Poincar\'e inequality for {$\Dh$}]\label{thm:poincare.Dh}
Let $(\Ker\Dh)^\perp$ be the orthogonal of $\Ker\Dh$ in $\Xdiv[h]$ for an inner product whose norm is (uniformly in $h$) equivalent to $\norm[\DIV,h]{{\cdot}}$. Then, $\Dh:(\Ker\Dh)^\perp\to\Poly{k}(\Th)$ is an isomorphism and
\begin{equation}\label{eq:poincare.Dh}
\norm[\DIV,h]{\uvec{w}_h}\lesssim \norm[\Omega]{\Dh\uvec{w}_h}\qquad\forall \uvec{w}_h\in(\Ker\Dh)^\perp.
\end{equation}
\end{theorem}

\begin{remark}[Topology of {$\Omega$}]
As for the exactness of $\Dh$ stated in Theorem \ref{thm:exactness}, Theorem \ref{thm:poincare.Dh} is valid for any polyhedral domain (even if its second Betti number is not zero).
\end{remark}

\begin{proof}
The fact that $\Dh:(\Ker\Dh)^\perp\to\Poly{k}(\Th)$ is an isomorphism follows directly from the exactness relation \eqref{eq:Im.div=Poly.Th} and the decomposition $\Xdiv[h]=\Ker\Dh\oplus (\Ker\Dh)^\perp$. 

Let $\uvec{w}_h\in(\Ker\Dh)^\perp$ and set $q_h\coloneq\Dh\uvec{w}_h$. As seen in the proof of \eqref{eq:Im.div=Poly.Th}, if $\bvec{v}\in \Hil^1(\Omega)^3$ is such that $\DIV\bvec{v}=q_h$, then the global interpolate $\uvec{v}_h\coloneq\uIdiv[h]\bvec{v}\in \Xdiv[h]$ of $\bvec{v}$ satisfies $\Dh\uvec{v}_h=q_h$. We can take $\bvec{v}$ such that $\norm[\Hil^1(\Omega)^3]{\bvec{v}}\lesssim\norm[\Omega]{q_h}$ (see \cite[Lemma 8.3]{Di-Pietro.Droniou:20}), and Lemma \ref{lem:bd.interp.div} then shows that 
\begin{equation}\label{eq:est.wh}
\tnorm[\DIV,h]{\uvec{v}_h}\lesssim \norm[\Hil^1(\Omega)^3]{\bvec{v}}\lesssim\norm[\Omega]{q_h}.
\end{equation}
We have $\Dh(\uvec{v}_h-\uvec{w}_h)=q_h-q_h=0$ so $\uvec{v}_h-\uvec{w}_h\in\Ker\Dh$ and $\uvec{w}_h$ is the orthogonal projection, for the inner product in the theorem, of $\uvec{v}_h$ on $(\Ker\Dh)^\perp$. The norm, for this inner product, of $\uvec{w}_h$ is therefore less than the norm of $\uvec{v}_h$. The norm equivalence stated in the theorem, \eqref{eq:equiv.Xdiv} (with $\clubsuit=h$) 
and \eqref{eq:est.wh} conclude the proof of \eqref{eq:poincare.Dh}. \end{proof}

\begin{theorem}[Isomorphism property and Poincar\'e inequality for {$\uCh$}]\label{thm:poincare.Ch}
Let $(\Ker\uCh)^\perp$ be the orthogonal of $\Ker\uCh$ in $\Xcurl[h]$ for an inner product whose norm is (uniformly in $h$) equivalent to $\norm[\mu,\CURL,h]{{\cdot}}$. Then, $\uCh:(\Ker\uCh)^\perp\to\Ker\Dh$ is an isomorphism. Moreover, under Assumption \ref{assum:poincare.edges} and choosing the inner product above as the one provided by this assumption, it holds
\begin{equation}\label{eq:poincare.Ch}
\norm[\mu,\CURL,h]{\uvec{\upsilon}_h}\lesssim \norm[\DIV,h]{\uCh\uvec{\upsilon}_h}\qquad\forall \uvec{\upsilon}_h\in(\Ker\uCh)^\perp.
\end{equation}
\end{theorem}

\begin{proof}
The isomorphism property follows from the exactness property \eqref{eq:Im.curl=Ker.div} and the orthogonal decomposition $\Xcurl[h]=\Ker\uCh \oplus (\Ker\uCh)^\perp$. The Poincar\'e inequality \eqref{eq:poincare.Ch} is a consequence of Lemma \ref{lem:poincare.with.edges}, Assumption \ref{assum:poincare.edges}, and the equivalence of norms \eqref{eq:equiv.Xdiv}.
\end{proof}

\subsection{Proof of Theorem \ref{thm:stability}}\label{sec:proof.stability}

We follow the arguments in the proof of \cite[Lemma 1]{Di-Pietro.Droniou.ea:20}. Let $\mathcal S$ denote the left-hand side of \eqref{eq:inf.sup.Ah} and let us start by choosing $(\uvec{\zeta}_h,\uvec{v}_h)=(\uvec{\upsilon}_h,\uvec{w}_h)$ in \eqref{eq:def.Ah}.
This readily gives
\begin{equation}\label{eq:stab.cont.1}
\mathcal S \ge \frac{\norm[\mu,\CURL,h]{\uvec{\upsilon}_h}^2+\norm[\Omega]{\Dh \uvec{w}_h}^2}{\norm[\mu,\CURL,1,h]{\uvec{\upsilon}_h}+\norm[\DIV,1,h]{\uvec{w}_h}}.
\end{equation}
We then select $(\uvec{\zeta}_h,\uvec{v}_h)=(\uvec{0},\uCh\uvec{\upsilon}_h)$ in \eqref{eq:def.Ah} and use $\Dh\uvec{v}_h=0$ (by the exactness \eqref{eq:Im.curl=Ker.div}) to obtain
$\mathcal S \ge \norm[\DIV,h]{\uCh\uvec{\upsilon}_h}$ which, combined with \eqref{eq:stab.cont.1}, yields
\begin{equation}\label{eq:stab.cont.2}
\mathcal S\gtrsim \frac{\norm[\mu,\CURL,1,h]{\uvec{\upsilon}_h}^2+\norm[\Omega]{\Dh\uvec{w}_h}^2}{\norm[\mu,\CURL,1,h]{\uvec{\upsilon}_h}+\norm[\DIV,1,h]{\uvec{w}_h}}.
\end{equation}
To conclude the proof of \eqref{eq:inf.sup.Ah}, it remains to estimate $\norm[\DIV,h]{\uvec{w}_h}$.
We split $\uvec{w}_h$ into
\[
\uvec{w}_h=\uvec{w}_h^\star+\uvec{w}_h^\perp\in \Ker\Dh \oplus (\Ker\Dh)^\perp=\Xdiv[h],
\]
where the orthogonal is taken with respect to the $(\cdot,\cdot)_{\DIV,h}$-inner product. Invoking Theorem \ref{thm:poincare.Dh}, we have
\begin{equation}\label{eq:stab.cont.3}
\norm[\DIV,h]{\uvec{w}_h^\perp}^2\lesssim \norm[\Omega]{\Dh\uvec{w}_h^\perp}^2=\norm[\Omega]{\Dh\uvec{w}_h}^2\lesssim
\mathcal S \left(\norm[\mu,\CURL,1,h]{\uvec{\upsilon}_h}+\norm[\DIV,1,h]{\uvec{w}_h}\right),
\end{equation}
where the equality follows from $\Dh\uvec{w}_h=\Dh(\uvec{w}_h^\star+\uvec{w}_h^\perp)=\Dh\uvec{w}_h^\perp$ and \eqref{eq:stab.cont.2} was used in the last inequality. To estimate $\norm[\DIV,h]{\uvec{w}_h^\star}$, we use Theorem \ref{thm:poincare.Ch} to find $\uvec{\zeta}_h\in(\Ker\uCh)^\perp$ (the orthogonal being taken for the inner product in Assumption \ref{assum:poincare.edges}) such that $\uCh\uvec{\zeta}_h=-\uvec{w}_h^\star$ and $\norm[\mu,\CURL,h]{\uvec{\zeta}_h}\lesssim \norm[\DIV,h]{\uvec{w}_h^\star}$. This immediately yields $\norm[\mu,\CURL,1,h]{\uvec{\zeta}_h}\lesssim\norm[\DIV,h]{\uvec{w}_h^\star}$ and, using this $\uvec{\zeta}_h$ together with $\uvec{v}_h=\uvec{0}$ in the definition \eqref{eq:def.Ah} of $\mathcal A_h$, we obtain
\[
\norm[\DIV,h]{\uvec{w}_h^\star}\mathcal S \gtrsim \norm[\mu,\CURL,1,h]{\uvec{\zeta}_h} \mathcal S
\ge (\uvec{\upsilon}_h,\uvec{\zeta}_h)_{\mu,\CURL,h}+(\uvec{w}_h^\star,\uvec{w}_h^\star+\uvec{w}_h^\perp)_{\DIV,h}.
\]
Cauchy--Schwarz inequalities and $\norm[\mu,\CURL,h]{\uvec{\zeta}_h}\lesssim \norm[\DIV,h]{\uvec{w}_h^\star}$, \eqref{eq:stab.cont.2}, and \eqref{eq:stab.cont.3} lead to
\begin{align*}
\norm[\DIV,h]{\uvec{w}_h^\star}^2\lesssim{}& \norm[\DIV,h]{\uvec{w}_h^\star}\norm[\DIV,h]{\uvec{w}_h^\perp}
+\norm[\mu,\CURL,h]{\uvec{\upsilon}_h}\norm[\mu,\CURL,h]{\uvec{\zeta}_h}+\norm[\DIV,h]{\uvec{w}_h^\star}\mathcal S\\
\lesssim{}& \norm[\DIV,h]{\uvec{w}_h^\star}\left[
  \left[
  \mathcal S (\norm[\mu,\CURL,1,h]{\uvec{\upsilon}_h}+\norm[\DIV,1,h]{\uvec{w}_h})
  \right]^{\frac12} + \mathcal S
  \right].
\end{align*}
Simplifying and using again \eqref{eq:stab.cont.3} we infer $\norm[\DIV,h]{\uvec{w}_h}^2\lesssim \mathcal S \left(\norm[\mu,\CURL,1,h]{\uvec{\upsilon}_h}+\norm[\DIV,1,h]{\uvec{w}_h}\right)+\mathcal S^2$ which, combined with \eqref{eq:stab.cont.2}, leads to
$\norm[\mu,\CURL,1,h]{\uvec{\upsilon}_h}^2+\norm[\DIV,1,h]{\uvec{w}_h}^2\lesssim \mathcal S \left(\norm[\mu,\CURL,1,h]{\uvec{\upsilon}_h}+\norm[\DIV,1,h]{\uvec{w}_h}\right)+\mathcal S^2$.
A Young inequality then concludes the proof of the inf--sup estimate \eqref{eq:inf.sup.Ah}.


\section*{Acknowledgements}

The authors thank Ian Wanless for fruitful discussions around the proof of Theorem \ref{thm:exactness}.
D. A. Di Pietro acknowledges the partial support of \emph{Agence Nationale de la Recherche} (grant number ANR-17-CE23-0019). J. Droniou was partially supported by the Australian Government through the \emph{Australian Research Council}'s Discovery Projects funding scheme (grant number DP170100605).


\bibliographystyle{plain}
\bibliography{maddr}

\begin{thebibliography}{10}

\bibitem{Arnold:18}
D.~Arnold.
\newblock {\em Finite Element Exterior Calculus}.
\newblock SIAM, 2018.

\bibitem{Beirao-da-Veiga.Brezzi.ea:18*2}
L.~Beir\~{a}o~da Veiga, F.~Brezzi, F.~Dassi, L.~D. Marini, and A.~Russo.
\newblock A family of three-dimensional virtual elements with applications to
  magnetostatics.
\newblock {\em SIAM J. Numer. Anal.}, 56(5):2940--2962, 2018.

\bibitem{Beirao-da-Veiga.Brezzi.ea:18*1}
L.~Beir\~{a}o~da Veiga, F.~Brezzi, F.~Dassi, L.~D. Marini, and A.~Russo.
\newblock Lowest order virtual element approximation of magnetostatic problems.
\newblock {\em Comput. Methods Appl. Mech. Engrg.}, 332:343--362, 2018.

\bibitem{Beirao-da-Veiga.Brezzi.ea:16}
L.~Beir\~{a}o~da Veiga, F.~Brezzi, L.~D. Marini, and A.~Russo.
\newblock {$H(\mathrm{div})$} and {$H(\mathrm{curl})$}-conforming {VEM}.
\newblock {\em Numer. Math.}, 133:303--332, 2016.

\bibitem{Beirao-da-Veiga.Lipnikov.ea:14}
L.~Beir\~ao~da Veiga, K.~Lipnikov, and G.~Manzini.
\newblock {\em The mimetic finite difference method for elliptic problems},
  volume~11 of {\em MS\&A. Modeling, Simulation and Applications}.
\newblock Springer, Cham, 2014.

\bibitem{Bonazzoli.Rapetti:17}
M.~Bonazzoli and F.~Rapetti.
\newblock High-order finite elements in numerical electromagnetism: degrees of
  freedom and generators in duality.
\newblock {\em Numer. Algorithms}, 74(1):111--136, 2017.

\bibitem{Bonelle:14}
J.~Bonelle.
\newblock {\em Compatible {D}iscrete {O}perator schemes on polyhedral meshes
  for elliptic and {S}tokes equations}.
\newblock PhD thesis, University of Paris-Est, 2014.

\bibitem{Bonelle.Di-Pietro.ea:15}
J.~Bonelle, D.~A. Di~Pietro, and A.~Ern.
\newblock Low-order reconstruction operators on polyhedral meshes: Application
  to {Compatible Discrete Operator} schemes.
\newblock {\em Computer Aided Geometric Design}, 35--36:27--41, 2015.

\bibitem{Bonelle.Ern:15}
J.~Bonelle and A.~Ern.
\newblock Analysis of compatible discrete operator schemes for the {S}tokes
  equations on polyhedral meshes.
\newblock {\em IMA J. Numer. Anal.}, 2015.

\bibitem{Buffa.Rivas.ea:11}
A.~Buffa, J.~Rivas, G.~Sangalli, and R.~V\'{a}zquez.
\newblock Isogeometric discrete differential forms in three dimensions.
\newblock {\em SIAM J. Numer. Anal.}, 49(2):818--844, 2011.

\bibitem{Buffa.Sangalli.ea:10}
A.~Buffa, G.~Sangalli, and R.~V\'{a}zquez.
\newblock Isogeometric analysis in electromagnetics: {B}-splines approximation.
\newblock {\em Comput. Methods Appl. Mech. Engrg.}, 199(17-20):1143--1152,
  2010.

\bibitem{Buffa.Sangalli.ea:14}
A.~Buffa, G.~Sangalli, and R.~V\'{a}zquez.
\newblock Isogeometric methods for computational electromagnetics: {B}-spline
  and {T}-spline discretizations.
\newblock {\em J. Comput. Phys.}, 257(part B):1291--1320, 2014.

\bibitem{Chave.Di-Pietro.ea:20*1}
F.~Chave, D.~A. Di~Pietro, and S.~Lemaire.
\newblock A discrete {Weber} inequality on three-dimensional hybrid spaces with
  application to the {HHO} approximation of magnetostatics.
\newblock submitted, 7 2020.

\bibitem{Chave.Di-Pietro.ea:20}
F.~Chave, D.~A. Di~Pietro, and S.~Lemaire.
\newblock A three-dimensional {Hybrid High-Order} method for magnetostatics.
\newblock In R.~Kl\"{o}fkorn, E.~Keilegavlen, F.~A. Radu, and J.~Fuhrmann,
  editors, {\em Finite Volumes for Complex Applications IX -- Methods,
  Theoretical Aspects, Examples}, pages 255--263, 2020.

\bibitem{Chen.Cui.ea:19}
G.~Chen, J.~Cui, and L.~Xu.
\newblock Analysis of a hybridizable discontinuous {G}alerkin method for the
  {M}axwell operator.
\newblock {\em ESAIM: Math. Model. Numer. Anal.}, 53(1):301--324, 2019.

\bibitem{Chen.Qiu.ea:17}
H.~Chen, W.~Qiu, K.~Shi, and M.~Solano.
\newblock A superconvergent {HDG} method for the {M}axwell equations.
\newblock {\em J. Sci. Comput.}, 70(3):1010--1029, 2017.

\bibitem{Codecasa.Specogna.ea:09}
L.~Codecasa, R.~Specogna, and F.~Trevisan.
\newblock Base functions and discrete constitutive relations for staggered
  polyhedral grids.
\newblock {\em Comput. Methods Appl. Mech. Engrg.}, 198(9-12):1117--1123, 2009.

\bibitem{Di-Pietro.Droniou:17}
D.~A. Di~Pietro and J.~Droniou.
\newblock A {Hybrid High-Order} method for {Leray--Lions} elliptic equations on
  general meshes.
\newblock {\em Math. Comp.}, 86(307):2159--2191, 2017.

\bibitem{Di-Pietro.Droniou:17*1}
D.~A. Di~Pietro and J.~Droniou.
\newblock {$W^{s,p}$}-approximation properties of elliptic projectors on
  polynomial spaces, with application to the error analysis of a {Hybrid
  High-Order} discretisation of {Leray--Lions} problems.
\newblock {\em Math. Models Methods Appl. Sci.}, 27(5):879--908, 2017.

\bibitem{Di-Pietro.Droniou:18}
D.~A. Di~Pietro and J.~Droniou.
\newblock A third {Strang} lemma for schemes in fully discrete formulation.
\newblock {\em Calcolo}, 55(40), 2018.

\bibitem{Di-Pietro.Droniou:20}
D.~A. Di~Pietro and J.~Droniou.
\newblock {\em The {Hybrid High-Order} method for polytopal meshes}.
\newblock Number~19 in Modeling, Simulation and Application. Springer
  International Publishing, 2020.

\bibitem{Di-Pietro.Droniou.ea:20}
D.~A. Di~Pietro, J.~Droniou, and F.~Rapetti.
\newblock Fully discrete polynomial {de Rham} sequences of arbitrary degree on
  polygons and polyhedra.
\newblock {\em Math. Models Methods Appl. Sci.}, 30(9):1809--1855, 2020.

\bibitem{Dotko.Specogna:13}
P.~D\l{}otko and R.~Specogna.
\newblock Cohomology in 3d magneto-quasistatics modeling.
\newblock {\em Communications in Computational Physics}, 14(1):48--76, 2013.

\bibitem{Kanayama.Motoyama.ea:90}
H.~Kanayama, H.~Motoyama, K.~Endo, and F.~Kikuchi.
\newblock Three-dimensional magnetostatic analysis using {N}\'ed\'elec's
  elements.
\newblock {\em IEEE Trans. Magn.}, 26(2):682--685, 1990.

\bibitem{Nguyen.Peraire.ea:11}
N.~C. Nguyen, J.~Peraire, and B.~Cockburn.
\newblock Hybridizable discontinuous {G}alerkin methods for the time-harmonic
  {M}axwell's equations.
\newblock {\em J. Comput. Phys.}, 230(19):7151--7175, 2011.

\bibitem{Perugia.Schotzau.ea:02}
I.~Perugia, D.~Sch\"otzau, and P.~Monk.
\newblock Stabilized interior penalty methods for the time-harmonic {M}axwell
  equations.
\newblock {\em Comput. Methods Appl. Mech. Engrg.}, 191(41--42):4675--4697,
  2002.

\bibitem{Rodriguez.Bertolazzi.ea:13}
Ana~Alonso Rodriguez, Enrico Bertolazzi, Riccardo Ghiloni, and Alberto Valli.
\newblock Construction of a finite element basis of the first de {R}ham
  cohomology group and numerical solution of 3{D} magnetostatic problems.
\newblock {\em SIAM J. Numer. Anal.}, 51(4):2380--2402, 2013.

\bibitem{Spanier:94}
Edwin~H. Spanier.
\newblock {\em Algebraic topology}.
\newblock Springer-Verlag, New York, 1994.
\newblock Corrected reprint of the 1966 original.

\end{thebibliography}
\end{document}